\documentclass[10pt,a4paper,reqno]{amsart}

\usepackage{hyperref}
\hypersetup{
	linktocpage=false,
	pdfauthor={Steven Charlton},
	pdftitle={A review of Dan's reduction method for multiple polylogarithms}	
}
\ifpdf\else
	\usepackage[hyphenbreaks]{breakurl}
\fi

\iftrue
\makeatletter
\def\@settitle{%
  \vspace*{-2em}
  \begin{flushleft}%
    \LARGE\bfseries
    \strut\@title\strut
  \end{flushleft}%
}
\def\@setauthors{%
  \begingroup
  \def\thanks{\protect\thanks@warning}%
  \trivlist
  \raggedright
  \large \@topsep27\p@\relax
  \advance\@topsep by -\baselineskip
  \item\relax
  \author@andify\authors
  \def\\{\protect\linebreak}%
  \authors
  \ifx\@empty\contribs
  \else
    ,\penalty-3 \space \@setcontribs
    \@closetoccontribs
  \fi
  \normalfont
  \endtrivlist
  \endgroup
}
\def\@setaddresses{\par
  \nobreak \begingroup
  \small\raggedright
  \def\author##1{\nobreak\addvspace\smallskipamount}%
  \def\\{\unskip, \ignorespaces}%
  \interlinepenalty\@M
  \def\address##1##2{\begingroup
    \par\addvspace\bigskipamount\noindent
    \@ifnotempty{##1}{(\ignorespaces##1\unskip) }%
    {\ignorespaces##2}\par\endgroup}%
  \def\curraddr##1##2{\begingroup
    \@ifnotempty{##2}{\nobreak\noindent\curraddrname
      \@ifnotempty{##1}{, \ignorespaces##1\unskip}\/:\space
      ##2\par}\endgroup}%
  \def\email##1##2{\begingroup
    \@ifnotempty{##2}{\nobreak\noindent E-mail address%
      \@ifnotempty{##1}{, \ignorespaces##1\unskip}\/:\space
      \ttfamily##2\par}\endgroup}%
  \def\urladdr##1##2{\begingroup
    \def~{\char`\~}%
    \@ifnotempty{##2}{\nobreak\noindent\urladdrname
      \@ifnotempty{##1}{, \ignorespaces##1\unskip}\/:\space
      \ttfamily##2\par}\endgroup}%
  \addresses
  \endgroup
  \global\let\addresses=\@empty
}
\def\@setabstracta{%
    \ifvoid\abstractbox
  \else
    \skip@17pt \advance\skip@-\lastskip
    \advance\skip@-\baselineskip \vskip\skip@
    \box\abstractbox
    \prevdepth\z@ 
    \vskip-28pt
  \fi
}
\renewenvironment{abstract}{%
  \ifx\maketitle\relax
    \ClassWarning{\@classname}{Abstract should precede
      \protect\maketitle\space in AMS document classes; reported}%
  \fi
  \global\setbox\abstractbox=\vtop \bgroup
    \normalfont\small
    \list{}{\labelwidth\z@
      \leftmargin0pc \rightmargin\leftmargin
      \listparindent\normalparindent \itemindent\z@
      \parsep\z@ \@plus\p@
      
    }%
    \item[\hskip\labelsep\bfseries\abstractname.]%
}{%
  \endlist\egroup
  \ifx\@setabstract\relax \@setabstracta \fi
}

\def\ps@headings{\ps@empty
  \def\@evenhead{%
    \setTrue{runhead}%
    \normalfont\scriptsize
    \rlap{\thepage}\hfill
    \def\thanks{\protect\thanks@warning}%
    \leftmark{}{}}%
  \def\@oddhead{%
    \setTrue{runhead}%
    \normalfont\scriptsize
    \def\thanks{\protect\thanks@warning}%
    \rightmark{}{}\hfill \llap{\thepage}}%
  \let\@mkboth\markboth
}\ps@headings

\def\section{\@startsection{section}{1}%
  \z@{-1.4\linespacing\@plus-.5\linespacing}{.8\linespacing}%
  {\normalfont\bfseries\Large}}
\def\subsection{\@startsection{subsection}{2}%
  \z@{-.8\linespacing\@plus-.3\linespacing}{.5\linespacing\@plus.2\linespacing}%
  {\normalfont\bfseries\large}}
\def\subsubsection{\@startsection{subsubsection}{3}%
  \z@{.7\linespacing\@plus.2\linespacing}{-1.5ex}%
  {\normalfont\bfseries}}
\def\paragraph{\@startsection{paragraph}{4}%
  \z@{.7\linespacing\@plus.2\linespacing}{-1.5ex}%
  {\normalfont\itshape}}
\def\@secnumfont{\bfseries}

\renewcommand\contentsnamefont{\bfseries}
\def\@starttoc#1#2{\begingroup
  \setTrue{#1}%
  \par\removelastskip\vskip\z@skip
  \@startsection{}\@M\z@{\linespacing\@plus\linespacing}%
    {.5\linespacing}{
      \contentsnamefont}{#2}%
  \ifx\contentsname#2%
  \else \addcontentsline{toc}{section}{#2}\fi
  \makeatletter
  \@input{\jobname.#1}%
  \if@filesw
    \@xp\newwrite\csname tf@#1\endcsname
    \immediate\@xp\openout\csname tf@#1\endcsname \jobname.#1\relax
  \fi
  \global\@nobreakfalse \endgroup
  \addvspace{32\p@\@plus14\p@}%
  \let\tableofcontents\relax
}
\def\contentsname{Contents}
\def\l@section{\@tocline{2}{.5ex}{0mm}{5pc}{}}
\def\l@subsection{\@tocline{2}{0pt}{2em}{5pc}{}}
\makeatother
\fi 
\usepackage{etoolbox}
\makeatletter
\expandafter\patchcmd\csname\string\proof\endcsname{\itshape}{\bfseries}{}{}
\makeatother

\usepackage[marginratio=1:1]{geometry}

\usepackage[backend=bibtex,maxbibnames=10,maxcitenames=5,bibstyle=alphabetic,citestyle=alphabetic]{biblatex}
\defbibheading{bibliography}[\refname]{%
	\section*{#1}%
}

\usepackage{lmodern}
\usepackage[T1]{fontenc}

\usepackage{graphicx}
\usepackage{xcolor}

\newcommand{\mailto}[1]{\href{mailto:#1}{#1}}
\newcommand{\mathematicanb}[1]{\href{calculations/#1}{\texttt{#1}}}

\usepackage{mathtools}
\usepackage{thmtools}
\newcommand\numberthis{\addtocounter{equation}{1}\tag{\theequation}}

\declaretheorem[
	style=plain,
	name=Theorem,
	numberwithin=section,
	refname={Theorem,Theorems},
	Refname={Theorem,Theorems}
]{Thm}
\declaretheorem[
	style=plain,
	name=Proposition,
	numberlike=Thm,
	refname={Proposition,Propositions},
	Refname={Proposition,Propositions}
]{Prop}
\declaretheorem[
	style=plain,
	name=Corollary,
	numberlike=Thm,
	refname={Corollary,Corollaries},
	Refname={Corollary,Corollaries}
]{Cor}
\declaretheorem[
	style=plain,
	name=Conjecture,
	numberlike=Thm,
	refname={Conjecture,Conjectures},
	Refname={Conjecture,Conjectures}
]{Conj}
\declaretheorem[
	style=plain,
	name=Lemma,
	numberlike=Thm,
	refname={Lemma,Lemmas},
	Refname={Lemma,Lemmas}
]{Lem}
\declaretheorem[
	style=definition,
	name=Definition,
	numberlike=Thm,
	refname={Definition,Definitions},
	Refname={Definition,Definitions}
]{Def}
\declaretheorem[
	style=definition,
	name=Notation,
	numberlike=Thm,
	refname={Notation,Notations},
	Refname={Notation,Notations}
]{Not}
\declaretheorem[
	style=definition,
	name=Example,
	numberlike=Thm,
	refname={Example,Examples},
	Refname={Example,Examples}
]{Eg}
\declaretheorem[
	style=definition,
	name=Remark,
	numberlike=Thm,
	refname={Remark,Remarks},
	Refname={Remark,Remarks}
]{Rem}
\declaretheorem[
	style=definition,
	name=Observation,
	numberlike=Thm,
	refname={Observation,Observations},
	Refname={Observation,Observations}
]{Obs}
\declaretheorem[
	style=definition,
	name=Claim,
	numberlike=Thm,
	refname={Claim,Claims},
	Refname={Claim,Claims}
]{Claim}
\declaretheorem[
	style=definition,
	name=Identity,
	numberlike=Thm,
	refname={Identity,Identities},
	Refname={Identity,Identities}
]{Id}

\allowdisplaybreaks

\usepackage{braket}
\usepackage{bbm}
\usepackage{shuffle}

\newcommand{\dd}{\mathrm{d}}

\newcommand{\modsh}{\overset{\shuffle}{=}}
\newcommand{\moddel}{\overset{\delta}{=}}
\newcommand{\sh}{\shuffle}

\newcommand{\argdash}{\mathop{\textrm{---}}}
\newcommand{\dblslash}{\mathrel{/\kern-3pt/}}

\newcommand{\floor}[1]{\left\lfloor #1 \right\rfloor}
\newcommand{\size}[1]{{\#{#1}}}
\newcommand{\abs}[1]{\left\lvert {#1} \right\rvert}

\DeclareMathOperator{\sgn}{sgn}

\DeclareMathOperator{\Cyc}{Cyc}

\newcommand{\Proj}{\mathbb{P}}
\newcommand{\Q}{\mathbb{Q}}
\newcommand{\C}{\mathbb{C}}

\newcommand{\Hbl}{\mathcal{H}}
\newcommand{\Bl}{\mathcal{B}}
\newcommand{\rel}{\mathcal{R}}

\DeclareMathOperator{\pr}{pr}
\DeclareMathOperator{\CR}{cr}
\DeclareMathOperator{\Li}{Li}
\DeclareMathOperator{\id}{id}

\begin{document}
	
	\title[Dan's reduction method for multiple polylogarithms]{A review of Dan's reduction method \\ for multiple polylogarithms}
	\author{Steven Charlton}
	
	\address{Fachbereich Mathematik \\	
		Auf der Morgenstelle 10 (Geb\"aude C) \\
		72076 \\
		T\"ubingen \\
		Germany}
	
	\email{\mailto{charlton@math.uni-tuebingen.de}}
	
	\date{8 March 2017}
	
	\keywords{Multiple polylogarithms, hyperlogarithms, iterated integrals, depth reduction}
	\subjclass[2010]{Primary 11G55}
	
	\begin{abstract}
		In this paper we will give an account of Dan's reduction method \cite{dan2011surla2} for reducing the weight \( n \) \emph{multiple logarithm} \( I_{1,1,\ldots,1}(x_1, x_2, \ldots, x_n) \) to an explicit sum of lower depth multiple polylogarithms in \( \leq n - 2 \) variables.  
		
		We provide a detailed explanation of the method Dan outlines, and we fill in the missing proofs for Dan's claims.  This establishes the validity of the method itself, and allows us to produce a corrected version of Dan's reduction of \( I_{1,1,1,1} \) to \( I_{3,1} \)'s and \( I_4 \)'s.  We then use the \emph{symbol of multiple polylogarithms} to answer Dan's question about how this reduction compares with his earlier reduction of \( I_{1,1,1,1} \), and his question about the nature of the resulting functional equation of \( I_{3,1} \).
		
		Finally, we apply the method to \( I_{1,1,1,1,1} \) at weight 5 to first produce a reduction to depth \( \leq 3 \) integrals.  Using some functional equations from our thesis, we further reduce this to \( I_{3,1,1} \), \( I_{3,2} \) and \( I_5 \), modulo products.  We also see how to reduce \( I_{3,1,1} \) to \( I_{3,2} \), modulo \( \delta \) (modulo products and depth 1 terms), and indicate how this allows us to reduce \( I_{1,1,1,1,1} \) to \( I_{3,2} \)'s only, modulo \( \delta \).
	\end{abstract}
	
	\maketitle
	
	\tableofcontents
	
	\section{Introduction}
	
	In \cite{dan2008surla}, \citeauthor{dan2008surla} describes a strategy for attempting to attack Zagier's polylogarithm conjecture for \( \zeta_F(4) \).  He reduces Zagier's conjecture to an analytic conjecture \cite[Conjecture 3 in ][]{dan2008surla} about the existence of a \emph{regulator map} with certain properties, and the following combinatorial conjecture about relations between weight 4 \emph{multiple polylogarithms} (MPL's).
	
	\begin{Conj}[Conjecture 6 in \cite{dan2008surla}]
		The following sum is seen to vanish under the co-boundary map \( \delta_{2,2} \colon \Hbl_4(E) \to \Hbl_2(E) \wedge \Hbl_2(E) \), 
		\[
			B(x, y; z) \coloneqq  I_{3,1}(x, z) - I_{3,1}(y, z) - I_{3,1}(\tfrac{x}{y}, z) + I_{3,1}(\tfrac{1-x}{1-y}, z) - I_{3,1}(\tfrac{1 - 1/x}{1-1/y}, z) \, .
		\]  Therefore \( B(x,y; z) \) should be expressible as a linear combination of \( \Li_4 \) terms.
	\end{Conj}
	
	As Dan notes, this combinatorial conjecture was already proposed by Goncharov \cite{goncharov1994polylogarithms}.  It has since been resolved by Gangl \cite[Theorem 16 in][]{gangl2015weight4mpl}, who gives a a 122-term expression involving \( \Li_4 \)'s which has the same symbol as (some version of) \( B(x,y;z) \), modulo products.

	\subsubsection*{Reduction of \( I_{1,1,1,1} \)} A key ingredient which allows Dan to carry through the reduction of Zagier's conjecture to the above combinatorial conjecture, is his reduction of the weight 4 hyperlogarithm \( I(a \mid b, c,d,e \mid f) \) to \( I_{3,1} \)'s and \( \Li_4 \)'s.  By setting the lower bound \( a = 0 \) and the upper bound \( f = 1 \) we can equivalently work with the \emph{quadruple-logarithm} \( I_{1,1,1,1}(w, x, y, z) \).  Dan presents his reduction of \( I(a \mid b,c,d,e \mid f) \) to \( I_{3,1} \) and \( \Li_4 \)'s as the main theorem, Th\'eor\`eme 3, in \cite{dan2008surla}. \medskip
	
	Unfortunately, the final expression Dan gives for this reduction is not correct, as it contains a number of typos.  By exploiting the structure of the reduction, Gangl was able to provide the necessary corrections to make this reduction work.  These mistakes and Gangl's corrections are given in \autoref{thm:dan_wt4_previous} and \autoref{rem:dan_mistake_1} below.
	
	\subsubsection*{Reduction of \( I_{1,\ldots,1}\)}  Later, in \cite{dan2011surla2}, \citeauthor{dan2011surla2} outlines a completely general and systematic method for reducing the generic weight \( n \) hyperlogarithm \( I(a_0\mid a_1,\ldots,a_n\mid a_{n+1}) \) to a combination of iterated integrals in `\( \leq n - 2 \) variables'.  That is to say, he reduces the depth \( n \) iterated integral to a sum of depth \( n - 2 \) iterated integrals, so the number of `slots' for non-zero arguments decreases by 2.  We can again set the lower and upper bounds to \( 0 \) and \( 1 \), to work with the \emph{multiple logarithm} \( I_{1,\ldots,1}(x_1, \ldots, x_n) \) instead. \medskip
	
	This systematic reduction method is \emph{not} simply a generalisation of the weight 4 case, but it involves a very different approach.  Dan presents this systematic method via a number of unsubstantiated claims, with limited explanation of the various steps.  After applying it to the case \( n = 4 \), Dan obtains a different reduction for \( I(a \mid b,c,d,e \mid f) \) to \( I_{3,1} \)'s and \( \Li_4 \)'s, in Th\'eor\`eme 2 of \cite{dan2011surla2}.
	
	Again, this weight 4 reduction is not correct as even the \emph{symbol} (an algebraic invariant attached to MPL's, which should vanish on any genuine multiple polylogarithm identity) does not vanish.  Moreover, the results of this reduction method lack sufficient structure to enable `easy' error correction; Gangl was unable to identify any candidate typos in the result.  This situation naturally casts much doubt on the validity of Dan's systematic reduction method.  We were therefore motivated to undertake an investigation to determine whether the method is correct, and to rectify the weight 4 reduction in Th\'eor\`eme 2 if possible. \medskip
	
	Fortunately, the systematic reduction method itself \emph{is} correct, so by implementing it in Mathematica \cite{Mathematica} I can present a corrected version of Dan's result in \autoref{thm:dan_wt4_correct} below.  This result is similar enough to Dan's (in number of terms, sizes of coefficients, agreement of cross-ratio argument) to make me believe that this was the version he intended to write down.  It is still not clear exactly where in the calculation Dan's mistake occurred, but the effort needed to find it not warranted now that we know the systematic reduction method \emph{is} correct.
	
	\subsection{Structure of the paper}
	
	In \autoref{sec:mplbackground} we review some background material on the the properties of Chen iterated integrals, the definitions of hyperlogarithms and multiple polylogarithms, and the construction of the abstract algebraic space \( \Hbl_n(E) \) of weight \( n \) multiple polylogarithms over \( E \).  In \autoref{sec:danmethod} we explain the steps of Dan's systematic reduction method in detail, and provide all necessary proofs.  In \autoref{sec:dan4} we apply the method to \( I(a \mid b,c,d,e \mid f) \leftrightarrow I_{1,1,1,1}(w, y, x, z) \) to correct Dan's weight 4 reduction.  In \autoref{sec:dan5} we also apply the method to \( I_{1,1,1,1,1}(v, w, x, y, z) \) to obtain a new reduction, writing \( I_{1,1,1,1,1} \) in terms of depth \( \leq 3 \) integrals initially.  By using certain identities and functional equations from \cite{charlton2016identities} we can further reduce this to \( I_{3,1,1} \), \( I_{3,2} \) and \( I_5 \), modulo products, or even to \( I_{3,2} \) `modulo \( \delta \)'. \medskip
	
	The reader who is only interested in seeing the correct reduction of \( I_{1,1,1,1}(w,x,y,z) \) can turn directly to \autoref{thm:dan_wt4_correct} and to \autoref{thm:dan_wt4_previous}.  The reductions of \( I_{1,1,1,1,1}(v,w,x,y,z) \) can be found in \autoref{sec:dan5}, specifically \autoref{id:i11111_dan_reduction}, \autoref{thm:dan_wt5_as_311_32_5_nonexplicit} (or \autoref{thm:dan_wt5_as_311_32_5} for the explicit version) and \autoref{thm:dan_wt5_as_i32}. 
	
	\subsubsection*{Supporting calculations} Mathematica worksheets with supporting calculations are available at \url{https://www.math.uni-tuebingen.de/user/charlton/publications/dan_reduction/}, or as ancillary files from the arXiv.
	
	\subsection{Overview of the reduction method}\label{sec:overview}
	
	We close this introduction with an overview of how the systematic reduction is supposed to work.  This will allow the reader to have a broad overview of the steps in the method, without agonising over the details initially.  References to the relevant definitions, and proofs are provided.
	
	\subsubsection*{Set-up:} Introduce (\autoref{def:generalisedhyperlog}), a slight generalisation \[
	H(a_0 \mid a_1, \ldots, a_n \dblslash x \mid a_{n+1})  \, , \]
	of the usual hyperlogarithm \( I(x_0 \mid x_1, \ldots, x_m \mid x_{m+1}) \), to be defined using the differential form
	\[
	\omega(a_i, x) \coloneqq \frac{(a_i - x)}{(t - a_i)(t - x)} \, \dd t \, .	
	\]
	This reduces directly to the usual hyperlogarithm when \( x = \infty \).  For other \( x \), this generalisation can still be expressed in terms of the usual hyperlogarithm.  (\Autoref{prop:hasi})
	
	This gives meaning to the symbol \( [a_0 \mid a_1, \ldots, a_n \dblslash x \mid a_{n+1}] \) in the algebraic space \( \mathcal{H}_n(E) \) describing multiple polylogarithms.  (\Autoref{def:hasi})
	
	\subsubsection*{Swap out \( x \):} Show that the hyperlogarithms obtained by swapping out one of the \( a_i \)'s with the new parameter \( x \), namely
	\[
	[a_0 \mid a_1, \ldots, a_n \dblslash x \mid a_{n+1}] + [a_0 \mid a_1, \ldots, a_{i-1}, x, a_{i+1}, \ldots, a_n \dblslash a_i \mid a_{n+1}] \, ,
	\]
	can be reduced to a sum in \( \leq n-2 \) variables (\autoref{prop:switchx}).  In the algebraic space \( \Hbl_n(E) \), this is done with the \( A \) and \( B \) operators (\autoref{def:operatorA}, \autoref{def:operatorB}), and packaged into the \( D \) operator (\autoref{def:operatorD}).
	
	\subsubsection*{Build a transposition of \( a_i \):} Do this three times, to swap out \( a_i \leftrightarrow x \), then swap out \( a_j \leftrightarrow a_i \), and finally swap out \( x \leftrightarrow a_j \).  This gives a transposition 
	\[
	[a_0 \mid a_1, \ldots, a_i, \ldots, a_j, \ldots, a_n \dblslash x \mid a_{n+1}] + [a_0 \mid a_1, \ldots, a_j, \ldots, a_i, \ldots, a_n \dblslash x \mid a_{n+1}]	\, ,
	\]
	of the \( a_i \)'s as a sum in \( \leq n-2 \) variables (\autoref{prop:transposition}).
	
	\subsubsection*{Apply to \( a_1 a_2 \shuffle a_3 \ldots a_n \):} Each term in this product can be converted back to \( a_1 \ldots a_n \), by some suitable permutation.  The previous step allows us to write the corresponding integrals as a sum in \( \leq n-2 \) variables.
	
	Since \( \sum_{\sigma \in S(2,n-2)} \sgn(\sigma) = \floor{n/2} \) (\autoref{prop:signof2shufflen-2}), we can sum all of the terms and write \( \floor{n/2} [a_0 \mid a_1, \ldots, a_n \dblslash x \mid a_{n+1}] \) as a sum in \( \leq n-2 \) variables. (\Autoref{thm:reductionofhyperlog})
	
	\subsubsection*{Proceed in an efficient way:} Write down the terms in the following manner
	\begin{align*}
	[a_0 \mid a_1 a_2 \shuffle a_3 \ldots a_n \dblslash x \mid a_{n+1}] ={} & A^n_{1,2} + {} \\
	& {} + A^n_{1,3} + A^n_{2,3} + {} \\
	& {} + A^n_{1,4} + A^n_{2,4} + A^n_{3,4} + {} \\
	& {} + A^n_{1,5} + A^n_{2,5} + A^n_{3,5} + A^n_{4,5} + \cdots \, ,
	\end{align*}
	where 
	\[
		A^n_{i,j} = [a_0 \mid \ldots, \underbrace{a_1}_{\mathclap{\text{position \( i \)}}}, \ldots, \underbrace{a_2}_{\mathclap{\text{position \( j \)}}}, \ldots \dblslash x \mid a_{n+1}]
	\] has \( a_1 \) in position \( i \), and \( a_2 \) in position \( j \).  Each \( A^n_{i,j} + A^n_{i+1,j} \) is a transposition, so can be written as a sum in \( \leq n-2 \) variables.  This leaves \( A^n_{1,2} + A^n_{1,4} + A^n_{1,6} + \cdots \).  (\Autoref{lem:shuffletoA1n})
	
	\subsubsection*{Finish:} Use that \( A^n_{1,2m} = (A^n_{1,2m} - A^n_{1,2m-1}) + (A^n_{1,2m-1} - A^n_{1,2m-2}) + A^n_{1,2m-2} \), to replace \( A^n_{1,2m} \) with \( A^n_{1,2m-2} \) and some transpositions that are a sum in \( \leq n-2 \) variables.  Push this all the way down to \( A^n_{1,2} \) (\autoref{lem:sumofleftover}), and so write \( \floor{n/2} A_{1,2} = \floor{n/2} [a_0 \mid a_1, \ldots, a_n \dblslash x \mid a_{n+1}] \) again as a sum in \( \leq n-2 \) variables, but in a more economical way (\autoref{thm:reduction}).
		
	\subsection*{Acknowledgements} 
	
		This work was undertaken as part my PhD thesis, and was completed with the support of Durham Doctoral Scholarship funding.  The preparation of this paper occurred while in T\"ubingen, with the support of Teach@T\"ubingen Scholarship funding.
		
		I would like to thank Herbert Gangl for initially suggesting this project, and for providing (several times) the corrections to Dan's first reduction of \( I_{1,1,1,1} \), in \autoref{thm:dan_wt4_previous} below.  Moreover, the compendium of weight 4 MPL functional equations and identities from \cite{gangl2015weight4mpl} was helpful in obtaining the reduction in \autoref{thm:dan_wt4_correct} in terms of \( I_{3,1} \) only, and in analysing Dan's resulting \( I_{3,1} \) functional equation in \autoref{sec:i31_fe}.  That compendium also motivated a similar study of weight 5 MPL identities, presented already in \cite{charlton2016identities}, which was used in the weight 5 reduction in \autoref{sec:dan5} below.
		
	\section{Background on iterated integrals, and multiple polylogarithms}\label{sec:mplbackground}
	
	We recall the properties of Chen iterated integrals, the definition of multiple polylogarithms, and describe the construction of an algebraic space \( \Hbl_n(E) \) which captures the relations of weight \( n \) multiple polylogarithms.  We also review the symbol of multiple polylogarithms.
	
	\subsection{Chen iterated integrals} \label{sec:chenmpl}
	
	Let \( M \) be a smooth manifold, and let \( \eta_1, \ldots, \eta_n \) be a collection of smooth 1-forms on \( M \).  Let \( \gamma \colon [0,1] \to M \) be a piecewise smooth path on \( M \).
	
	\begin{Def}[Chen iterated integrals, \cite{chen1977integrals}]
		The iterated integral of \( \eta_1, \ldots, \eta_n \) along the path \( \gamma \) is defined by
		\[
			\int_\gamma \eta_1 \circ \cdots \circ \eta_n \coloneqq \int_{\Delta} \gamma^\ast\eta_1(t_1) \wedge \cdots \wedge \gamma^\ast\eta_n(t_n) \, ,
		\]
		where the region of integration \( \Delta \) consists of all \( n \)-tuples \( (t_1, t_2, \ldots, t_n) \) with \( 0 \leq t_1 < t_2 < \cdots < t_n \leq 1 \).  Here \( \gamma^\ast \eta_i \) is the pullback of \( \eta_i \) to the interval \( [0,1] \) by \( \gamma \).
	\end{Def}
	
	\begin{Rem}\label{rem:iterated}
		The name iterated integral is justified by the observation that
		\[
			\int_\gamma \eta_1 \circ \cdots \circ \eta_{n-1} \circ \eta_n = \int_{t=0}^1 \left( \int_{\gamma_t} \eta_1 \circ \cdots \circ \eta_{n-1} \right) \gamma^\ast \eta_n(t) \, ,
		\]
		where \( \gamma_t = \gamma\rvert_{[0,t]} \) is the restriction of \( \gamma \) to the interval \( [0,t] \).  This allows one to recursively find \( \int_\gamma \eta_1 \circ \cdots \circ \eta_n \) by integrating `shorter length' iterated integrals.
	\end{Rem}
	
	We will need to make use of the following properties of iterated integrals.
	
	\subsubsection*{Empty integral}  For any path \( \gamma \), the empty integral satisfies
	\[
		\int_\gamma \coloneqq 1
	\]
	
	\subsubsection*{Path composition}  Let \( \alpha, \beta : [0,1] \to M \) be two paths, such that \( \alpha(1) = \beta(0) \).  Write \( \alpha\beta \) for path obtained by travelling along \( \alpha \) first, followed by \( \beta \).  Then
	\[
	\int_{\alpha\beta} \eta_1 \circ \cdots \circ \eta_n = \sum_{i=0}^n \int_\alpha \eta_1 \circ \cdots \circ \eta_i \int_{\beta} \eta_{i+1} \circ \cdots \circ \eta_n \, .
	\]
	
	\subsubsection*{Shuffle product}  Iterated integrals along the path \( \gamma \) multiply using the shuffle product formula
	\[
		\int_\gamma \eta_1 \circ \cdots \circ \eta_r \int_\gamma \eta_{r+1} \circ \cdots \circ \eta_{r+s} = \sum_{\sigma \in S(r,s)} \int_\gamma \eta_{\sigma(1)} \circ \cdots \circ \eta_{\sigma(r+s)} \, ,
	\]
	Here \( S(r,s) \) is the set of \( (r,s) \)-shuffles, defined as the following subset of permutations in the symmetric group \( S_{r+s} \)
	\[
		S(r,s) \coloneqq \Set{\sigma \in S_{r+s} | \text{\(\sigma(1) < \cdots < \sigma(r)\), and \( \sigma(r+1) < \cdots < \sigma(r+s) \)} } \, .
	\]
	
	\begin{Rem}\label{rem:shuffle}
	We can formally write the integrands in the shuffle product result by way of the \emph{shuffle product of words} over the alphabet \( \mathfrak{A} = \Set{\eta_1, \ldots, \eta_n} \).  We have
	\[
		\eta_1 \cdots \eta_r \shuffle \eta_{r+1} \cdots \eta_{r+s} = \sum_{\sigma \in S(r,s)} \eta_{\sigma(1)} \cdots \eta_{\sigma(r+s)} \, .
	\]
	The terms in the shuffle product arise from \emph{riffle shuffling} the two words in all possible ways, like a deck of cards.   Taking the integral over \( \gamma \) of both sides gives the previous formula.
	\end{Rem}
	
	\begin{Def}[Recursive definition of \( \sh \)]\label{def:shuffle}
	The shuffle product of words over some alphabet \( \mathfrak{A} \) is amenable to a recursive definition, as follows
	\[
		a w_1 \shuffle b w_2 = a (w_1 \shuffle b w_2) + b (a w_1 \shuffle w_2) \, ,
	\]
	where \( w_1, w_2 \) are two words over some \( A \), and \( a, b \) are letters from the alphabet \( \mathfrak{A} \).  The base cases involving a shuffle product with the empty word \( \mathbbm{1} \) are as follows
	\[
		\mathbbm{1} \shuffle w = w \shuffle \mathbbm{1} \coloneqq w \, .
	\]
	\end{Def}
	
	These integrals also satisfy a number of further properties.
	
	\subsubsection*{Reversal of paths} If \( \gamma \colon [0,1] \to M \) is a path, and \( \gamma^{-1}(t) = \gamma(1-t) \) denotes the path \( \gamma \) traversed in the opposite direction, then we have the following equality
	\[
	\int_{\gamma ^{-1}} \eta_1 \circ \cdots \circ \eta_n = (-1)^n \int_{\gamma} \eta_n \circ \cdots \circ \eta_1 \, .
	\]
	
	\subsubsection*{Functoriality} If \( f \colon M' \to M \) is a smooth map between manifolds, and \( \gamma \colon [0,1] \to M' \) is a path in \( M \) ,then
	\[
	\int_\gamma f^\ast \eta_1 \circ \cdots \circ f^\ast \eta_n = \int_{f(\gamma)} \eta_1 \circ \cdots \circ \eta_n \, .
	\]
	
	\subsection{Definition of hyperlogarithms and multiple polylogarithns}
	
	Iterated integrals in the above sense are rather general objects.  We are interested in the specific case where \( M = \Proj^1(\C) \setminus S \), for some set of points \( S \), which leads to hyperlogarithms and multiple polylogarithms.  We introduce the following basic differential form of interest.
	
	\begin{Def}
		The unique differential form \( \omega(a_i) \) of degree 1, holomorphic on \( \Proj^1(\C) \setminus \Set{a_i} \) which has a pole of order 1 and residue \( +1 \) at \( a_i \) is
		\[
		\omega(a_i) \coloneqq \frac{\dd t}{t - a_i} \, .
		\]
	\end{Def}
	
	With the family of differential forms \( \omega(a_i) \), \( a_i \in S \), we define the hyperlogarithm functions as follows.
	
	\begin{Def}[Hyperlogarithm]
		Let \( a_0, \dots, a_{n+1} \in \Proj^1(\C) \), such that \( a_0 \neq a_1 \), and \( a_n \neq a_{n+1} \).  The hyperlogarithm \( I(a_0 \mid a_1, \ldots, a_n \mid a_{n+1}) \) is a complex multivalued function of the \( (n+2) \)-tuple \( (a_0, a_1, \ldots, a_n, a_{n+1}) \) defined by the following iterated integral
		\[
			I(a_0 \mid a_1, \ldots, a_n \mid a_{n+1}) \coloneqq \int_{a_0}^{a_{n+1}} \omega(a_1) \circ \cdots \circ \omega(a_n) \, ,
		\]
		where \( \gamma \) is some path from \( a_0 \) to \( a_{n+1} \).
		
		Since the integral depends on (the homotopy class in \( \Proj^1(\C) \setminus \Set{a_1, \ldots, a_n} \) of) the path \( \gamma \), the result is a multivalued function.  The requirement \( a_0 \neq a_1 \), and \( a_{n} \neq a_{n+1} \) ensures that the integral is convergent.
	\end{Def}

	Dan uses the notation \( H(a_0 \mid a_1, \ldots, a_n \mid a_{n+1}) \) for the hyperlogarithm \( I(a_0 \mid a_1, \ldots, a_n \mid a_{n+1}) \).  We prefer to use the \( I \) notation as this seems to be more established in the literature \cite{goncharov2001multiple,gangl2015weight4mpl}, although we compromise and separate the integral bounds \( a_0 \), \( a_{n+1} \) with `\( \mid \)' instead of the `;' used elsewhere.  We reserve \( H \) for Dan's generalisation, which we introduce in \autoref{sec:methodstart}.

	\begin{Rem}[Invariance under affine transformations]\label{rem:affine}
		The hyperlogarithm is invariant under affine transformations of the form \( a_i \mapsto \alpha a_i + \beta \), for \( \alpha \neq 0 \in \C \), and \( \beta \in \C \).  This follows from the functoriality of iterated integrals under the map \( \C\to \C, x \mapsto \alpha x + \beta \).
	\end{Rem}
	
	As Goncharov notes \cite[Section 2]{goncharov2001multiple}, the properties of the hyperlogarithm \( I(a_0 \mid a_1, \ldots, a_n \mid a_{n+1}) \) change drastically if any one of the variables \( a_i \) becomes \( 0 \).  It is therefore useful to introduce some extra notation to clearly distinguish the cases.
	
	\begin{Def}[Multiple polylogarithm]\label{def:mpl}
		Let \( s_1, \ldots, s_k \) be positive integers, and \( a_1, \ldots, a_k \in \C\setminus\set{0} \) be non-zero complex variables.  The multiple polylogarithm (MPL) \( I_{s_1,\ldots,s_k}(a_1,\ldots,a_k) \) is defined as the following hyperlogarithm
		\[
			I_{s_1,\ldots,s_k}(a_1,\ldots,a_k) \coloneqq I(0 \mid a_1, \{0\}^{s_1-1}, \ldots, a_k, \{0\}^{s_k-1} \mid 1) \, ,
		\]
		where \( \{0\}^s \) denotes the string \( 0, \ldots, 0 \) consisting of \( s \) repetitions of \( 0 \).
		
		We also define the following auxiliary notions
		\begin{itemize}
			\item The sum \( s_1 + \cdots + s_k \) of the indices is called the \emph{weight} of the MPL \( I_{s_1,\ldots,s_k} \).
			\item The total number \( k \) of indices is called the \emph{depth} of the MPL \( I_{s_1,\ldots,s_k} \).
		\end{itemize}
	\end{Def}
	
	\begin{Rem}[Multiple polylogarithm as series]
		More usually, we define the multiple polylogarithm functions \( \Li_{s_1,\ldots,s_k}(z_1,\ldots,z_k) \) by the following infinite series inside the polydisc \( \abs{z_i} < 1 \)
		\[
			\Li_{s_1,\ldots,s_k}(z_1,\ldots,z_k) = \sum_{0 < n_1 < \cdots < n_k} \frac{z_1^{n_1} \cdots z_k^{n_k}}{n_1^{s_1} \cdots n_k^{s_k}} \, .
		\]
		
		For \( \abs{z_i} < 1 \), this series definition is related to the iterated integral definition via the following change of variables
		\[
			\Li_{s_1,\ldots,s_k}(z_1,\ldots,z_k) = (-1)^k I_{s_1,\ldots,s_k}(\tfrac{1}{z_1\cdots z_k}, \tfrac{1}{z_2\cdots z_k}, \ldots, \tfrac{1}{z_k}) \, .
		\]
		A proof of this is given in Theorem 2.2 of \cite{goncharov2001multiple}.  We are therefore justified in using the name `multiple polylogarithm' in both cases.  We tend to prefer using the iterated integral definition for depth \( > 1 \); for depth \( 1 \) we use both \( \Li_n(z) \) and \( I_n(z) \) frequently.
	\end{Rem}
	
	\subsection{\texorpdfstring{The space of multiple polylogarithms \( \Hbl_n(E) \)}{The space of multiple polylogarithms H\_n(E)}}
	\label{sec:mplgroup}
	
	Dan constructs an algebraic space \( \Hbl_n(E) \) describing abstractly the relations between multiple polylogarithms over some field \( E \), modulo products.  Most of Dan's reduction procedure takes place in this algebraic space, once it is seen that certain basic relations for the \emph{generalised hyperlogarithm} (\autoref{def:generalisedhyperlog}) indeed hold in the algebraic space (\autoref{lem:sumtoB}).  The construction of \( \Hbl_n(E) \) is as follows. \medskip

	\subsubsection*{The subset \( E_\ast^{n+2} \) of \( (n+2) \)-tuples}  Write \( E_\ast^{n+2} \) for the following subset of \( (n+2) \)-tules
	\[
	E_\ast^{n+2} \coloneqq \Set{ (a_0, \ldots, a_{n+1}) \mid \text{\( a_0 \neq a_1 \) and \( a_{n} \neq a_{n+1} \)}} \, ,
	\]
	these corresponds to convergent iterated integrals.  We know from \autoref{rem:affine} that the usual iterated integrals 
	\[
	I(a_0 \mid a_1, \ldots, a_n \mid a_{n+1}) \coloneqq \int_{a_0}^{a_{n+1}} \omega(a_1) \circ \cdots \circ \omega(a_n) \, ,
	\]
	are invariant under affine transformations \( a_i \mapsto \alpha a_i + \beta \), \( \alpha \neq 0 \in \C, \beta \in \C \).  We capture this property by considering the quotient
	\[
	E_\ast^{n+2} / (E^\ast \times E)
	\]
	where \( (\alpha, \beta) \in E^\ast \times E \) acts as the affine transformations \( a_i \mapsto \alpha a_i + \beta \).
	
	\subsubsection*{The bialgebra \( \mathcal{A}(E) \)}  Write \( \mathcal{A}_n(E) \) for the \( \Q \)-vector space generated by the symbols
	\[
		[a_0 \mid a_1, \ldots, a_n \mid a_{n+1}]
	\] for \( (a_0, \ldots, a_{n+1}) \in E_\ast^{n+2} / (E^\ast \times E) \).  The graded vector space
	\[
	\mathcal{A}(E) \coloneqq \bigoplus_{n\geq0} \mathcal{A}_n(E)
	\]
	admits a bialgebra structure.
	
	The multiplication \( \cdot \) on \( \mathcal{A}(E) \) comes from the shuffle product \( \sh \) of iterated integrals.  We have 
	\begin{align*}
	& [a_0 \mid a_1, \ldots, a_k \mid a_{k+l+1}] \cdot [a_0 \mid a_{k+1}, \ldots, a_{k+1} \mid a_{k+l+1}] \\
	&= [a_0 \mid a_1\cdots a_k \shuffle a_{k+1}\cdots a_{k+l} \mid a_{k+l+1}] \\
	&= \sum_{\sigma \in S(k,l)} [a_0 \mid a_{\sigma(1)}, \ldots, a_{\sigma(k+k)} \mid a_{k+l+1}] \, .
	\end{align*}
	Here \( S(k,l) \) is the set of \( (k,l) \)-shuffles defined in the shuffle product property of \autoref{sec:chenmpl}, and \( \shuffle \) is the shuffle product of words introduced in \autoref{rem:shuffle}.  In the second line, we extend \( [a_0 \mid \argdash \mid a_{k+l+1}] \) to formal linear combinations of words over the alphabet \( \mathfrak{A} = \Set{a_1, \ldots, a_{k+l}} \). \medskip

	The coproduct \( \Delta \) is given by Goncharov's formula for the coproduct on the Hopf algebra of (motivic) iterated, Theorem 1.2 in \cite{goncharov2005galois}.  In the bialgebra \( \mathcal{A}(E) \), Goncharov's formula reads
	\begin{gather*}
	\Delta([a_0 \mid a_1, \ldots, a_n \mid a_{n+1}]) = {}\\
	\sum_{0 = i_0 < i_1 < \ldots < i_k < i_{k+1} = n+1} [a_0 \mid a_{i_1}, \ldots, a_{i_k} \mid a_{n+1}] \otimes \prod_{p=0}^k [a_{i_p} \mid a_{i_p + 1}, \ldots, a_{i_{p+1}-1} \mid a_{i_{p+1}}] \, .
	\end{gather*}
	
	\begin{Rem}
	This formula has an elegant interpretation in terms of cutting off segments of a semicircular polygon.  For example, the term:
	\[
	[a_0 \mid a_1, a_3, a_6 \mid a_9] \otimes [a_0 \mid a_1] \cdot [a_1 \mid a_2 \mid a_3] \cdot [a_3 \mid a_4, a_5 \mid a_6] \cdot [a_6 \mid a_7, a_8 \mid a_9]
	\]
	in the coproduct \( \Delta [a_0 \mid a_1, \ldots, a_8 \mid a_9] \) corresponds to cutting off the indicated segments from the semicircular polygon below:
	\begin{center}
		\ifpdf
\begingroup%
  \makeatletter%
  \providecommand\color[2][]{%
    \errmessage{(Inkscape) Color is used for the text in Inkscape, but the package 'color.sty' is not loaded}%
    \renewcommand\color[2][]{}%
  }%
  \providecommand\transparent[1]{%
    \errmessage{(Inkscape) Transparency is used (non-zero) for the text in Inkscape, but the package 'transparent.sty' is not loaded}%
    \renewcommand\transparent[1]{}%
  }%
  \providecommand\rotatebox[2]{#2}%
  \ifx\svgwidth\undefined%
    \setlength{\unitlength}{224bp}%
    \ifx\svgscale\undefined%
      \relax%
    \else%
      \setlength{\unitlength}{\unitlength * \real{\svgscale}}%
    \fi%
  \else%
    \setlength{\unitlength}{\svgwidth}%
  \fi%
  \global\let\svgwidth\undefined%
  \global\let\svgscale\undefined%
  \makeatother%
  \begin{picture}(1,0.5)%
    \put(0,0){\includegraphics[width=\unitlength,page=1]{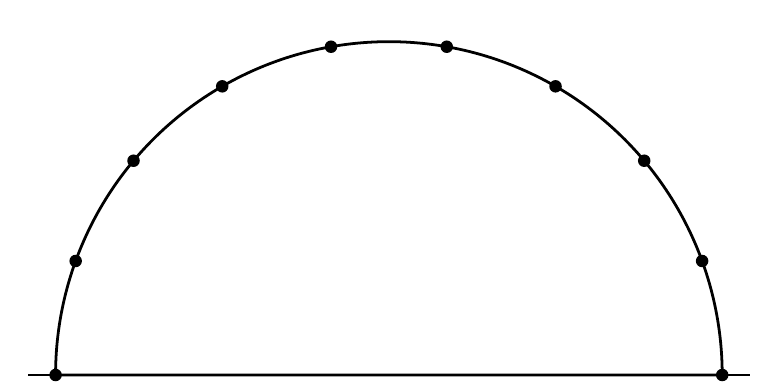}}%
    \put(0.96428574,0.03571433){\color[rgb]{0,0,0}\makebox(0,0)[b]{\smash{$a_9$}}}%
    \put(0.93628587,0.17665235){\color[rgb]{0,0,0}\makebox(0,0)[b]{\smash{$a_8$}}}%
    \put(0.85566355,0.31629425){\color[rgb]{0,0,0}\makebox(0,0)[b]{\smash{$a_7$}}}%
    \put(0.73214289,0.41994045){\color[rgb]{0,0,0}\makebox(0,0)[b]{\smash{$a_6$}}}%
    \put(0.5806224,0.47508942){\color[rgb]{0,0,0}\makebox(0,0)[b]{\smash{$a_5$}}}%
    \put(0.41937766,0.47508942){\color[rgb]{0,0,0}\makebox(0,0)[b]{\smash{$a_4$}}}%
    \put(0.2678572,0.41994045){\color[rgb]{0,0,0}\makebox(0,0)[b]{\smash{$a_3$}}}%
    \put(0.14433654,0.31629425){\color[rgb]{0,0,0}\makebox(0,0)[b]{\smash{$a_2$}}}%
    \put(0.06371417,0.17665235){\color[rgb]{0,0,0}\makebox(0,0)[b]{\smash{$a_1$}}}%
    \put(0.03571429,0.03571433){\color[rgb]{0,0,0}\makebox(0,0)[b]{\smash{$a_0$}}}%
    \put(0.03101021,-0.03617854){\color[rgb]{0,0,0}\makebox(0,0)[b]{\smash{}}}%
    \put(0,0){\includegraphics[width=\unitlength,page=2]{semicircular_polygon.pdf}}%
  \end{picture}%
\endgroup%

		\else
			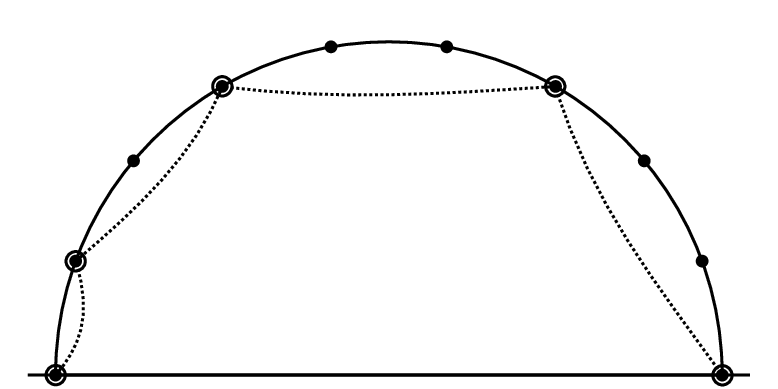
		\fi
	\end{center}
	The left hand term comes from the resulting main polygon above, while the right hand terms come from each individual cut-off segment.
	
	The other terms arise from taking all other possible choices of segments.
	\end{Rem}
	
	\subsubsection*{Lie coalgebra \( \Bl(E) \) of irreducibles}  Consider now the Lie coalgebra \( \Bl(E) \coloneqq \mathcal{A}(E) / (\mathcal{A}_{>0} \cdot \mathcal{A}_{>0}) \) of irreducibles.  This admits a co-derivation 
	\[
	\delta = \bigoplus_n \left( \delta_n \colon \Bl(E) \to \Bl(E) \otimes \Bl(E) \right) \, ,
	\] given by factoring the map \( \pr_\Bl \circ \Delta \) through \( \mathcal{A}_{>0} \cdot \mathcal{A}_{>0} \).
	
	This co-derivation will be used to inductively define the vector space of multiple polylogarithm relations \( \rel_n(E) \subset \Bl_n(E) \).  We will then set \( \Hbl(E) \coloneqq \Bl_n(E) / \rel_n(E) \) to be the space of multiple polylogarithms, and write \( [a_0 \mid a_1, \ldots,a_n \mid a_{n+1}] \) for the image of the same element in \( \Bl(E) \) modulo the relations \( \rel_n(E) \).
	
	\subsubsection*{Space \( \rel_n(E) \) of MPL relations}  The vector space of weight 1 relations is generated by the following elements
	\[
	\rel_1(E) \coloneqq \langle [a \mid z \mid b] + [b \mid z \mid c] = [a \mid z \mid c] : \text{\( z, a, b, c \in E \) with \( z \neq a, b, c \)} \rangle \, .
	\]
	We have
	\[
		I(a \mid z \mid b) = \log\left(\frac{z - b}{z - a}\right) \, ,
	\] so the generating of \( \rel_1(E) \) is equivalent to
	\[
		\log\left(\frac{z - b}{z - a}\right) + \log\left(\frac{z - c}{z-b}\right) = \log\left( \frac{z-b}{z-a} \cdot \frac{z - c}{z-b} \right) = \log\left(\frac{z - c}{z - b}\right) \, .
	\]
	This captures the fundamental functional equation of the logarithm function \( \log(x) \). \medskip
	
	Now write \( \mathcal{K}_n(E) \) for the kernel of the map
	\[
	(\pr_\Hbl \otimes \pr_\Hbl) \circ \delta_n \colon \Bl_n(E) \to (\Hbl(E) \otimes \Hbl(E))_n \, ,
	\]
	where \( \pr_\Hbl \colon \Bl_k(E) \to \Hbl_k(E) \) is already defined for \( k < n \).  We use this to define the full space of weight \( n \) MPL relations.
	
	\begin{Def}[Space of relations \( \rel_n(E) \)]
		The space of multiple polylogarithm relations is generated by the following elements
		\[
		\rel_n(E) \coloneqq \Set{ \alpha(1) - \alpha(0) \mid \alpha \in \mathcal{K}_n(E(t)) }
		\]
	\end{Def}
	
	The motivation for this seemingly abstruse definition of MPL relations comes from a result of Zagier \cite[Proposition 1 in][]{zagier1991polylogarithms}; an element \( \alpha(t) \) is in \( \mathcal{K}_n(E(t)) \) if and only if the map \( t \mapsto \mathcal{L}_n(\alpha(t)) \) is constant.  Here \( \mathcal{L}_n \) is certain modified version of the polylogarithm \( \Li_n \).  In our setting of \emph{multiple} polylogarithms, this means that the elements \( \alpha(1) - \alpha(0) \) are specialisations to \( E \) of all MPL functional equations instead.
	
	\subsubsection*{The space of \( \Hbl_n(E) \) of MPL's} With the above space of relations defined inductively, we also obtain at each step the space
	\[
		\Hbl_n(E) \coloneqq \Bl_n(E) / \rel_n(E) \, .
	\]
	This should be thought of as the space of weight \( n \) multiple polylogarithms, taken modulo products.
	
	The map \( (\pr_\Hbl \otimes \pr_\Hbl) \circ \delta_n \) factors through \( \rel_n(E) \), to give a map
	\[
	\delta_n \colon \Hbl_n(E) \to (\Hbl(E) \otimes \Hbl(E))_n \, .
	\]
	This gives \( \Hbl(E) \) the structure of a graded Lie coalgebra.

	\subsection{The symbol of multiple polylogarithms}
	
	The \emph{symbol} of multiple polylogarithms is an algebraic invariant in \( \Q(x_1,\ldots,x_m)^{\otimes n} \) attached to an iterated integral, which captures the differential properties of the integral.  The symbol was first defined by Goncharov, in Section 4 of \cite{goncharov2005galois}, under the name \( \otimes^m \)-invariant.

	The symbol attached to a weight \( n \) iterated integral, where the integrand is written as a \emph{total derivative} in the following form
	\[
		F = \int \dd \log(w_1) \circ \dd \log(w_2) \circ \cdots \circ \dd \log(w_n)
	\]
	is
	\[
		\mathcal{S}(F) \coloneqq w_1 \otimes w_2 \otimes \cdots \otimes w_n \in \Q(x_1,\ldots,x_m)^{\otimes n} \, .
	\]
	Here \( x_1, \ldots, x_m \) are the variables appearing in \( w_1, \ldots, w_n \).  More generally, the symbol of an arbitrary iterated integral can be computed by repeated differentiation to put it into the above form.
	
	One can also use the original trivalent tree definition of the \( \otimes^m \)-invariant as given in Section 4.4 of \cite{goncharov2005galois}.  In Proposition 4.5 of \cite{goncharov2005galois}, Goncharov shows that the \( \otimes^m \)-invariant can also be computed by maximally iterating the coproduct on his Hopf algebra of iterated integrals.  \citeauthor{duhr2012polygons} give another method to compute the symbol, by working in the algebra of \( R \)-deco polygons from \cite{gangl2009multiple}.  Duhr has implemented this method as part of the PolylogTools package \cite{PolylogTools} for Mathematica \cite{Mathematica}. \medskip
	
	The symbol should \emph{vanish} on any genuine multiple polylogarithm identity.  However, the symbol only captures the `top-slice' of such identities since it cannot see terms of the form \( \text{constant} \times \text{lower weight} \).  To access these, one may use further slices of the iterated integral coproduct.
	
	On the other hand, it is often useful to work at an even coarser level, by ignoring products, or ignoring products and depth 1 terms in the symbol.  This allows one to capture the `leading' terms in identities, and build up more precise identities from there.
	
	\subsubsection*{Symbol modulo products} In Section 5.4 of \cite{duhr2012polygons}, the authors describe a family of operators \( \Pi_w \) on length \( w \) tensors, where \( \Pi_1 = \id \), and
	\begin{align*}
	\Pi_w(a_1 \otimes \cdots \otimes a_w) \coloneqq \frac{w-1}{w} \left(
	\Pi_{w-1} (a_1 \otimes \cdots a_{w-1}) \otimes a_w - 
	\Pi_{w-1} (a_2 \otimes \cdots a_w) \otimes a_1
	\right) \, .
	\end{align*}
	These operators kill shuffle products.  Moreover, Proposition 1 of \cite{duhr2012polygons} establishes that the kernel of \( \Pi_w \) is exactly the ideal generated by all shuffle products.
	
	By applying \( \Pi_w \) we compute the symbol, modulo products.  When discussing identities which hold on this level, we will write \( {} \modsh {} \).  This is implemented in the PolylogTools \cite{PolylogTools} package as \texttt{sh}.
					
	\subsubsection*{Symbol modulo \(\delta\)} The coboundary map \( \delta \) from \( \Hbl_n(E) \) in \autoref{sec:mplgroup} has an analogue on the symbol, which kills the depth 1 terms \( \Li_n \) (and Nielsen polylogarithms, which conjectural should be expressible in terms of \( \Li_n \), but is currently unclear).  
	
	In the weight 4 case, the coboundary map \( \delta \) on symbols is described as an 8-fold symmetrisation in \cite{goncharov2010classical} sending \( a \otimes b \otimes c \otimes d \mapsto (a \wedge b) \wedge (c \wedge d) \).  Here it is used as an important tool in simplifying the 17 page weight 4 \emph{multiple polylogarithm} expression for the `two-loop Hexagon Wilson loop' \( R^{(2)}_{6,WL} \) in appendix H of \cite{duca2010two}.  The simplified expression consists of a \emph{single line} of classical \( \Li_4 \) polylogarithms.
	
	The coboundary map on symbols can be extended to higher weight.  Duhr implements it in the PolylogTools package \cite{PolylogTools} as \texttt{del}.  When discussing identities which hold on this level, modulo products and \( \Li_n \) terms, we will write \( {} \moddel {} \).

	\section{Dan's reduction method}\label{sec:danmethod}
	
	We are now in a position to discuss the reduction method itself.  The goal of this section is to provide a full account of the reduction method in the second paper \cite{dan2011surla2}, and to furnish any necessary explanations and proofs for all of the steps in the method.
	
	\subsubsection*{Supporting calculations} Dan's reduction procedure is implemented in \mathematicanb{dan\_procedure.m}.  This is in turn used by the worksheets \mathematicanb{wt4\_dan.nb} and \mathematicanb{wt5\_dan.nb} to produce the correct weight 4 reduction in \autoref{sec:dan4} and the new weight 5 reduction in \autoref{sec:dan5}.
	
	\subsection{Definition of Dan's generalised hyperlogarithm}\label{sec:methodstart}
	
	Before defining Dan's generalised hyperlogarithm, we introduce a generalisation of our `standard' differential form \( \omega(a_i) \).
	
	\begin{Def}
		The unique differential form \( \omega(a_i, x) \) of degree 1, holomorphic on \( \Proj^1(\C) \setminus \Set{ a_i, x } \) which is 0 if \( a_i = x \), and otherwise has a pole of order 1 and residue +1 at \( a_i \) and a pole of order 1 and residue \( -1 \) at \( x \) is
		\[
		\omega(a_i, x) \coloneqq \frac{(a_i - x)}{(t - a_i)(t - x)} \, \dd t \, .
		\]
		
		The correct differential form to take when \( x = \infty \) is
		\[
		\omega(a_i, \infty) \coloneqq \frac{\dd t}{t - a_i} \, ,
		\]
		since this agrees with \( \omega(\tfrac{1}{a_i}, 0)(s) \) under the change of variables \( s = \tfrac{1}{t} \), sending \( \infty \mapsto 0 \) and \( a_i \mapsto \tfrac{1}{a_i} \).
	\end{Def}
	
	We can then use the family of differential forms \( \omega(a_i, x) \) to define Dan's generalised hyperlogarithm function.
	
	\begin{Def}[Generalised hyperlogarithm]\label{def:generalisedhyperlog}
		Let \( a_0, \ldots, a_{n+1}, x \in \Proj^1(\C) \), such that \( a_0 \neq a_1 \), \( a_0 \neq x \), \( a_n \neq a_{n+1} \) and \( x \neq a_{n+1} \).  Then the \emph{generalised hyperlogarithm} \( H(a_0 \mid a_1, \ldots, a_n \dblslash x \mid a_{n+1}) \) is defined by the following iterated integral
		\[
		H(a_0 \mid a_1, \ldots, a_n \dblslash x \mid a_{n+1}) \coloneqq \int_{a_0}^{a_{n+1}} \omega(a_1, x) \circ \omega(a_2, x) \circ \cdots \circ \omega(a_n, x) \, ,
		\]
		for some path \( \gamma \) from \( a_0 \) to \( a_{n+1} \) in \( \Proj^1(\C) \setminus \Set{a_1, \ldots, a_n, x} \).  The requirements \( a_0 \neq a_1 \), \( a_0 \neq x \), \( a_n \neq a_{n+1} \) and \( x \neq a_{n+1} \) ensure that the resulting integral is convergent.
	\end{Def}
	
	The relationship between Dan's generalised hyperlogarithm, and the ordinary hyperlogarithm is straightforward.
	
	\begin{Prop}\label{prop:hasi}
		If \( x = \infty \) then
		\[
		H(a_0 \mid a_1, \ldots, a_n \dblslash \infty \mid a_{n+1}) = I(a_0 \mid a_1, \ldots, a_n \mid a_{n+1}) \, .
		\]
		Otherwise
		\[
		H(a_0 \mid a_1, \ldots, a_n \dblslash x \mid a_{n+1}) = I(\tfrac{1}{a_0 - x} \mid \tfrac{1}{a_1 - x}, \ldots, \tfrac{1}{a_n - x} \mid \tfrac{1}{a_{n+1} - x}) \, .
		\]
		
		\begin{proof}
			If \( x = \infty \), then the differential form \( \omega(a_i, x) \) reduces to the `standard' form \( \omega(a_i) = \frac{\dd t}{t - a_i} \) appearing in the definition of the hyperlogarithm, so the result holds.
			
			Otherwise, change variables via \( t' = 1/(t - x) \), which sends \( x \mapsto \infty \), and \( a_i \mapsto 1/(a_i - x) \).  We have that \( t = \frac{1}{t'} + x \), so that
			\begin{align*}
			\omega(a_i, x)(t) &= \frac{(a_i - x)}{(t-a_i)(t-x)} \, \dd t \\
			&= \frac{a_i - x}{(\frac{1}{t'} + x - a_i)(\frac{1}{t'} + x - x)} \frac{-1}{t'^2} \, \dd t' \\
			&= \frac{1}{t' - \frac{1}{a_i-x}} \, \dd t' \\
			&= \omega(\tfrac{1}{a_i-x}, \infty)(t') \, .
			\end{align*}
			The bounds \( a_0 \) and \( a_{n+1} \) change to \( \tfrac{1}{a_0 - x} \) and \( \tfrac{1}{a_{n+1} - x} \) respectively, so we obtain the required result.
		\end{proof}
	\end{Prop}
	
	We can use the above relation to the usual hyperlogarithm, to give meaning to the symbol \( [a_0 \mid a_1, \ldots, a_n \dblslash x \mid a_{n+1}] \) in the space \( \Hbl_n(E) \) of multiple polylogarithms on \(E\) , as follows.
	
	\begin{Def}\label{def:hasi}
		We set
		\[
		[a_0 \mid a_1, \ldots, a_n \dblslash \infty \mid a_{n+1}] \coloneqq [a_0 \mid a_1, \ldots, a_n \mid a_{n+1}] \, ,
		\] and for \( x \neq \infty \),
		\[
		[a_0 \mid a_1, \ldots, a_n \dblslash x \mid a_{n+1}] \coloneqq [\tfrac{1}{a_0 - x} \mid \tfrac{1}{a_1 - x}, \ldots, \tfrac{1}{a_n - x} \mid \tfrac{1}{a_{n+1} - x} ] \, .
		\]
	\end{Def}
	
	This completes the first step of the method, as outlined in \autoref{sec:overview}.  We now want to start the second step of the method, but we first must study some relations which hold on the level of iterated integrals, before moving to \( \Hbl_n(E) \). \medskip
	
	Dan observes the following `decomposition' of the differential forms \( \omega(a_i, x) \), and uses this to give a relation  expressing \( H(a_0 \mid a_1, \ldots, a_n \dblslash x \mid a_{n+1}) \) in terms of expressions of the forms \( H(a_0 \mid \argdash \dblslash y \mid a_{n+1}) \).  This relation is later exploited as a key tool in the reduction procedure.
	
	\begin{Obs}\label{obs:difference_of_omegas}
		By a straightforward calculation, we can write
		\[
		\omega(a_i, x) = \omega(a_i, y) - \omega(x, y) \, .
		\]
	\end{Obs}

	Before stating this relation, it is convenient to introduce some notation about replacing variables in certain arguments of the generalised hyperlogarithm.  Some related notation, in the form of the \( A \) operator in \autoref{def:operatorA}, will be used in the space \( \Hbl_(E) \) later.
	
	\begin{Def}[\( X \) operator]
		Let \( I \) be a subset of \( \Set{1, 2, \ldots, n} \).  Define
		\[
		X(H(a_0 \mid a_1, \ldots, a_n \dblslash x \mid a_{n+1}), I, y)
		\] to be the hyperlogarithm
		\[
		H(a_0 \mid a_1, \ldots, a_n \dblslash x \mid a_{n+1})
		\] where the positions \( j \in I \) are replaced by the variable \( y \).
	\end{Def}
	
	\begin{Eg}
		For example, we have
		\begin{align*}
		& X(H(a_0 \mid a_1,a_2,a_3,a_4 \dblslash x \mid a_5), \Set{1,3}, y) = H(a_0 \mid y, a_2, y, a_4 \dblslash x \mid a_5) \\
			& X(H(a_0 \mid a_1,a_2,a_3,a_4,a_5,a_6 \dblslash x \mid a_7), \Set{3,4,5,6}, z) = H(a_0 \mid a_1, a_2, y,y,y,y, \dblslash x \mid a_7)
		\end{align*}
	\end{Eg}
	
	\begin{Prop}\label{prop:Hx_as_alternating_sum_of_Hy}
		The hyperlogarithm \( H(a_0 \mid a_1, \ldots, a_n \dblslash x \mid a_{n+1}) \) can be expressed as an alternating sum of hyperlogarithms of the form \( H(a_0 \mid \argdash \dblslash y \mid a_{n+1}) \).  More precisely, we have
		\[
		H(a_0 \mid a_1, \ldots, a_n \dblslash x \mid a_{n+1}) = \sum_{I} (-1)^\size{I} X(H(a_0 \mid a_1, \ldots, a_n  \dblslash y \mid a_{n+1}), I, x) \, ,
		\] where the sum is taken over all \( I \subset \Set{1, 2, \ldots, n} \).
		
		\begin{proof}
			We can prove this by induction on the depth \( n \).  In the case \( n = 1 \), we explicitly write out both sides.  On the left hand side we have \( H(a_0 \mid a_1 \dblslash x \mid a_2) \), and on the right hand side we have
			\[
			\sum_{I} (-1)^\size{I} X(H(a_0 \mid a_1 \dblslash y \mid a_2), I, x) \, ,
			\] taken over all \( I  \subset \Set{1} \).  That is, over \( I = \emptyset, \Set{ 1 } \).  This gives
			\[
			H(a_0 \mid a_1 \dblslash y \mid a_2)-H(a_0 \mid x \dblslash y \mid a_2) \, ,
			\] which is equal to
			\begin{align*}
			\int_{a_0}^{a_2} \omega(a_1, y) - \omega(x,y) &= \int_{a_0}^{a_2} \omega(a_1, x) \\
			&= H(a_0 \mid a_1 \dblslash x \mid a_2) \, ,
			\end{align*} using \autoref{obs:difference_of_omegas}.  Hence the case \( n = 1 \) holds. \medskip
			
			Suppose now that the result holds for depth \( n-1 \).  Then for depth \( n \) we have the following.  We can sum over \( I \subset \Set{1, \ldots, n} \) by first taking \( I \) with \( n \in I \), and then taking \( I \) with \( n \notin I \).  So
			\begin{align*}
			&\sum_{I} (-1)^\size{I} X(H(a_0 \mid a_1, \ldots, a_n \dblslash y \mid a_{n+1}), I, x) = \\
			&\hspace{2em} \sum_{\substack{\text{\( I \) with} \\ n \in I}} (-1)^\size{I} X(H(a_0 \mid a_1, \ldots, a_n \dblslash y \mid a_{n+1}), I, x) + \\ & \hspace{2em} {} + \sum_{\substack{\text{\( I \) with } \\ n \notin I}} (-1)^\size{I} X(H(a_0 \mid a_1, \ldots, a_n \dblslash y \mid a_{n+1}), I, x) \, .
			\end{align*}
			In the first sum we know \( n \in I \), so we can remove \( n \) from \( I \), replace \( a_n \) with \( x \) and insert one minus sign already.  Then the sum is over \( I' \subset \Set{1, \ldots, n-1} \). In the second sum, \( 1 \notin I \), so the sum is over \( I' \subset \Set{1, \ldots, n-1} \) already, giving
			\begin{align*}
			&= {} - \sum_{I'} (-1)^\size{I'} X(H(a_0 \mid a_1, \ldots, a_{n-1}, x \dblslash y \mid a_{n+1}), I', x) + \\ & \hspace{2em} {} + \sum_{I'} (-1)^\size{I'} X(H(a_0 \mid a_1, \ldots, a_{n-1}, a_n \dblslash y \mid a_{n+1}), I', x) \, .
			\end{align*}
			
			Now recall from \autoref{rem:iterated} that iterated integrals do deserve the name iterated, meaning the iterated integral \( H(a_0 \mid a_1, \ldots, a_n \dblslash x \mid a_{n+1}) \) can be expanded as follows
			\begin{align*}
			H(a_0 \mid a_1, \ldots, a_n \dblslash x \mid a_{n+1}) &= \int_{t=a_0}^{a_{n+1}} H(a_0 \mid a_1, \ldots, a_{n-1} \dblslash x \mid t) \cdot \omega(a_n, x)(t) \, .
			\end{align*}
			If we do this with the integrals in the sum above, we obtain
			\begin{align*}
			&= \int_{t=a_0}^{a_{n+1}}\sum_{I'} (-1)^\size{I'} X(H(a_0 \mid a_1, \ldots, a_{n-1} \dblslash y \mid t), I', x) \cdot \omega(x,y)(t) + {} \\ 
			& \hspace{6em} {} + \sum_{I'} (-1)^\size{I'} X(H(a_0 \mid a_1, \ldots, a_{n-1} \dblslash y \mid t), I', x) \cdot \omega(a_n, y)(t)
			\end{align*}
			Using the induction assumption, this can be written as
			\begin{align*}
			& = \int_{t=a_0}^{a_{n+1}} H(a_0 \mid a_1, \ldots, a_{n-1} \dblslash x \mid t) \cdot -\omega(x,y)(t) + {} \\
			& \hspace{6em} {} + H(a_0 \mid a_1, \ldots, a_{n-1} \dblslash x \mid t) \cdot \omega(a_n,y)(t) \\
			&= \int_{t=a_0}^{a_{n+1}} H(a_0 \mid a_1, \ldots, a_{n-1} \dblslash x \mid t) \cdot (\omega(a_n,y)(t) - \omega(x,y)(t))
			\end{align*}
			Now apply \autoref{obs:difference_of_omegas} to rewrite the difference of \( \omega \)'s, and use the \autoref{rem:iterated} to evaluate the result as an iterated integral, and we obtain
		\begin{align*}
			&= \int_{t=a_0}^{a_{n+1}} H(a_0 \mid a_1, \ldots, a_{n-1} \dblslash x \mid t) \cdot  \omega(a_n, x)(t) \\ &= H(a_0 \mid a_1, a_2, \ldots, a_n, \dblslash x \mid a_{n+1}) \, .
		\end{align*}
		This completes the proof.
		\end{proof}
	\end{Prop}
	
	\subsection{\texorpdfstring{Operators \( A \) and \( B \), reducing to \( \leq n - 2 \) variables}{Operators A and B, reducing to <= n-2 variables}}
	
	We start now with the second step proper of the method, as outlined in \autoref{sec:overview}.
	
	We will exploit the identity \autoref{prop:Hx_as_alternating_sum_of_Hy} on hyperlogarithms to obtain certain useful relations in the space \( \Hbl_n(E) \).  To do this, we introduce the operators \( A \) and \( B \) which will give us tools to systematically reduce the hyperlogarithms.  The payoff comes with \autoref{obs:nminus1vars} where we initially see a reduction to \( n - 1 \) variables using \( B \), and then in \autoref{prop:switchx} where we get a reduction to \( \leq n - 2 \) variables after some further work.
	
	\begin{Def}[\( A \) operator]\label{def:operatorA}
		Let \( 1 \leq i \leq n \) and let \( I \) be a subset of \( \Set{1, 2, \ldots, n} \) containing \( i \).  Define
		\[
		A([a_0 \mid a_1, \ldots, a_n \dblslash x \mid A_{n+1}], i, I)
		\] to be the symbol
		\[
		[a_0 \mid a_1, \ldots, a_n \dblslash x \mid a_{n+1}]
		\] where the positions \( j \in I \) are replaced by the variable \( a_i \) from position \( i \).
	\end{Def}
	
	\begin{Eg}
		For example, we have
		\begin{align*}
		& A([a_0 \mid a_1, a_2, a_3, a_4 \dblslash x \mid a_5], \Set{1, 3}, 3) = [a_0 \mid a_3, a_2, a_3, a_4 \dblslash x \mid a_5] \\
		& A([a_0 \mid a_1, a_2, a_3, a_4, a_5, a_6 \dblslash x \mid a_7], \Set{3, 4, 5, 6}, 4) = [a_0 \mid a_1, a_2, a_4, a_4, a_4, a_4 \dblslash x \mid a_7]
		\end{align*}
	\end{Eg}
	
	\begin{Def}[\( B \) operator]\label{def:operatorB}
		Now define
		\[
		B([a_0 \mid a_1, \ldots, a_n \dblslash x \mid a_{n+1}], i) \coloneqq \sum_{I} (-1)^\size{I} A([a_0, \mid a_1, \ldots, a_n \dblslash x \mid a_{n+1}], i, I) \, ,
		\]
		where the sum is taken over all subsets \( I \) of the set \( \Set{1,2,\ldots,n} \) containing \( i \) and having cardinality \( \size{I} \geq 2 \).
	\end{Def}
	
	\begin{Eg}
		In computing \( B \) with \( i = 2 \) and \( n = 3 \), we would have to sum over the sets \( \Set{1,2}, \Set{2,3}, \Set{1,2,3} \).  So we get
		\begin{align*}
		B([a_0 \mid a_1, a_2, a_3 \dblslash x \mid a_4], 2) &= \begin{aligned}[t]
		&(-1)^{\size{\Set{1,2}}} A([a_0 \mid a_1, a_2, a_3 \dblslash x \mid a_4], 2, \Set{1,2}) \\
		& + (-1)^{\size{\Set{2,3}}} A([a_0 \mid a_1, a_2, a_3 \dblslash x \mid a_4], 2, \Set{2,3}) \\
		& + (-1)^{\size{\Set{1,2,3}}} A([a_0 \mid a_1, a_2, a_3 \dblslash x \mid a_4], 2, \Set{1,2,3})
		\end{aligned}\\[1ex]
		&= \begin{aligned}[t]
		&[a_0 \mid a_2, a_2, a_3 \dblslash x \mid a_4] \\
		& + [a_0 \mid a_1, a_2, a_2 \dblslash x \mid a_4] \\
		& - [a_0 \mid a_2, a_2, a_2 \dblslash x \mid a_4] \, .
		\end{aligned}
		\end{align*}
	\end{Eg}
	
	According to Dan, the considerations from \autoref{prop:Hx_as_alternating_sum_of_Hy}, applied when \( y = a_i \), suggest a relation in \( \Hbl_n(E) \).  Indeed, setting \( y = a_i \) in \autoref{prop:Hx_as_alternating_sum_of_Hy} gives
	\[
	H(a_0 \mid a_1, \ldots, a_n \dblslash x \mid a_{n+1}) = \sum_{I} (-1)^\size{I} X(H(a_0 \mid a_1, \ldots, a_n \dblslash a_i \mid a_{n+1}), I, x) \, .
	\]
	Notice that whenever \( i \notin I \), so that \( a_i \) is not replaced by \( x \), we obtain an integral like \( H(a_0 \mid \ldots, a_i, \ldots \dblslash a_i \mid a_{n+1}) \) which contains the differential form \( \omega(a_i, a_i) = 0 \).  The resulting integral is therefore 0, and does not contribute to the total.  It makes sense, then, to reduce the sum to \( I \subset \Set{1, 2, \ldots, n} \), such that \( i \in I \).  Moreover, there is only one possible \( I' \) with \( \size{I'} = 1 \), so we can deal with term separately.  We obtain
	\begin{align*} H(a_0 \mid a_1, \ldots, a_n \dblslash x \mid a_{n+1}) & = - H(a_0 \mid a_1, \ldots, a_{i-1}, x, a_i, \ldots, a_n \dblslash a_i \mid a_{n+1}) + {} \\
	& \hspace{2em} {} + \sum_{I'} (-1)^\size{I'}  X(H(a_0 \mid a_1, \ldots, a_n \dblslash a_i \mid a_{n+1}), I', x) \, ,
	\end{align*}
	where the sum is taken over all \( I' \subset \Set{1, 2, \ldots,n} \) such that \( i \in I \) and \( \size{I} \geq 2 \).
	
	Rearranging this gives
	\begin{align*}
	& H(a_0 \mid a_1, \ldots, a_n \dblslash x \mid a_{n+1}) +  H(a_0 \mid a_1, \ldots, a_{i-1}, x, a_i, \ldots, a_n \dblslash a_i \mid a_{n+1}) \\
	& \hspace{5em} = \sum_{I'} (-1)^\size{I'} X(H(a_0 \mid a_1, \ldots, a_n \dblslash a_i \mid a_{n+1}), I', x) \\
	& \hspace{5em} = \sum_{I'} (-1)^\size{I'} X(H(a_0 \mid a_1, \ldots, a_n \dblslash x \mid a_{n+1}), I', a_i) \, .
	\end{align*}
	The last equality comes from the symmetry under \( a_i \leftrightarrow x \) in the first line.  We can translate this result to \( \Hbl_n(E) \) to obtain the following.
	
	\begin{Lem}\label{lem:sumtoB}
		In \( \Hbl_n(E) \) the following the following relation holds
		\begin{align*}
		& [a_0 \mid a_1, \ldots, a_n \dblslash x \mid a_{n+1}] \\
		& + [a_0 \mid a_1, \ldots, a_{i-1}, x, a_{i+1}, \ldots, a_n \dblslash a_i \mid a_{n+1}] \\
		& {} = B([a_0 \mid a_1, \ldots, a_n \dblslash x \mid a_{n+1}], i) \, .
		\end{align*}
		
		\begin{proof}
			From the discussion above, we have the result
			\begin{align*}
			& H(a_0 \mid a_1, \ldots, a_n \dblslash x \mid a_{n+1}) + {} \\ 
			& \hspace{2em} {} + H(a_0 \mid a_1, \ldots, a_{i-1}, x, a_{i+1}, \ldots, a_n \dblslash a_i \mid a_{n+1}) + {} \\ 
			& \hspace{2em} {} - B(H(a_0 \mid a_1, \ldots, a_n \dblslash x \mid a_{n+1}), i) = 0
			\end{align*}
			on the level of integrals, with the obvious re-definition of \( B \) to hyperlogarithms.
			
			Replace \( a_{n+1} \rightsquigarrow a_0 + t (a_{n+1} - a_0) \), and go to \( \Hbl_n(E) \) to define
			\begin{align*}
			\alpha(t) & = [a_0 \mid a_1, \ldots, a_n \dblslash x \mid  a_0 + t (a_{n+1} - a_0)] + {} \\ 
			& \hspace{2em} {} + [a_0 \mid a_1, \ldots, a_{i-1}, x, a_{i+1}, \ldots, a_n \dblslash a_i \mid  a_0 + t (a_{n+1} - a_0)] + {} \\ 
			& \hspace{2em} {} - B([a_0 \mid a_1, \ldots, a_n \dblslash x \mid  a_0 + t (a_{n+1} - a_0)], i) \, .
			\end{align*}
			This means \( \alpha(t) \) is constant on the level of integrals, and so \( \alpha(t) \in \mathcal{K}_n(E(t)) \) by Proposition 1 in \cite{zagier1991polylogarithms}.  Therefore we find the following relation \( \alpha(1) - \alpha(0) \in \rel_n(E) \).  But this difference evaluates to
			\begin{align*}
			\alpha(1) - \alpha(0) & =  [a_0 \mid a_1, \ldots, a_n \dblslash x \mid a_{n+1}] + {} \\ 
			& \hspace{2em} {} + [a_0 \mid a_1, \ldots, a_{i-1}, x, a_{i+1}, \ldots, a_n \dblslash a_i \mid   a_{n+1}] + {} \\
			& \hspace{2em} {} - B([a_0 \mid a_1, \ldots, a_n \dblslash x \mid a_{n+1}], i) \, .
			\end{align*}
			Here we use make use of the fact that \( [a_0 \mid \argdash \dblslash x \mid a_0] = 0 \) in \( \Hbl_n(E) \).  We therefore get the claimed result.
		\end{proof}
	\end{Lem}
	
	Now we give a way of `shuffling out' variables from the first position of a hyperlogarithm.  This will be used in \autoref{obs:nminus1vars} to ensure \( a_i \) does not appear in the first position.  We may then split the integral at position \( a_i \) to further reduce the number of variables in each term to obtain the key result in \autoref{prop:switchx}.
	
	\begin{Lem}\label{lem:shuffleouty}
		In \( \Hbl_n(E) \) we have for \( 0 \leq s \leq n \), that
		\begin{align*}
		& [a_0 \mid \{y\}^s, b_{s+1}, b_{s+2}, \ldots, b_n \dblslash x \mid a_{n+1}] \\ 
		&= (-1)^s [a_0 \mid b_{s+1}, \{y\}^s \shuffle (b_{s+2} \cdots b_n) \dblslash x \mid a_{n+1}] \\
		&= (-1)^s \sum_{J} C_J \, .
		\end{align*}
		Here \( C_J \) denotes the symbol \( [a_0 \mid b_{s+1}, \argdash \dblslash x \mid a_{n+1}] \) with the positions \( J \) occupied by \( y \), and the remaining positions by \( b_{s+2}, \ldots, b_n \) in that order.  In the sum, \( J \) runs through subsets of size \( s \) of the set \( \Set{2,3, \ldots, n} \), and \( \Set{y}^s \) means \( y, \ldots, y \) repeated \( s \) times.
		
		\begin{proof}			
			Firstly, observe that the equality with \( \sum_J C_J \) in the third line just comes from writing out the terms of the shuffle product on the previous line.
			
			Each term in the shuffle product \( b_{s+1} \big( \{ y \}^s \shuffle (b_{s+2} \cdots b_n) \big) \) is uniquely determined by which positions contain \( y \).  Since we prepend the result with \( b_{s+1} \), these positions are in the range \( \Set{2, 3, \ldots, n} \), and any subset of these occurs.
			
			We therefore only need to tackle the equality between the first and second lines. \medskip
			
			We see this equality is true for \( s = 0 \).  The first line is
			\[
			[a_0 \mid \{y\}^0, b_{0+1}, b_{0+2}, \ldots, b_n \dblslash x \mid a_{n+1}] = [a_0 \mid b_1, b_2, \ldots, b_n \dblslash x \mid a_{n+1}] \, .
			\]
			The second line is
			\begin{align*}
			& (-1)^0[a_0 \mid b_{0+1}, \{y\}^0 \shuffle (b_{0+2} \cdots b_n) \dblslash x \mid a_{n+1}] \\
			&= [a_0 \mid b_{0+1}, \emptyset \shuffle (b_{0+2} \cdots b_n) \dblslash x \mid a_{n+1}] \\
			&= [a_0 \mid b_1, b_2, \ldots, b_n \dblslash x \mid a_{n+1}] \, .
			\end{align*}
			These are indeed equal. \medskip
			
			Now comes the inductive step.  Recall the inductive definition of \( \shuffle \) from \autoref{def:shuffle}.  It says that
			\[
			aw_1 \shuffle bw_2 = a(w_1 \shuffle bw_2) + b(aw_1\shuffle w_2) \, .
			\]
			So we have that
			\begin{equation}\label{lem:yshuffleb}
			\{y\}^{s+1}\shuffle (b_{s+2}\cdots b_n) = y(\{y\}^s\shuffle b_{s+2}\cdots b_n) + b_{s+2}(\{y\}^{s+1}\shuffle b_{s+3}\cdots b_n) \, .
			\end{equation}
			Therefore we scompute that
			\begin{align*}
			&[a_0 \mid \{y\}^{s+1}, b_{s+2}, b_{s+3}, \cdots, b_n \dblslash x \mid a_{n+1}] \\
			&= [a_0 \mid \{y\}^{s}, y, b_{s+2}, b_{s+3}, \cdots, b_n \dblslash x \mid a_{n+1}] \\
			&= (-1)^s[a_0 \mid y, \{y\}^s \shuffle (b_{s+2} \cdots b_n) \dblslash x \mid a_{n+1}] \, ,
			\end{align*}
			using the induction assumption for \( s \) with \( b_{s+1} = y \).  Now use the relation in \autoref{lem:yshuffleb}, to say
			\begin{align*}
			&= (-1)^s [a_0 \mid \{y\}^{s+1} \shuffle (b_{s+2}\cdots b_n) - b_{s+2}(\{y\}^{s+1}\shuffle b_{s+3}\cdots b_n) \dblslash x \mid a_{n+1}] \\
			& \modsh (-1)^{s+1} [a_0 \mid b_{s+2}, (\{y\}^{s+1} \shuffle (b_{s+3} \cdots b_n) \dblslash x \mid a_{n+1}] \, ,
			\end{align*} since we work modulo products in \( \Hbl_n(E) \).  This proves the result.
		\end{proof}
		
		\begin{Obs}\label{obs:nminus1vars}
			We can apply the above lemma to each term \( A(S, i, I) \) in \( B(S,i) \) from \autoref{lem:sumtoB}, with \( y = a_i \) and \( s \) as large as possible.  Firstly, each term in \( B(S,i) \) has at least one variable \( a_j \) replaced with \( a_i \), so we have reduced the number of variables per term to \( n - 1 \), at most.  Then  by shuffling out \( a_i \), we can guarantee that it does not appear in the first position.  This means \( B(s,i) \) is a sum of hyperlogs \( [a_0 \mid \argdash \dblslash x \mid a_{n+1}] \), where each contains \( \leq n-1 \) variables, and such that \( a_i \) never appears in the first position.
		\end{Obs}

		We now want to split up each integral at the point \( a_i \) to further reduce the number of variables in each term to \( \leq n - 2 \).  We have the following relation in \( \Hbl_n(E) \).
		
		\begin{Lem}\label{lem:integralbounddifference}
			In \( \Hbl_n(E) \), the following relation holds for any generic \( c \), specifically \( c \) such that \( a_1 \neq c \), and \( c \neq x \),
			\[
			[a_0 \mid a_1, \ldots, a_n \dblslash x \mid a_{n+1}] = [c \mid  a_1, \ldots, a_n \dblslash x \mid a_{n+1}] - [c \mid  a_1, \ldots, a_n \dblslash x \mid a_0] \, .
			\]
			
			\begin{proof}
				This follows from the composition of paths property from \autoref{sec:chenmpl}.  For two paths \( \alpha, \beta \), where \( \alpha(1) = \beta(0) \), it states that
				\[
				\int_{\alpha\beta} \eta_1 \circ \cdots \circ \eta_n = \sum_{i=0}^n \int_\alpha \eta_1 \circ \cdots \circ \eta_i \int_\beta \eta_{i+1} \circ \cdots \circ \eta_n \, .
				\]
				Recall now that the empty integral \( \int_\alpha \empty = 1 \).  If we work modulo products, only the integrals coming from \( i = 0 \), and \( i = n \) survive.  Therefore we have
				\[
				\int_{\alpha\beta} \eta_1 \circ \cdots \circ \eta_n \modsh \int_\alpha \eta_1 \circ \cdots \circ \eta_n + \int_\beta \eta_1 \circ \cdots \circ \eta_n \, ,
				\] where \( \modsh \) emphasises that we work modulo products. \medskip
				
				By choosing such a generic \( c \), all the integrals involved will converge.  Then take \( \alpha \) to be a path \( c \to a_0 \) and \( \beta \) a path \( a_0 \to a_{n+1} \).  Choosing \( \eta_i = \omega(a_i,x) \) to be our differential form of interest used in \autoref{def:generalisedhyperlog}, we obtain
				\[
				H(c \mid  a_1, \ldots, a_n \dblslash x \mid a_0) + H(a_0 \mid a_1, \ldots, a_n \dblslash x \mid a_{n+1}) \modsh H(c \mid  a_1, \ldots, a_n \dblslash x \mid a_{n+1}) \, ,
				\] modulo products.  Now view this in \( \Hbl_n(E) \), and rearrange to obtain the above identity.
			\end{proof}
		\end{Lem}
	\end{Lem}
	
	According to \autoref{obs:nminus1vars}, \( a_i \) does not appear in the first slot of any term in \( B(S, i) \), so we may use the above lemma to rewrite each term of this sum as 
	\[
	[a_0 \mid \argdash \dblslash x \mid a_{n+1}] = [a_i \mid \argdash \dblslash x \mid a_{n+1}] - [a_i \mid \argdash \dblslash x \mid a_0]
	\]
	This breaks the single term with \( \leq n-1 \) variables into two terms each with \( \leq n-2 \) variables, where the variable \( a_0 \) is avoided in favour of \( a_i \) in the first summand, and the variable \( a_{n+1} \) is avoided in favour of \( a_i \) in the second summand.
	
	\begin{Eg}
		We have
		\begin{align*}
		& A([a_0 \mid a_1, a_2, a_3, a_4, a_5 \dblslash x \mid a_6], 2, \Set{1,2}) \\ 
		{}={}& [a_0 \mid a_2, a_2, a_3, a_4, a_5 \dblslash x \mid a_6] \\ 
		{}={}& (-1)^2 [a_0 \mid a_3, (a_2^2 \shuffle a_4a_5) \dblslash x \mid a_6] \\[1ex]
		{}={}& \begin{aligned}[t]
		&[a_0 \mid a_3, a_2, a_2, a_4, a_5 \dblslash x \mid a_6] + [a_0 \mid a_3, a_2, a_4, a_2, a_5 \dblslash x \mid a_6] + {} \\ &{} + [a_0 \mid a_3, a_4, a_2, a_2, a_5 \dblslash x \mid a_6] + [a_0 \mid a_3, a_2, a_4, a_5, a_2 \dblslash x \mid a_6] + {} \\ &{} + [a_0 \mid a_3, a_4, a_2, a_5, a_2 \dblslash x \mid a_6] + [a_0 \mid a_3, a_4, a_5, a_2, a_2 \dblslash x \mid a_6]
		\end{aligned}
		\end{align*}
		Then each term can be split as indicated above.  So the first term would become
		\[
		[a_2 \mid a_3, a_2, a_2, a_4, a_5 \dblslash x \mid a_6] - [a_2 \mid a_3, a_2, a_2, a_4, a_5 \dblslash x \mid a_0] \, ,
		\]
		and similarly for the rest.
	\end{Eg}
	
	Finally, this proves the following proposition
	
	\begin{Prop}\label{prop:switchx}
		We may express
		\begin{multline*}
		[a_0 \mid a_1, \ldots, a_n \dblslash x \mid a_{n+1}] + [a_0 \mid a_1, \ldots, a_{i-1}, x, a_{i+1}, \ldots, a_n \dblslash a_i \mid a_{n+1}] \\= B([a_0 \mid a_1, \ldots, a_n \dblslash x \mid a_{n+1}], i) \, ,
		\end{multline*}
		as an explicit sum of hyperlogs in \( \leq n-2 \) variables.
	\end{Prop}

	This completes the second step of the method, as outlined in \autoref{sec:overview}.
	
	\subsection{\texorpdfstring{Operator \( D \)}{Operator D}}\label{sec:operatorD}

	We continue now with the third step of the method, a outlined in \autoref{sec:overview}.
	
	We introduce the operator \( D \) which packages up the result in \autoref{prop:switchx}.  We use this to build up a reduction for a transposition \( a_i \leftrightarrow a_j \) to a sum in \( \leq n-2 \) variables, and then such a reduction for any permutation of the \( a_i \).
	
	\begin{Def}\label{def:operatorD}
		The explicit sum in \( \leq n - 2 \) variables produced in \autoref{prop:switchx}, above, will be denoted 
		\[
		D([a_0 \mid a_1, \ldots, a_n \dblslash x \mid a_{n+1}], i) \, .
		\]
	\end{Def}
	
	\begin{Prop}\label{prop:transposition}
		A transposition of two variables can be expressed in terms of three \( D \) operations as follows
		\begin{align*}
		& [a_0 \mid \ldots, a_i, \ldots, a_j, \ldots \dblslash x \mid a_{n+1}] + [a_0 \mid \ldots, a_j, \ldots, a_i, \ldots \dblslash x \mid a_{n+1}] \\
		&= D([a_0 \mid \ldots, a_i, \ldots, a_j, \ldots \dblslash x \mid a_{n+1}], i)- D([a_0 \mid \ldots, x, \ldots, a_j, \ldots \dblslash a_i \mid a_{n+1}], j) + {} \\
		& \quad \quad {} + D([a_0 \mid \ldots, x, \ldots, a_i, \ldots \dblslash a_j \mid a_{n+1}], i) \, .
		\end{align*}
		
		\begin{proof}
			This is just a case of writing out the result of the three applications of \( D \).  Namely
			\begin{align*}
			& D([a_0 \mid \ldots, a_i, \ldots, a_j, \ldots \dblslash x \mid a_{n+1}], i) + {} \\
			& \hspace{1em} {} - D([a_0 \mid \ldots, x, \ldots, a_j, \ldots \dblslash a_i \mid a_{n+1}], j) + {} \\
			& \hspace{1em} {} + D([a_0 \mid \ldots, x, \ldots, a_i, \ldots \dblslash a_j \mid a_{n+1}], i) \\[1ex]
			& = ([a_0 \mid \ldots, a_i, \ldots, a_j, \ldots \dblslash x \mid a_{n+1}] + [a_0 \mid \ldots, x, \ldots, a_j, \ldots \dblslash a_i \mid a_{n+1}]) + {} \\
			& \hspace{2em} {} - ([a_0 \mid \ldots, x, \ldots, a_j, \ldots \dblslash a_i \mid a_{n+1}] + [a_0 \mid \ldots, x, \ldots, a_i, \ldots \dblslash a_j \mid a_{n+1}]) + {} \\
			& \hspace{2em} {} + ([a_0 \mid \ldots, x, \ldots, a_i, \ldots \dblslash a_j \mid a_{n+1}] + [a_0 \mid \ldots, a_j, \ldots, a_i, \ldots \dblslash x \mid a_{n+1}]) \\[1ex]
			& = [a_0 \mid \ldots, a_i, \ldots, a_j, \ldots \dblslash x \mid a_{n+1}] + [a_0 \mid \ldots, a_j, \ldots, a_i, \ldots \dblslash x \mid a_{n+1}] \, . \qedhere
			\end{align*}
		\end{proof}
	\end{Prop}
	
	\begin{Cor}\label{cor:transpotition}
		The combination
		\[
		[a_0 \mid \ldots, a_i, \ldots, a_j, \ldots \dblslash x \mid a_{n+1}] + [a_0 \mid \ldots, a_j, \ldots, a_i, \ldots \dblslash x \mid a_{n+1}]
		\]
		is an explicit sum of hyperlogs in \( \leq n-2 \) variables.  More generally, for any permutation \( \sigma \in S_n \),
		\[
		[a_0 \mid \sigma \cdot (a_1, \ldots, a_n) \dblslash x \mid a_{n+1}] - \sgn(\sigma) [a_0 \mid a_1, \ldots, a_n \dblslash x \mid a_{n+1}]
		\]
		is an explicit sum of hyperlogs in \( \leq n-2 \) variables.  Here \( \sigma \) acts by permuting the indices \( \sigma \cdot (a_1, \ldots, a_n) = (a_{\sigma(1)}, \ldots, a_{\sigma(n)}) \).
		
		\begin{proof}
			The first claim is immediate because we defined \( D \) to be such an explicit sum of hyperlogs in \( \leq  n-2 \) variables.  By decomposing a permutation as a product of transpositions, we get by induction the result for any permutation \( \sigma \), as follows.  Suppose the claim holds for \( \sigma \).  Let \( \tau \in S_n \) be a transposition.  Then \( \sgn(\tau) = -1 \), and for \( \tau \sigma \) we have
			\begin{align*}
			& [a_0 \mid \tau \sigma\cdot  (a_1, \ldots, a_n) \dblslash x \mid a_{n+1}] - \sgn(\tau\sigma)[a_0 \mid a_1, \ldots, a_n \dblslash x \mid a_{n+1}] \\
			& = [a_0 \mid \tau \sigma\cdot  (a_1, \ldots, a_n) \dblslash x \mid a_{n+1}] - \sgn(\tau)[a_0 \mid \sigma \cdot (a_1, \ldots, a_n) \dblslash x \mid a_{n+1}]) + {} \\
			& \hspace{2em} {} - ([a_0 \mid \sigma\cdot  (a_1, \ldots, a_n) \dblslash x \mid a_{n+1}] - \sgn(\sigma)[a_0 \mid a_1, \ldots, a_n \dblslash x \mid a_{n+1}]) \, .
			\end{align*}
			Both of these summands is a sum in \( \leq n-2 \) variables, so the claim holds.
		\end{proof}	
	\end{Cor}
	
	This completes the third step of the method, as outlined in \autoref{sec:overview}.
	
	\subsection{Reducing a single hyperlog}
	
	We continue with the fourth step of the method, as outlined in \autoref{sec:overview}.  This produces a proof of concept reduction of the hyperlogarithm \( [a_0 \mid a_1, \ldots, a_n \dblslash x \mid a_{n+1}] \) to an explicit sum in \( \leq n - 2 \) variables. \medskip
	
	The above considerations allow us to write a combination of two hyperlogarithms, whose arguments differ by a permutation, as a sum of hyperlogarithms in \( \leq n-2 \) variables.  By carefully considering these combinations, it is possible to write a single hyperlogarithm \( [a_0 \mid a_1, \ldots, a_n \dblslash x \mid a_{n+1}] \) in such a manner.  We do this by considering the sum of the signs of the permutations in the set \( S(2,n-2) \) of \( (2,n-2) \)-shuffles.
	
	\begin{Prop}\label{prop:signof2shufflen-2}
		Let \( S(2,n-2) \) be the set of \( (2, n-2) \)-shuffles, which can be considered the as the words from \( a_1 a_2 \shuffle a_3\cdots a_n \).  Then
		\[
		\sum_{\sigma \in S(2,n-2)} \sgn(\sigma) = \floor{\frac{n}{2}}
		\]
		
		\begin{proof}
			Observe that every permutation in the set of \( (2,n-2)\)-shuffles, is uniquely determined by the position of \( 1 \) and the position of \( 2 \).  Moreover, 2 must appear after 1 since this is the ordering in the original multiplicand.  So each term is described by
			\[
			S_{i,j}^n \coloneqq \Set{a_3, a_4, \ldots, \smash{\underbrace{a_1}_{\mathclap{\text{position \( i \) }}}}, \ldots, \smash{\underbrace{a_2}_{\mathclap{\text{position \( j \) }}}}, \ldots }  \mathllap{\phantom{\underbrace{a_1}_{\text{position}}}} \, ,
			\]
			for some \( 1 \leq i < j \leq n \).  For example
			\[
				S_{2,5}^7 = \Set{ a_3, \smash{\underbrace{a_1}_{\mathclap{\text{position 2 }}}}, a_4, a_5, \smash{\underbrace{a_2}_{\mathclap{\text{position 5}}}}, a_6, a_7 } \mathllap{\phantom{\underbrace{a_1}_{\text{position}}}} \, .
			\]
			
			What is \( \sgn(S_{i,j}^n) \)?  To put \( 2 \) into position \( j \) from its original position 2 requires \( j - 2 \) swaps.  Then to put \( 1 \) into position \( i \) from its original position 1 requires a further \( i - 1 \) swaps.  So the total number of swaps is \( i + j - 3 \).  We find
			\[
			\sgn(S_{i,j}^n) = \begin{cases}
			-1 & \text{if \( i+j \) is even} \\
			1 & \text{if \( i+j \) is odd.}
			\end{cases}
			\]
			
			If we sum all the signs, we obtain
			\begin{align*}
			\sum_{\sigma \in S} \sgn(\sigma) = \sum_{i = 1}^{n} \sum_{j = i+1}^{n} \sgn(S_{i,j}^n) \, .
			\end{align*}
			Observe that in the inner sum, consecutive terms have opposite signs.  At term \( j \), the value \( i+j \) has one parity, which means at term \( j+1 \), the parity of \( i+(j+1) \) is different.  If there are an even number of terms in the inner sum, then they all cancel in pairs to 0.  Otherwise the terms after the first cancel, and we are left with \( \sgn(S_{i,i+1}^n) = 1 \) since \( i + (i+1) \) is odd.  The number of terms in the inner sum is \( n - (i+1) + 1 = n - i \), so this is odd if and only if \( n \) and \( i \) have different parities.
			
			If \( n = 2m \) is even, we obtain:
			\[
			\sum_{i = 1}^{n} \sum_{j = i+1}^{n} \sgn(S_{i,j}^n) = \sum_{\substack{i = 1 \\ \text{\( i \) odd}}}^{2m} 1 = m = \floor{n/2} \, .
			\]
			
			And if \( n = 2m+1 \) is odd, we obtain:
			\[
			\sum_{i = 1}^{n} \sum_{j = i+1}^{n} \sgn(S_{i,j}^n) = \sum_{\substack{i = 1 \\ \text{\( i \) even}}}^{2m+1} 1 = m = \floor{n/2} \, .
			\]
			
			This proves the result.
		\end{proof}
	\end{Prop}
	
	There is enough here now to prove that a depth \( n \) hyperlog in \( n \geq 3 \) variables can be reduced to a sum of hyperlogs in \( \leq n - 2 \) variables.  We obtain the following.
	
	\begin{Thm}\label{thm:reductionofhyperlog}
		For \( n \geq 3 \), the hyperlog
		\[
		[a_0 \mid a_1, \ldots, a_n \dblslash x \mid a_{n+1}]
		\]
		can be expressed as a sum of hyperlogs in \( \leq n-2 \) variables.
		
		\begin{proof}
			For each \( \sigma \in S(2,n-2) \), we have that
			\[
			[a_0 \mid \sigma \cdot (a_1, \ldots, a_n) \dblslash x \mid a_{n+1}] - \sgn(\sigma)[a_0 \mid a_1, \ldots, a_n \dblslash x \mid a_{n+1}] 
			\] can be expressed as a sum in \( \leq n-2 \) variables.  Now sum over all such \( \sigma \in S(2,n-2) \).  The left hand terms sum over \( a_1a_2 \shuffle a_3\cdots a_n \).  The right hand terms are all the same, so sum to the multiple \( \sum_{\sigma \in S} \sgn(\sigma) = \floor{n/2} \).  Therefore we get that
			\[
			[a_0 \mid a_1 a_2 \shuffle a_3 \cdots a_n \dblslash x \mid a_{n+1}] - \floor{n/2} [a_0 \mid a_1, \ldots, a_n \dblslash x \mid a_{n+1}]
			\] is a sum of hyperlogarithms in \( \leq n-2 \) variables. \medskip
			
			As we work modulo products the first term here is actually 0 if \( n \geq 3 \), so this shows that
			\[
				[a_0 \mid a_1, \ldots, a_n \dblslash x \mid a_{n+1}]
			\]
			is a sum of hyperlogs in \( \leq n-2 \) variables.
		\end{proof}
	\end{Thm}
	
	This completes the fourth step of the reduction method, as outlined in \autoref{sec:overview}.  It should be noted, however, that the reduction in this theorem is really only intended as a proof-of-concept.  The number of terms generated by relating \( [a_0 \mid \sigma \cdot (a_1, \ldots, a_n) \dblslash x \mid a_{n+1}] \) back to \( [a_0 \mid a_1, \ldots, a_n \dblslash x \mid a_{n+1}] \), for every permutation \( \sigma \in S(2,n-2) \) is excessive; DAn goes on to describes a more efficient approach.
	
	\subsection{More efficient approach for the reduction}\label{sec:structured}
	
	\subsubsection{\texorpdfstring{Efficient approach for any \( n \)}{Efficient approach for any n}}
	\label{sec:structured_alln}

	In the fifth and sixth steps of the reduction method, as outlined in \autoref{sec:overview}, Dan proceeds in a more efficient way to combine terms in \( a_1 a_2 \shuffle a_3\cdots a_n \) using transpositions as far as possible.  Some of the ideas involved are already present in the proof of \autoref{prop:signof2shufflen-2}.
		
	\begin{Def}
		Define the symbol \( A_{i,j}^n \) to be the following
		\[
		A_{i,j}^n \coloneqq [a_0 \mid a_3, a_4, \ldots, \underbrace{a_1}_{\mathclap{\text{position \( i \) }}}, \ldots, \underbrace{a_2}_{\mathclap{\text{position \( j \)}}}, \ldots \dblslash x \mid a_{n+1}] \, ,
		\]
		where position \( i \) is filled with \( a_1 \), and position \( j \) is filled with \( a_2 \).  The remaining positions are filled with \( a_3, \ldots, a_n \) in this order.  (Notice the similarity to \( S_{i,j}^n \) from the proof of \autoref{prop:signof2shufflen-2}, essentially \( A_{i,j}^n = [a_0 \mid S_{i,j}^n \dblslash x \mid a_{n+1}] \).)
	\end{Def}
	
	In the original article, Dan uses the notation \( A_{i,j} \), leaving the dependence on \( n \) implicit only.  For clarity here, and later, I write \( A_{i,j}^n \) in order to make the dependence on \( n \) as explicit as possible.  Similarly we will write \( R^n \) where Dan later writes \( R \), and \( c^n \) where Dan writes \( c \).
	
	\begin{Eg}
		With \( n = 7 \), and \( i = 2, j = 5 \), we have
		\[
		A_{2,5}^7 = [a_0 \mid a_3, \underbrace{a_1}_{\mathclap{\text{position 2 }}}, a_4, a_5, \underbrace{a_2}_{\mathclap{\text{position 5}}}, a_6, a_7 \dblslash x \mid a_8] \, .
		\]
	\end{Eg}
	
	\begin{Lem}\label{lem:ai-1jij}
		Consider now the expression \( A_{i-1,j}^n + A_{i,j}^n \).  This can be expressed as an explicit sum of hyperlogs in \( \leq n-2 \) variables.  Similarly \( A_{i,j}^n + A_{i,j+1}^n \) can be expressed as an explicit sum of hyperlogs in \( \leq n-2 \) variables.
		
		\begin{proof}
			Going from \( A_{i-1,j}^n \) to \( A_{i,j}^n \) requires a single transposition swapping positions \( i-1 \) and \( i \).  Similarly, going from \( A_{i,j}^n \) to \( A_{i,j+1}^n \) requires a single transposition swapping positions \( j \) and \( j+1 \).  So by \autoref{cor:transpotition} the result follows.
		\end{proof}
	\end{Lem}
	
	\begin{Def}
		Write \( R^n(i-1, j \mid i,j) \) for the relation in \autoref{lem:ai-1jij}, above, expressing \( A^n_{i-1,j} + A^n_{i,j} \) as a sum in \( \leq n-2 \) variables.  Also write \( R^n(i,j \mid i,j+1) \) for the relation expressing \( A^n_{i,j} + A^n_{i,j+1} \) as a sum in \( \leq n-2 \) variables.
	\end{Def}
	
	At this point Dan considers some remarkable combination of \( R^n \)'s with certain coefficients \( c^n(\argdash) \), and claims (without proof) that from this one deduces a reduction formula.  I want to develop this sum in a step-by-step manner, and fill in the missing proofs. \medskip
	
	Consider the shuffle product \( a_1 a_2 \shuffle a_3 \cdots a_n \).  Each term of this is a word of length \( n \) where \( a_1 \) and \( a_2 \) occupy certain positions, and the string \( a_3, a_4, \ldots, a_n \) covers the remaining positions in order.  Therefore each term of the shuffle product is \( A_{i,j}^n \) for some \( i,j \).  Moreover, since 1 always occurs at a position before 2, we have \( i < j \).  Otherwise there is complete freedom to choose \( i \) and \( j \) between 1 and \( n \).  Therefore
	\[
	[a_0 \mid a_1 a_2 \shuffle a_3 a_4 \cdots a_n \dblslash x \mid a_{n+1}] = \sum_{1 \leq i < j \leq n} A_{i,j}^n \, .
	\]
	
	Now sum in the following order to get
	\[
	\sum_{1 \leq i < j \leq n} A_{i,j}^n = \sum_{j = 2}^{n} \sum_{i = 1}^{j-1} A_{i,j}^n \, .
	\]
	When \( j \) is odd, the inner sum \( \sum_{i=1}^{j-1} A_{i,j}^n \) can be written as
	\[
	\sum_{\substack{i=1 \\ \text{\( i \) even}}}^{j-1} (A_{i-1,j}^n + A_{i,j}^n) = \sum_{\substack{i = 1 \\ \text{\( i \) even}}}^{j-1} R^n(i-1,j \mid i, j) \, .
	\]
	
	When \( j \) is even, the inner sum \( \sum_{i=1}^{j-1} A_{i,j}^n \) can be written as
	\[
	A_{1,j} + \sum_{\substack{i = 3 \\ \text{\( i \) odd}}}^{j-1} (A_{i-1,j}^n + A_{i,j}^n) = A_{1,j}^n + \sum_{\substack{i = 3 \\ \text{\( i \) odd}}}^{j-1} R(i-1, j \mid i, j)^n \, .
	\]
	
	For convenience we want to sum over the full range \( i = 2, \ldots, j-1 \), including all even and odd indices, but this will introduce spurious extra terms.  To fix this, introduce coefficients \( c^n(i-1, j, \mid i,j) \) corresponding to the relation \( R^n(i-1,j \mid i,j) \).  When \( j \) is odd we need the even terms to live, so impose \( c^n(i-1, j \mid i, j) = 1 \) when \( i \) even and \( j \) odd, and \( c^n(i-1,j \mid i,j) = 0 \) when \( i \) odd and \( j \) odd.  When \( j \) is even, we need the odd terms to live, so impose \( c^n(i-1,j \mid i,j) = 1 \) when \( i \) odd and \( j \) even, and \( c^n(i-1,j \mid i,j) = 0 \) when \( i \) even and \( j \) even.  This can be summarised by saying
	\[
	c^n(i -1, j \mid i, j) = \begin{cases} 1 & \text{if \( i - j \) odd} \\ 0 & \text{otherwise,} \end{cases}
	\]
	in accordance with Dan's definition.  (We write \( c^n \) rather than just \( c \) because a later extension of \( c \) will explicitly depend on \( n \).) \medskip
	
	Plugging these into the sum above, we find that
	\begin{align*}
	&  [a_0 \mid a_1 a_2 \shuffle a_3 a_4 \cdots a_n \dblslash x \mid a_{n+1}] = \\
	&  \sum_{2 \leq i < j \leq n} c(i-1,j \mid i,j) R^n(i-1,j \mid i,j) + \sum_{\substack{j = 2 \\ \text{\( j \) even}}}^{n} A_{1,j}^n \, .
	\end{align*}
	
	Now consider the leftover terms \( \sum_{\text{\( j=2 \), \( j \) even}}^n A_{1,j}^n \).  Observe that we can write the following equality
	\begin{align*}
	A_{1,j}^n &= (A_{1,j}^n + A_{1,j-1}^n) - (A_{1,j-1}^n + A_{1,j-2}^n) + A_{1,j-2}^n \\
	&= R^n(1,j-1 \mid 1,j) - R^n(1,j-2 \mid 1, j-1) + A_{1,j-2}^n \\
	\shortintertext{ and by iterating,}
	&= R^n(1,j-1 \mid 1,j) - R^n(1,j-2 \mid 1, j-1) + {} \\
	& \hspace{2em} {} + R^n(1,j-3 \mid 1,j-2) - R^n(1,j-4 \mid 1, j-3) + A_{1,j-4}^n \, .
	\end{align*}
	This means we can eliminate \( A_{1,j}^n \) in favour of \( A_{1,j-2}^n \) and some relations \( R^n \).  By iterating this, we can push this as far as we want, as follows.
	
	\begin{Lem} \label{lem:shuffletoA1n} 
		For any even \( 2 \leq m \leq j-2 \), we have
		\[
		A_{1,j}^n = \sum_{\substack{k=m \\ \text{\( k \) even}}}^{j-2} (R^n(1,k+1\mid1,k+2) - R^n(1,k\mid1,k+1)) + A_{1,m}^n \, .
		\]
		\begin{proof}
			Certainly the result is true for \( m = j-2 \), by the observation preceding this lemma.
			
			Now suppose the result holds for \( m \).  Then for \( m -2 \) we have
			\begin{align*}
			& \sum_{\substack{k = m-2 \\ \text{\( m \) even}}}^{j-2} (R^n(1,k+1\mid1,k+2) - R^n(1,k\mid1,k+1)) \\
			& = \sum_{\substack{k = m \\ \text{\( m \) even}}}^{j-2} (R^n(1,k+1\mid1,k+2) - R^n(1,k\mid1,k+1)) + {} \\
			& \hspace{2em} + (R^n(1,m-1\mid1,m) - R^n(1,m-2\mid1,m-1)) \, ,
			\end{align*}
			which by the induction assumption equals
			\begin{align*}
			& = A_{1,j}^n - A_{1,m}^n + (R^n(1,m-1\mid1,m) - R^n(1,m-2\mid1,m-1)) \\
			&= A_{1,j}^n - A_{1,m}^n + ((A_{1,m-1}^n + A_{1,m}^n) - (A_{1,m-2}^n + A_{1,m-1}^n)) \\
			&= A_{1,j}^n - A_{1,m-2}^n \, .
			\end{align*}
			So the result holds for \( m - 2 \) also.
		\end{proof}
	\end{Lem}
	
	In particular, for \( m = 2 \), we obtain
	\[
	A_{1,j}^n = \sum_{\substack{k = 2 \\ \text{\( k \) even}}}^{j-2} (R^n(1,k+1\mid1,k+2) - R^n(1,k\mid1,k+1)) + A_{1,2}^n \, ,
	\]
	and we can use this to establish the following result.
	
	\begin{Lem}\label{lem:sumofleftover}
		The sum of the leftover terms is given by
		\begin{align*}
		\sum_{\substack{j = 2 \\ \text{\( j \) even}}}^{n} A_{1,j}^n &= \floor{n/2} A_{1,2}^n + {} \\
		& + \sum_{\substack{j = 2 \\ \text{\( j \) even }}}^{n-2} (\floor{n/2} - j/2)(R^n(1,j+1\mid1,j+2) - R^n(1,j\mid1,j+1)) \, .
		\end{align*}
		
		\begin{proof}
			We may use the above result to give an expression for \( A_{1,j}^n \), and sum as follows
			\begin{align*}
			\sum_{\substack{j = 2 \\ \text{\( j \) even }}}^n A_{i,j} &= \sum_{\substack{j = 2 \\ \text{\( j \) even}}}^n \bigg( A_{1,2}^n +  \sum_{\substack{k = 2 \\ \text{\( k \) even}}}^{j-2} (R^n(1,k+1\mid1,k+2) - R^n(1,k\mid1,k-1))\bigg) \\
			&= \floor{n/2} A_{1,2}^n + \sum_{\substack{j = 2 \\ \text{\( j \) even}}}^n \sum_{\substack{k = 2 \\ \text{\( k \) even}}}^{j-2} (R^n(1,k+1\mid1,k+2) - R^n(1,k\mid1,k+1)) \, .
			\end{align*}
			Now swap the order of summation, to obtain
			\begin{align*}
			&= \floor{n/2} A_{1,2}^n + \sum_{\substack{k = 2 \\ \text{\( k \) even}}}^{n-2} \sum_{\substack{j = k+2 \\ \text{\( j \) even}}}^{n} (R^n(1,k+1\mid1,k+2) - R^n(1,k\mid1,k+1)) \, .
			\end{align*}
			Since the summand does not depend on the index of the inner sum, we just obtain a multiple of it based on the number of terms summed.  In this case we have \( \floor{n/2} - k/2 \) terms, so we get
			\begin{align*}
			&= \floor{n/2} A_{1,2}^n + \sum_{\substack{k = 2 \\ \text{\( k \) even}}}^{n-2} (\floor{n/2} - k/2)(R^n(1,k+1\mid1,k+2) - R^n(1,k\mid1,k+1)) \, .
			\end{align*}
			Finally, change the summation index from \( k \) to \( j \) to obtain the result.
		\end{proof}
	\end{Lem}
	
	Here Dan also wishes to sum over the full range \( j = 2, \ldots, n-2 \).  This is more straightforward to do, since we can break the sum up and reindex it as follows.
	\begin{align*}
	& \sum_{\substack{j = 2 \\ \text{\( j \) even}}}^{n-2} (\floor{n/2} - j/2)(R^n(1,j+1\mid1,j+2) - R^n(1,j\mid1,j+1)) \\
	&=  \sum_{\substack{j = 2 \\ \text{\( j \) even}}}^{n-2} (\floor{n/2} - j/2)R^n(1,j+1\mid1,j+2) - \sum_{\substack{j = 2 \\ \text{\( j \) even}}}^{n-2} (\floor{n/2} - j/2)R^n(1,j\mid1,j+1) \, .
	\end{align*}
	Now put \( j \mapsto j-1 \) in the first sum.  The range changes to \( j = 3 \) to \( n-1 \), \( j \) odd, giving
	\begin{align*}
	&=  \sum_{\substack{j = 3 \\ \text{\( j \) odd}}}^{n-1} (\floor{n/2} - (j-1)/2)R^n(1,j\mid1,j+1) - \sum_{\substack{j = 2 \\ \text{\( j \) even}}}^{n-2} (\floor{n/2} - j/2)R^n(1,j\mid1,j+1) \, .
	\end{align*}
	Observe that when \( j \) is odd, \( (j-1)/2 = \floor{j/2} \).  And when \( j \) is even, \( j/2 = \floor{j/2} \).  Both sums can be combined to give
	\begin{align*}
	&= -\sum_{j=2}^{n-1} (-1)^j (\floor{n/2} - \floor{j/2}) R^n(1,j \mid 1,j+1) \, .
	\end{align*}
	We can then set
	\[
	c^n(1, j \mid 1, j+1) = (-1)^j (\floor{n/2} - \floor{j/2}) \, ,
	\]
	in accordance with Dan.  (Writing \( c^n \) rather than just \( c \) to emphasis the dependence on \( n \).)
	
	Overall, we have
	\begin{align*}
	& [a_0 \mid a_1 a_2 \shuffle a_3 a_4 \cdots a_n \dblslash x \mid a_{n+1}] \\
	&= \sum_{2 \leq i < j \leq n} c^n(i-1,j \mid i,j) R(i-1,j \mid i,j) + \floor{n/2} A_{1,2}^n + {} \\
	& \hspace{2em} {} - \sum_{j=2}^{n-1} c^n(1,j\mid1,j+1) R^n(1,j \mid 1,j+1) \, .
	\end{align*}
	
	By rearranging this, we therefore obtain the following theorem
	\begin{Thm}\label{thm:reduction}
		The following equality holds
		\begin{align*}
		& \hspace{-5em} \floor{n/2} [a_0 \mid a_1, \ldots, a_n \dblslash x \mid a_{n+1}] = \\
		& - \sum_{2\leq i < j \leq n} c^n(i-1,j \mid i,j) R(i-1,j \mid i,j) + {} \\
		& + \sum_{2 \leq j \leq n-1} c^n(1,j \mid 1,j+1) R^n(1,j \mid 1,j+1) + {} \\
		& + [a_0 \mid a_1 a_2 \shuffle a_3 a_4 \cdots a_n \dblslash x \mid a_{n+1}] \, .
		\end{align*}
		And in particular for \( n \geq 3 \),
		\[
		[a_0 \mid a_1, \ldots, a_n \dblslash x \mid a_{n+1}]
		\]
		is explicitly given as a sum of hyperlogs in \( \leq n-2 \) variables, modulo products.
	\end{Thm}
	
	This completes the sixth and final step of the reduction, as outlined in \autoref{sec:overview}.  As a final punchline, we have the following corollary.
	
	\begin{Cor}
		By setting \( x = \infty \), we get an expression for \( [a_0 \mid a_1, \ldots, a_n \mid a_{n+1}] \) as a sum of hyperlogs in \( \leq n-2 \) variables.
	\end{Cor}
	
	\subsubsection{\texorpdfstring{Efficient approach for \( n \) odd}{Efficient approach for n odd}}
	\label{sec:structured_oddn}
	
	Dan remarks that when \( n \) is odd, one can obtain an even simpler expression for this reduction.  This is done as follows.
	
	\begin{Lem}
		Let \( S(1,n-1) \) be the set of \( (1, n-1) \) shuffles, which can be identified with the terms in the shuffle product \( a_1 \shuffle a_2 a_3 \cdots a_n  \).  Then
		\[
		\sum_{\sigma \in S(1,n-1)} \sgn(\sigma) = 1 \, .
		\]
		
		\begin{proof}
			Each term in \( S \) is completely determined by the position of \( a_1 \).  If \( a_1 \) is in position \( j \), then it takes \( j-1 \) swaps to put the permutation into the original order.  Hence
			\[
			\sum_{\sigma \in S(1,n-1)} \sgn(\sigma) = \sum_{j = 1}^n (-1)^{j-1} \, .
			\]
			Since \( n = 2k + 1 \) is odd, we can break this up into
			\begin{align*}
			{} = \sum_{\substack{j = 1 \\ \text{\( j \) even}}} (-1) + \sum_{\substack{j = 1 \\ \text{\( j \) odd}}} 1
			 = j \cdot (-1) + (j+1) \cdot 1 = 1 \, .
			\end{align*}
			This completes the proof.
		\end{proof}
	\end{Lem}
	
	Now write \( [a_0 \mid \sigma \cdot (a_1, \ldots, a_n) \dblslash x \mid a_{n+1}] - \sgn(\sigma)[a_0 \mid a_1, \ldots, a_n \dblslash x \mid a_{n+1}] \) as a sum in \( \leq n - 2 \) variables, and sum over all \( \sigma \in S(1,n-1) \).  One obtains an analogue of \autoref{thm:reductionofhyperlog}.  Dan now combines terms differing by transpositions, to give the following explicit reduction.
	
	\begin{Def}
		Write
		\[
		A_{i}^n \coloneqq [a_0 \mid a_2, \ldots, a_{i}, \underbrace{a_1}_{\mathclap{\text{position \( i \) }}}, a_{i+1}, \ldots, a_n \dblslash x \mid a_{n+1}] \, .
		\]
		where position \( i \) is filled with \( a_1 \), and the remaining positions are filled with \( a_2, \ldots, a_n \) in this order.
	\end{Def}
	
	\begin{Lem}
		Consider the expression \( A_{i}^n + A_{i+1}^n \).  This can be expressed as an explicit sum of hyperlogs in \( \leq n - 2 \) variables.
		
		\begin{proof}
			Observe that \( A_{i}^n + A_{i+1}^n \) is a transposition, obtained by swapping positions \( i \) and \( i+1 \).  Since it is a transposition, it can be expressed as a sum in \( \leq n-2 \) variables using the \autoref{cor:transpotition} and the \( D \) operator.  So the result holds.
		\end{proof}
	\end{Lem}
	
	\begin{Def}
		Denote by \( R_{i}^n \) the relation expressing \( A_i^n + A_{i+1}^n \) as a sum in \( \leq n - 2 \) variables.
	\end{Def}
	
	Since each term in \( [a_0 \mid a_1 \shuffle a_2 a_3 \cdots a_n \dblslash x \mid a_{n+1}] \) is determined by the position of \( a_1 \), we obtain
	\[
	[a_0 \mid a_1 \shuffle a_2 a_3 \cdots a_n \dblslash x \mid a_{n+1}] = \sum_{i=1}^n A_{i}^n \, .
	\]
	Since \( n \) is odd, we may write this as
	\begin{align*}
	{} = A_{1}^n + \sum_{\substack{i = 2 \\ \text{\( i \) even}}}^{n} A_{i}^n + A_{i+1}^n 
	= A_{1}^n + \sum_{\substack{i = 2 \\ \text{\( i \) even}}}^{n} R_{i}^n {} 
	= A_{1}^n + \sum_{j = 1}^{\floor{n/2}} R_{2j}^n \, .
	\end{align*}
	
	By rearranging this, we obtain the following theorem.
	
	\begin{Thm}\label{thm:A12reduction}
		For odd \( n \), the following equality holds
		\begin{align*}
		[a_0 \mid a_1, \ldots, a_n \dblslash x \mid a_{n+1}] = \\
		& \hspace{-7em} [a_0 \mid a_1 \shuffle a_2 a_3 \cdots a_n \dblslash x \mid a_{n+1}] - \sum_{j=1}^{\floor{n/2}} R_{2j}^n \, .
		\end{align*}
		
		And in particular for odd \( n \geq 3 \),
		\[
		[a_0 \mid a_1, \ldots, a_n \dblslash x \mid a_{n+1}] \, ,
		\] is explicitly given as a sum of hyperlogs in \( \leq n - 2 \) variables, modulo products.
	\end{Thm}
	
	\subsection{Reduction of generalised hyperlogarithms to multiple polylogarithms}\label{sec:asmplccr}
	
	As a last addendum to the reduction procedure, we show how the hyperlogarithms it produces are written as MPLs of depth \( \leq n - 2 \), involving `coupled cross-ratio' arguments. \medskip
	
	We first recall the definition of the cross-ratio.
	
	\begin{Def}[Cross-ratio]\label{def:crossratio}
		Let \( a, b, c, d \in \Proj^1(\C) \) be four generic points.  The \emph{cross-ratio} \( \CR(a,b,c,d) \), which we may abbreviate as \( abcd \), of these points is defined by
		\[
			abcd = \CR(a,b,c,d) \coloneqq \frac{a-c}{a-d} \bigg/ \frac{b-c}{b-d} \, .
		\]
		If one of the points is \( \infty \in \Proj^1(\C) \), the arithmetic of infinity in the Riemann sphere gives correct results such as
		\[
			abc\infty = \CR(a,b,c,\infty) = \frac{a - c}{a - \infty} \bigg/ \frac{b - c}{b - \infty} = \frac{a - c}{b - c} \, .
		\]
	\end{Def}
	
	These type of arguments play a significant role in describing identities and functional equations between weight 4 MPL's \cite{gangl2015weight4mpl} and between 5 MPL's \cite{charlton2016identities}.  The deep reason for this is that cross-ratios provide coordinates on the moduli space \( \mathfrak{M}_{0,n} \) of \( n \) marked points; the connection between multiple polylogarithms, cross-ratio arguments, and \( \mathfrak{M}_{0,n} \) is discussed in Section 6 of \cite{brown2009multiple}. \medskip
	
	After applying Dan's reduction procedure, we will obtain a number of terms of the form
	\(
	[a_0 \mid a_1, \ldots, a_n \dblslash x \mid a_{n+1}] \, .
	\)
	which we want to recognise as some MPL \( I_{s_1,\ldots,s_k}(z_1,\ldots,z_k) \).  We do this by first converting to an ordinary hyperlogarithm, with \( x \rightsquigarrow \infty \) using \autoref{prop:hasi}.  We have
	\[
	[a_0 \mid a_1, \ldots, a_n \dblslash x \mid a_{n+1}]  = [\tfrac{1}{a_0 - x} \mid \tfrac{1}{a_1 - x}, \ldots, \tfrac{1}{a_n - x} \mid \tfrac{1}{a_{n+1} - x}] \, .
	\]
	This ordinary hyperlogarithm is invariant under affine transformation, so apply the translation \( t \mapsto t - \tfrac{1}{a_0 - x} \).  This sets the lower bound of integral to \( 0 \).  The other arguments change as follows
	\[
	\frac{1}{a_i - x} \mapsto \frac{1}{a_i - x} - \frac{1}{a_0 - x} = \frac{a_0 - a_i}{(a_i - x)(a_0 - x)} \, .
	\]
	Now apply the scaling \( t \mapsto t \tfrac{(a_{n+1} - x)(a_0 - x)}{a_0 - a_{n+1}} \), which sets the upper bound of the integral to \( 1 \).  The other arguments change to
	\begin{align*}
	\frac{a_0 - a_i}{(a_i - x)(a_0 - x)} &\mapsto \frac{a_0 - a_i}{(a_i - x)(a_0 - x)} \frac{(a_{n+1} - x)(a_0 - x)}{a_0 - a_{n+1}} \\
	&  = \frac{(a_0 - a_i)(a_{n+1} - x)}{(a_i - x)(a_0 - a_{n+1})} \\
	& \eqqcolon \CR(x, a_0, a_{n+1}, a_i) \, .
	\end{align*}
	Overall, we find that
	\begin{align*}
	 [a_0 \mid a_1, \ldots, a_n \dblslash x \mid a_{n+1}]
	= [0 \mid \CR(x, a_0, a_{n+1}, a_1), \ldots, \CR(x, a_0, a_{n+1}, a_n) \mid 1] \, .
	\end{align*}
	This can then be written as an MPL by way of \autoref{def:mpl}.
		
	In the reduction procedure, the number of variables is reduced from \( n \) to \( n-2 \) in each integral.  This means that at least two of the \( a_i \)'s above will equal \( a_0 \) when we apply this conversion.  In this situation the cross-ratio reduces to 0 (or indeed the argument itself will be identically 0 after the translation step).  This has the effect of giving an MPL of depth \( \leq n - 2 \). \medskip
	
	\subsubsection*{Coupled cross-ratio arguments} Finally, notice the cross-ratios which appear in the arguments of this MPL all start with the same 3 variables \( \CR(x, a_0, a_{n+1}, {} \mathop{\cdot} {}) \).  The cross-ratios are somehow `coupled' together, and so I refer to these as `\emph{coupled cross-ratio}' arguments.  The identities in \cite{gangl2015weight4mpl} and \cite{charlton2016identities} typically use coupled cross-ratio arguments to better highlight the symmetries involved.
	
	Since coupled cross-ratio arguments appear frequently, we employ the following shorthand notation when writing them.
	
	\begin{Not}[Shorthand for coupled cross-ratio arguments]\label{not:ccr}
		If the multiple polylogarithm \( I_{s_1,\ldots,s_k}(abcd_1, abcd_2, \ldots, abcd_k) \) has coupled cross-ratio arguments, we may employ the following shorthand to write the arguments
		\[
			I_{s_1,\ldots,s_k}(abcd_1d_2\cdots d_k) \coloneqq I_{s_1,\ldots,s_k}(abcd_1, abcd_2, \ldots, abcd_k) \, .
		\]
	\end{Not}
	
	\section{\texorpdfstring{Reduction of \( I_{1,1,1,1} \)}{Reduction of I\_1111}}\label{sec:dan4}
	
	\subsection{\texorpdfstring{Procedure when \( n = 4 \), correcting Dan's reduction of \( I_{1,1,1,1} \)}{Procedure when n = 4, correcting Dan's reduction of I\textunderscore{}1111}}
	
	In this section we will run the reduction method for \( n = 4 \), in order to correct the expression Dan gives for \( I_{1,1,1,1}(w, x, y, z) \), or more precisely for \( I(a; b, c, d, e; f) \).  We can obtain the reduction for \( I_{1,1,1,1}(w, y, x, z) \) by setting \( a = 0 \) and \( f = 1 \). \medskip
	
	To start the reduction procedure, we apply \autoref{thm:reduction}, with \( n = 4 \) to obtain
	\begin{align*}
	& 2 [a_0 \mid a_1, a_2, a_3, a_4 \dblslash x \mid a_5] \\
	& = (R^4(1, 2 \mid 1, 3) - R^4(1, 3, \mid 1, 4)) - (R^4(1, 3, \mid 2, 3) - R^4(2, 4 \mid 3, 4)) \, ,
	\end{align*}
	modulo products.
	
	Let us focus on the term \( R^4(1,2 \mid 1,3) \) now.  This is supposed to be the expression for \[
	A^4_{1,2} + A^4_{1,3} = [a_0 \mid a_1, a_2, a_3, a_4 \dblslash x \mid a_5] + [a_0 \mid a_1, a_3, a_2, a_4 \dblslash x \mid a_5] \, ,
	\]
	as a sum in \( \leq n-2 \) variables, using the \( D \) operator and \autoref{prop:transposition}.  By this, we have
	\begin{align*}\label{n4:transposition}
	A^4_{1,2}  + A^4_{1,3} &= D([a_0 \mid a_1, a_2, a_3, a_4 \dblslash x \mid a_5], 2) + {} \\
	& \hspace{2em} {} - D([a_0 \mid a_1, x, a_3, a_4 \dblslash a_2 \mid a_5], 3) + {} \numberthis \\
	& \hspace{2em} {} + D([a_0 \mid a_1, x, a_2, a_4 \dblslash a_3 \mid a_5], 2) \, .
	\end{align*}
	
	Now each \( D \) is an explicit sum in \( \leq n-2 \) variables, using \autoref{prop:switchx} and the operator \( B \).  Doing this for the first term gives
	\begin{align*}
	D([a_0 \mid a_1, a_2, a_3, a_4 \dblslash x \mid a_5], 2) &= B([a_0 \mid a_1, a_2, a_3, a_4 \dblslash x \mid a_5], 2) \\
	&= \sum_{I} (-1)^\size{I} A([a_0 \mid a_1, a_2, a_3, a_4 \dblslash x \mid a_5], 2, I) \, ,
	\end{align*}
	where the sum runs over all \( I \subset \Set{1,2,\ldots,n} \) containing \( 2 \) and having \( \size{I} \geq 2 \).  In this case the \( I \) ranges over the subsets \( \Set{1,2} \), \( \Set{2,3} \), \( \Set{2,4} \), \( \Set{1,2,3} \), \( \Set{1,2,4}, \Set{2,3,4} \) and \( \Set{1,2,3,4} \).  We obtain
	\begin{align*}
	{} ={} & [a_0 \mid a_2, a_2, a_3, a_4 \dblslash x \mid a_5]
	+ [a_0 \mid a_1, a_2, a_2, a_4 \dblslash x \mid a_5]
	+ [a_0 \mid a_1, a_2, a_3, a_2 \dblslash x \mid a_5] + {}\\
	& {} - [a_0 \mid a_2, a_2, a_2, a_4 \dblslash x \mid a_5]
	- [a_0 \mid a_2, a_2, a_3, a_2 \dblslash x \mid a_5]
	- [a_0 \mid a_1, a_2, a_2, a_2 \dblslash x \mid a_5] + {} \\
	& {} + [a_0 \mid a_2, a_2, a_2, a_2 \dblslash x \mid a_5] \, .
	\end{align*}
	From here \( a_2 \) must be shuffled out of the first position of each term using \autoref{lem:shuffleouty}.  This will let us express each term as the difference of an integral from \( a_2 \) to \( a_0 \) and from \( a_2 \) to \( a_5 \), as in \autoref{lem:integralbounddifference}.  Doing so shows that this can be written as
	\[
		\psi(a_0) - \psi(a_5) \, ,
	\] where
	\begin{align*}
	\psi(c) ={} & 
	[a_2 \mid a_1, a_2, a_2, a_2 \dblslash x \mid c]
	- [a_2 \mid a_1, a_2, a_2, a_4 \dblslash x \mid c]
	- [a_2 \mid a_1, a_2, a_3, a_2 \dblslash x \mid c] + {} \\
	& {} + 3 [a_2 \mid a_3, a_2, a_2, a_2 \dblslash x \mid c]
	- [a_2 \mid a_3, a_2, a_2, a_4 \dblslash x \mid c]
	- [a_2 \mid a_3, a_2, a_4, a_2 \dblslash x \mid c] + {} \\
	&{} - [a_2 \mid a_3, a_4, a_2, a_2 \dblslash x \mid c]
	- [a_2 \mid a_4, a_2, a_2, a_2 \dblslash x \mid c] \, .
	\end{align*}
	
	Now this must be repeated for the other two occurrences of \( D \) in \autoref{n4:transposition}, in order to get an expression for \( R^4(1,2\mid1,3) \).  Then the whole procedure must be repeated for the remaining 3 relations \( R^4 \).  Finally we convert to MPL's with coupled cross-ratio arguments, as in  \autoref{sec:asmplccr}.
	
	After doing this, we may use the following identities to convert between \( I_{1,3} \), and \( I_{2,2} \) and \( I_{3,1} \).  Certain versions of these identities can be found in \cite{gangl2015weight4mpl}.  Dan must have known about these identities (or similar ones), in order to obtain his reduction of \( I_{1,1,1,1} \), but the primary source of them is not clear.  Dan does not explicitly list the identity he uses to obtain his Th\'eor\`eme.
	\begin{Id}
	Modulo products, the following identities relate \( I_{2,2} \), \( I_{1,3} \) and \( I_{3,1} \).
	\begin{align*}
	I_{2, 2}(x, y) & \modsh -I_{1, 3}(x, y)-I_{1, 3}(y, x)-I_{3, 1}(x, y) \\
	I_{1, 3}(x, y) & \modsh I_{4}(x)-I_{3, 1}(x, \tfrac{x}{y}) \, .
	\end{align*}
	
	We can write these identities in the following way to see they preserve coupled cross-ratio arguments.
	\begin{align*}
		I_{2,2}(abcde) & \modsh -I_{1,3}(abcde) - I_{1,3}(abced) - I_{3,1}(abcde) \\
		I_{1,3}(abcde) & \modsh I_4(abcd) - I_{3,1}(badce) \, .
	\end{align*}
	\end{Id}
	
	By converting all terms of the result to \( I_{3,1} \) and \( I_4 \), using the above identities, we obtain the following theorem.
	
	\begin{Thm}[Correction to Th\'eor\`eme 2 in \cite{dan2011surla2}]\label{thm:dan_wt4_correct}
		As shorthand, write \( abcd \coloneqq \CR(a,b,c,d) \) for the cross-ratio.  Moreover, write \( [x,y]_{3,1} \coloneqq [0 \mid x, 0, 0, y \mid 1] \) and \( [x]_4 \coloneqq [0 \mid x, 0, 0, 0, \mid 1] \).  Finally write \( [abcde]_{3,1} \coloneqq [abcd,abce]_{3,1} \) as shorthand for coupled cross-ratios.  Then modulo products
		\[
		[a \mid b, c, d, e \mid f] \modsh \phi_4(a; b, c, d, e) - \phi_4(f; b, c, d, e) \, ,
		\]
		where
		\begin{align*}\label{eqn:i1111_as_i31_and_i4}
		& \hspace{-2em} 2\phi_4(a; b, c, d, e) \coloneqq  \\
		& [b d a \infty c]_{3,1} 
		- [b d a \infty e]_{3,1} 
		- [b e a d c]_{3,1} 
		+ [b e a \infty c]_{3,1} 
		- [b \infty c a d]_{3,1} + {} \\
		& {}
		+ [b \infty c a e]_{3,1} 
		+ [b \infty d a c]_{3,1} 
		+ [b \infty e a c]_{3,1} 
		+ [c d a b e]_{3,1} 
		- [c d a \infty e]_{3,1}  + {} \\
		& {}
		+ [c e a b \infty]_{3,1} 
		- [c e a d \infty]_{3,1} 
		- 2 [c \infty b a d]_{3,1} 
		- [c \infty d a b]_{3,1} 
		- [c \infty d a e]_{3,1}  + {} \\
		& {}
		- 2 [c \infty e a d]_{3,1} 
		- [d b c a \infty]_{3,1} 
		- [d b e a c]_{3,1} 
		- [d b \infty a c]_{3,1} 
		- [d c b a \infty]_{3,1}  + {} \\
		& {}
		- [d c e a \infty]_{3,1} 
		- [d c \infty a e]_{3,1} 
		+ 2 [d \infty a b c]_{3,1} 
		- 2 [d \infty a b e]_{3,1} 
		+ [e b c a \infty]_{3,1}  + {} \numberthis \\
		& {}
		+ [e b d a \infty]_{3,1} 
		+ [e b \infty a c]_{3,1} 
		+ [e c b a d]_{3,1} 
		- [e c b a \infty]_{3,1} 
		- [e c d a \infty]_{3,1}  + {} \\
		& {} 
		- [e c \infty a b]_{3,1} 
		+ [e c \infty a d]_{3,1} 
		+ [e \infty a b c]_{3,1} 
		- [e \infty a d b]_{3,1} 
		+ 2 [\infty b a d c]_{3,1}   + {} \\
		& {} 
		- 2 [\infty b a d e]_{3,1} 
		+ [\infty c a b d]_{3,1} 
		- [\infty c a d e]_{3,1} 
		- [\infty d b a e]_{3,1} 
		- 2 [\infty d c a b]_{3,1}   + {} \\
		& {} 
		- [\infty d e a b]_{3,1} 
		+ [\infty e c a b]_{3,1} 
		- [\infty e c a d]_{3,1} 
		+ [\infty e d a b]_{3,1} 
		- [\infty e d a c]_{3,1}  + {} \\ 
		& \hspace{2em} {} + \gamma_4(a;b,c,d,e) \, .
		\end{align*}
		And
		\begin{align*}
		& \hspace{-2em} \gamma_4(a;b,c,d,e) {} \coloneqq {} \\
		& [a b c e]_{4} 
		 + 5 [a b d \infty]_{4} 
		 - 4 [a b e \infty]_{4} 
		 + 2 [a b \infty c]_{4} 
		 + 4 [a c b d]_{4} 
		 - 2 [a c b e]_{4} 
		 + 2 [a c d \infty]_{4} + {} \\
		 & {}
		 + 2 [a c e \infty]_{4} 
		 - 2 [a c \infty b]_{4} 
		 - 2 [a d c e]_{4} 
		 + 2 [a d e \infty]_{4} 
		 + [a d \infty b]_{4} 
		 + 8 [a d \infty c]_{4} 
		 + 2 [a e c d]_{4}  + {} \\
		 & {}
		 + [a e d \infty]_{4} 
		 - 3 [a e \infty b]_{4} 
		 + 2 [a e \infty c]_{4} 
		 + 2 [a \infty b d]_{4} 
		 - 4 [a \infty b e]_{4} 
		 + 4 [a \infty c d]_{4} 
		 + 3 [a \infty c e]_{4}
		\end{align*}
		is an explicit sum of \( I_4 \)'s of rational functions.
	\end{Thm}
	
	\begin{proof}[Supporting calculations]
		The above reduction was computed and automatically \LaTeX{}ed with the worksheet \mathematicanb{wt4\_dan.nb} and the TeXUtilities package \cite{TeXUtilities} to ensure no typos occur.  This worksheet verifies the reduction against precomputed results in \mathematicanb{wt4\_dan\_precomputed.m}.  In the worksheet \mathematicanb{wt4\_check\_dan.nb}, these precomputed reductions are themselves checked by finding the symbol with the PolyLogTools package \cite{PolylogTools}.  This confirms that \autoref{thm:dan_wt4_correct} holds.  I have also implemented a standalone `lite' version of PolyLogTools in \mathematicanb{PolyLogToolsLite.m} for this purpose.
	\end{proof}
	
	\begin{Rem}
		In the original paper, Dan does not give the \( I_4 \) terms explicitly, but says only that such an explicit linear combination exists.  We give it give it here explicitly, for completeness.  We also write the arguments in the \( I_{3,1} \) terms as coupled cross-ratios because this highlights an important structural property of the reduction, which plays into the analysis of Dan's \( I_{3,1} \) functional equation \autoref{sec:i31_fe} below.
	\end{Rem}
	
	On the level of the \emph{MPL symbol}, the above result holds modulo products.  Working modulo \( \delta \), the terms in \( \gamma_4(a;b,c,d,e) \) go to 0, giving the remaining terms of \( \phi_4(a;b,c,d,e) \) as the leading terms in the expression. \medskip
	
	Potentially more interesting is a reduction to \( I_4 \) and \( I_{2,2} \) in light of the folklore conjecture that indices 1 can always be eliminated from MPL's.  For that, we can make use of the following identity, from Gangl \cite{gangl2015weight4mpl}.
	
	\begin{Id}[Proposition 9 in \cite{gangl2015weight4mpl}]
	The following identity expresses \( I_{3,1} \) in terms of \( I_{2,2} \).
	\[
	I_{3, 1}(x, y) \modsh \tfrac{1}{2} (I_{2, 2}(y, x)-I_{2, 2}(x, y)) \, .
	\]
	We can rewrite this to see it also preserves coupled cross-ratio arguments
	\[
			I_{3,1}(abcde) \modsh \tfrac{1}{2} ( I_{2,2}(abced) - I_{2,2}(abcde) ) \, .
	\]
	\end{Id}
		
	By using this identity, we can turn \autoref{thm:dan_wt4_correct} into a reduction to \( I_{2,2} \) and \( I_4 \), modulo products, to obtain the following corollary.
		
	\begin{Cor}[Reduction of \( I_{1,1,1,1} \) to \( I_{2,2} \) and \( I_4 \)]
		As shorthand, write \( abcd \coloneqq \CR(a,b,c,d) \) for the cross-ratio.  Also write \( [x,y]_{2,2} = [0 \mid x, 0, y, 0 \mid 1] \) and \( [x]_4 = [0 \mid x, 0, 0, 0, \mid 1] \).  Then modulo products
		\[
		[a \mid b, c, d, e \mid f] \modsh \phi_4(a;b,c,d,e) - \phi_4(f;b,c,d,e) \, ,
		\]
		where \( \phi_4(a;b,c,d,e) \) is exactly as given in \autoref{eqn:i1111_as_i31_and_i4}, and we understand that \( [x,y]_{3,1} \) is now expanded as follows
		\[
		[x,y]_{3,1} = \tfrac{1}{2} ( [y,x]_{2,2} - [x,y]_{2,2} ) \, .
		\]
	\end{Cor}
	
	\subsection{\texorpdfstring{Relation to Dan's previous reduction, and \( I_{3,1} \) functional equations}{Relation to Dan's previous reduction, and I\textunderscore{}31 functional equations}}\label{sec:i31_fe}
	
	Recall that in Th\'eor\`eme 3 of \cite{dan2008surla}, Dan gives a different reduction for \( I_{1,1,1,1} \) to \( I_{3,1} \) and \( I_4 \) terms.  This version is specific to the weight 4 case \( I_{1,1,1,1} \), and produces a more symmetrical and structured identity.  Nevertheless, there is a typo in the expression Dan gives, but fortunately one can take advantage of the extra structure to easily correct the result.  The correction below was provided by Gangl.
	
	\begin{Thm}[Th\'eor\`eme 3 in \cite{dan2008surla}, corrected by Gangl]\label{thm:dan_wt4_previous}
		As shorthand, write \( abcd \coloneqq \CR(a,b,c,d) \), and \( abc \coloneqq \CR(a,b,c,\infty) \) for the cross-ratio.  Moreover, write \( [x,y]_{3,1} = [0 \mid x, 0, 0, y \mid 1] \) and \( [x]_4 = [0 \mid x, 0, 0, 0, \mid 1] \).  Then modulo products
		\[
		[a \mid b, c, d, e \mid f] \modsh F(a; b, c, d, e) - F(f; b, c, d, e) \, ,
		\]
		where
		\begin{align*}
		20 F(a; b, c, d, e) \coloneqq {} & G(a, b, c, d, e) + {} \\
		& {} - G(\infty, b, c, d, e) - G(a, \infty, c, d, e) - G(a, b, \infty, d, e) + {} \\
		& {} - G(a, b, c, \infty, e) - G(a, b, c, d, \infty) + {} \\ 
		& {} + 10 H(a, b, c, d, e) \, .
		\end{align*}
		And \( G \) and \( H \) are defined by
		\begin{align*}
		G(a,b,c,d,e) & \coloneqq \Cyc_{\Set{a,b,c,d,e}} \big( [abcd,abce]_{3,1} - [edcb,edca]_{3,1} \\
			& \hspace{4em} - 3 [abdc,abde]_{3,1} + 3 [edbc,edba]_{3,1}  \big) \\
		H(a,b,c,d,e) & \coloneqq \Cyc_{\Set{a,b,c,d,e}} \left( [cab]_4 - [bda]_4 + [adb]_4 - [bad]_4 \right) \, .
		\end{align*}
	\end{Thm}
	
	\begin{Rem}\label{rem:dan_mistake_1}
		The mistakes in Dan's expression occur in the first summand of \( G \), where he write \( [abcd,bcde]_{3,1} \) rather than \( [abcd,abce]_{3,1} \).  This is easily corrected upon noticing that for the remaining summands, the first 3 cross-ratio slots agree in each pair -- that is, each is a coupled cross-ratio.  There is also a mistake in (his equivalent of) \( H \), where the sign of the first term \( [cab]_4 \) is flipped.  Moreover there appears to be a global sign error, so \( -20 \) in the definition of \( F \) is replaced with \( 20 \) above.
	\end{Rem}
	
	\begin{proof}[Supporting calculations] The worksheet \mathematicanb{wt4\_check\_dan.nb} computes the above reduction \( F \), and verifies it against a precomputed result in \mathematicanb{wt4\_dan\_precomputed.m}.  Then this precomputed result is itself checked using the symbol, to see \autoref{thm:dan_wt4_previous} holds.
	\end{proof}
	
	Once Dan has these two reductions, he wonders how the combinations \( \phi_4 \) and \( F \) relate.  By setting the two reductions \autoref{thm:dan_wt4_correct} and \autoref{thm:dan_wt4_previous} equal,
	\begin{align*}
	\phi_4(a;b,c,d,e) - \phi_4(f;b,c,d,e) & \modsh [a \mid b,c,d,e \mid f]\\
	& \modsh F(a;b,c,d,e) - F(f;b,c,d,e) \, ,
	\end{align*} one obtains a functional equation reducing a certain combination of \( I_{3,1} \)'s to \( I_4 \)'s.  Specifically he asks the question of whether \( \phi_4 \) and \( F \) are exactly equal, and whether this functional equation, \emph{a prior} of 4 variables, splits into two functional equations of 3 variables. \medskip
	
	Using the symbol, we can answer Dan's question as follows.
	
	\begin{Claim}
		The combinations \( F(a;b,c,d,e) \) and \( \phi_4(a;b,c,d,e) \) are \emph{not} equal.
		
		\begin{proof}
			If \( F \) and \( \phi_4 \) were equal, then their symbols modulo \( \sh \) would also have to be equal.  The worksheet \mathematicanb{wt4\_check\_i31\_fe.nb} shows that these functions have different symbols, modulo \( \sh \), and even modulo \( \delta \). 
			
			Moreover, by checking the symbol, this worksheet shows that \( \phi_4(a;b,c,d,e) \) is \emph{not} even cyclically invariant.  So there is no hope that \( F \) and \( \phi_4 \) agree!
		\end{proof}
	\end{Claim}
	
	Nevertheless, by comparing the two expansions
	\begin{align*}
	 \phi_4(a;b,c,d,e) - \phi_4(f;b,c,d,e) & \modsh [a \mid b,c,d,e \mid f] \\
	 & \modsh F(a;b,c,d,e) - F(f;b,c,d,e) \, ,
	 \end{align*}
	 we obtain a 4-variable functional equation relating a certain combination of \( I_{3,1} \)'s to \( I_4 \)'s, modulo products \( \sh \).  Unfortunately, the functional equation which results is not, perhaps, as interesting as one might hope.  To explain this statement, we must first recall the functional equations for \( I_{3,1} \) found by Gangl \cite{gangl2015weight4mpl}.
	
	\begin{Thm}[Gangl, Theorem 7 in \cite{gangl2015weight4mpl}]\label{thm:gangli31fe}
		Using the shorthand \( abcd \coloneqq \CR(a,b,c,d) \) for the cross-ratio, and the shorthand notation \( I_{3,1}(abcde) \coloneqq I_{3,1}(abcd,abce) \) for coupled cross-ratios, as in \autoref{not:ccr}, the following identities hold, modulo products.
		\begin{align*}
		& I_{3,1}\big( [(a b) c d e] 
		- [(b a) c d e]\big) + {} \\
		& {} + I_4(-[abcd] 
		+ [abce]
		+ 3 [abde] ) 
		\modsh 0  \tag{$I_{3,1}$ $ab$} \\[1em]
		& I_{3,1}\big( [a (b c) d e] 
		- [a (c b) d e] \big) + {} \\
		& {} + I_4([c b a d]
		 - [c b a e]
		 + 2 [a b d e] 
		 + 2 [c a d e] 
		 + [c b d e] ) 
		\modsh 0 \tag{$I_{3,1}$ $bc$} \\[1em]
		& I_{3,1}\big( [a b c (d e)]
		 + [a b c (e d)] \big)
		\modsh 0 \tag{$I_{3,1}$ $de$} \\[1em]
		& I_{3,1}\big( [(a b c d) e] 
		 + [(b c d a) e]
		 + [(c d a b) e]
		 + [(d a b c) e] \big) + {}\\
		& {} + I_4\big( [a c b e]
		 + [b d c e] 
		 + [c a d e] 
		 + [d b a e] + {} \tag{$I_{3,1}$ cyc}  \\
		& \hspace{2.5em} {} + 2 [abde]
		 + 2 [b c a e]
		 + 2 [c d b e] 
		 + 2 [d a c e] \big)
		  \modsh 0
		\end{align*}
	\end{Thm}
		
	\begin{Claim}[Decomposition of Dan's \( I_{3,1} \) functional equation]
		Dan's functional equation for \( I_{3,1} \) arising from comparing
		\begin{align*}
		\phi_4(a;b,c,d,e) - \phi_4(f;b,c,d,e) & \modsh [a \mid b,c,d,e \mid f] \\
		& \modsh F(a;b,c,d,e) - F(f;b,c,d,e) \, ,
		\end{align*}
		can be written as a combination of 629 instances of the \( I_{3,1} \) functional equations in \autoref{thm:gangli31fe}.  The leftover \( I_4 \) terms cancel pairwise using the inversion relation \( I_4(x) = -I_4(\tfrac{1}{x}) \).
		
		\begin{proof}
			The decomposition of Dan's \( I_{3,1} \) functional equation into 629 instances of Gangl's \( I_{3,1} \) functional equations is given in \mathematicanb{wt4\_fe\_i31.m}.  The worksheet \mathematicanb{wt4\_check\_i31\_fe.nb} verifies that this decomposition holds.
		\end{proof}
	\end{Claim}
	
	In fact, this was to be expected.  Gangl has found that the functional equations in \autoref{thm:gangli31fe} provide a basis for the space of all relations between \( I_{3,1}(abcde) \) terms.  Moreover, we see from \autoref{thm:dan_wt4_correct} and \autoref{thm:dan_wt4_previous} that every term of the weight 4 reductions can be written with coupled cross-ratio arguments.
	
	\section{\texorpdfstring{Reduction of \( I_{1,1,1,1,1} \)}{Reduction of I\_11111}} \label{sec:dan5}
	
	We shall now apply Dan's reduction procedure to the quintuple-log \( I_{1,1,1,1,1}(v,w,x,y,z) \) to obtain expressions for it in terms of lower depth multiple polylogarithms.  Or rather we shall apply it to the hyperlogarithm \( I(a \mid b, c, d, e, f \mid g) \), like above.
	
	Firstly we will examine the `raw' output of the reduction procedure which reduces \( I_{1,1,1,1,1} \) to the 11 depth \( \leq 3 \) integrals \( I_{5}, I_{4,1}, I_{3,2}, I_{3,1,1},  I_{2,2,1} \), and permutations of these indices.  Then using some identities from Chapter 4 and Chapter 5 of \cite{charlton2016identities} (and reproduced in \autoref{sec:dan_wt5_as_311_32_5} below), we will be able to reduce this expression to explicit \( I_{5}, I_{3,2}, I_{3,1,1} \) terms only, modulo products.
	
	In order to explicitly confirm the folklore conjecture that indices \( 1 \) can always be eliminated from MPL's, we then need to reduce \( I_{3,1,1} \) to \( I_{3,2} \) terms and \( I_5 \) terms.  On the symbol level, we can do this in a rather `brute-force' way to give a reduction of \( I_{3,1,1} \) to \( I_{3,2} \), but only modulo \( \delta \). \medskip
	
	\subsubsection*{Supporting calculations} All of the reduction results in this section were computed and automatically \LaTeX{}ed using with the worksheet \mathematicanb{wt5\_dan.nb} and TeXUtilities \cite{TeXUtilities}.  This worksheet verifies the reductions against precomputed results in \mathematicanb{wt5\_dan\_precomputed.m}.  Using the PolyLogTools Package \cite{PolylogTools} to calculate the symbol, these precomputed reductions are themselves checked in the worksheet \mathematicanb{wt5\_check\_dan.nb}.  I have also implemented a standalone `lite' version of PolyLogTools in \mathematicanb{PolyLogToolsLite.m}, for this purpose.
	
	\subsection{\texorpdfstring{`Raw' output of \( I_{1,1,1,1,1} \) reduction}{Raw output of I\_11111 reduction}}
	
	When attempting to reduce \( I_{1,1,1,1,1} \) with Dan's reduction method, there are two choices.  We can either use the efficient approach from \autoref{sec:structured_alln} which works for all \( n \).  Or we can use the efficient approach from \autoref{sec:structured_oddn} which works only for \( n \) odd.  The \( n \) odd approach has the advantage of producing significantly shorter reductions.  We will compare the two initial results to see how much better the \( n \) odd approach works.
	
	\subsubsection{All \( n \) approach:} Apply the all \( n \) approach to \( [a \mid b, c, d, e, f \mid g] \).  The result can be written as \( \psi_5(a; b, c, d, e, f) - \psi_5(g; b, c, d, e, f) \), where \( \psi_5 \) consists of those terms which contain the variable \( a \).  We obtain the following distribution of terms in \( \psi_5 \).
	
	\begin{center}
		\begin{tabular}{c|c||c|c}
			Integral & Number in \( \psi_5 \) & Integral & Number in \( \psi_5 \) \\ \hline
			$ I_{3, 1, 1} $ & 22 & $ I_{4, 1} $ & 34 \\ 
			$ I_{2, 2, 1} $ & 26 & $ I_{3, 2} $ & 41 \\ 
			$ I_{2, 1, 2} $ & 22 & $ I_{2, 3} $ & 39 \\ 
			$ I_{1, 3, 1} $ & 21 & $ I_{1, 4} $ & 29 \\ 
			$ I_{1, 2, 2} $ & 22 & $ I_{5} $ & 37 \\ 
			$ I_{1, 1, 3} $ & 14 & & \\ \hline
			\multicolumn{3}{c|}{Total number} & 307
		\end{tabular}
	\end{center}
	
	\subsubsection{Odd \( n \) approach:} Apply the odd \( n \) approach to \( [a \mid b, c, d, e, f \mid g] \).  The result can be written as \( \phi_5(a; b, c, d, e, f) - \phi_5(g; b, c, d, e, f) \), where \( \phi_5 \) consists of those terms which contain the variable \( a \).  We obtain the following distribution of terms in \( \phi_5 \).
	
	\begin{center}
		\begin{tabular}{c|c||c|c}
			Integral & Number in \( \phi_5 \) & Integral & Number in \( \phi_ 5 \) \\ \hline
			$ I_{3, 1, 1} $ & 6 & $ I_{4, 1} $ & 11 \\ 
			$ I_{2, 2, 1} $ & 7 & $ I_{3, 2} $ & 17 \\ 
			$ I_{2, 1, 2} $ & 7 & $ I_{2, 3} $ & 17 \\ 
			$ I_{1, 3, 1} $ & 6 & $ I_{1, 4} $ & 11 \\
			$ I_{1, 2, 2} $ & 5 & $ I_{5} $ & 20 \\ 
			$ I_{1, 1, 3} $ & 6 & & \\ \hline
			\multicolumn{3}{c|}{Total number} & 113
		\end{tabular}
	\end{center}
	
	Already one can see that the \( n \) odd approach is significantly better as it involves only about one-third the number of terms, compared to the all \( n \) approach.  This reduction of \( I_{1,1,1,1,1} \) to depth \( \leq 3 \) integrals is (just) short enough to give explicitly.
	
	\begin{Id}[Reduction of \( I_{1,1,1,1,1} \) to depth \( \leq 3 \)] \label{id:i11111_dan_reduction}
		As shorthand, write \( abcd \coloneqq \CR(a,b,c,d) \) for the cross-ratio.  Moreover, recall the coupled cross-ratio shorthand notation from \autoref{not:ccr}, which has \( I_{s_1,\ldots,s_k}(abcd_1 \ldots d_k) \coloneqq \allowbreak I_{s_1,\ldots,s_k}(abcd_1, \ldots, abcd_k) \).  Then modulo products, Dan's efficient odd \( n \) reduction procedure gives the following reduction
		\[
		[a \mid b, c, d, e, f \mid g] \modsh \phi_5(a; b, c, d, e, f) - \phi_5(g; b, c, d, e, f) \, ,
		\]
		where
		\begin{align*}	& \hspace{-1em} \phi_5(a; b, c, d, e, f) \coloneqq \\
		& {}
			I_{3, 1, 1}( -[b d a c e f] 
			+ 2 [b d a \infty e f] 
			+ [d \infty a b e f] 
			+ [d \infty a c e f] 
			+ [\infty b a c e f] 
			+ [\infty b a d e f] ) + {} \\[1ex]
			& {} + I_{2, 2, 1}( [b d a \infty e f] 
			+ [d \infty a b e f] 
			+ [d \infty a c b f] 
			+ [f \infty a c e b] + {} \\
			& \hspace{3.5em} {}
			- [f \infty a d e b] 
			+ [\infty b a c d f] 
			+ [\infty b a c e f] ) + {} \\[1ex] 
			& {} + I_{2, 1, 2}( [b d a \infty e f] 
			- [b f a c e \infty] 
			+ [b f a d e \infty] 
			+ [d \infty a b e f]  {} \\
			& \hspace{3.5em} {}
			+ [d \infty a c b e] 
			+ [\infty b a c d e] 
			+ [\infty b a d e f] ) + {} \\[1ex]
			& {} + I_{1, 3, 1}( -[b d a c \infty f] 
			+ [d \infty a b e f] 
			+ [f \infty a c d b] 
			- 2 [f \infty a d e b] 
			+ [\infty b a c d f] 
			- [\infty b a d e f] ) + {} \\[1ex]
			& {} + I_{1, 2, 2}( -[b d a c \infty e] 
			- [b f a c d \infty] 
			+ [b f a d e \infty] 
			+ [d \infty a b e f] 
			- [f \infty a d e b] ) + {} \\[1ex]
			& {} + I_{1, 1, 3}( -[b f a c d e] 
			+ 2 [b f a d e \infty] 
			+ [d \infty a b e f] 
			+ [f \infty a c d e] 
			+ [\infty b a c d e] 
			+ [\infty b a d e f] ) + {} \\[1ex]
			& {} + I_{4, 1}( [b d a c f] 
			- [b d a e f] 
			- 3 [b d a \infty f] 
			- 3 [d \infty a b f] 
			- [d \infty a c f] 
			+ [d \infty a e f]  {} \\
			& \hspace{3em} {}
			- [f \infty a c b] 
			+ 3 [f \infty a d b] 
			- 3 [f \infty a e b] 
			- 2 [\infty b a c f] 
			- 2 [\infty b a e f] ) + {} \\[1ex] 
			& {} + I_{3, 2}( [b d a c e] 
			- [b d a e f] 
			- 2 [b d a \infty e] 
			- [b d a \infty f] 
			+ [b f a c \infty] 
			- 2 [b f a d \infty]  {} \\
			& \hspace{3em} {}
			+ [b f a e \infty] 
			- [d \infty a b e] 
			- 2 [d \infty a b f] 
			- [d \infty a c e] 
			+ [d \infty a e f] 
			+ [f \infty a d b]  {} \\
			& \hspace{3em} {}
			- 2 [f \infty a e b] 
			- [\infty b a c e] 
			- [\infty b a d e] 
			- [\infty b a d f] 
			- [\infty b a e f] ) + {} \\[1ex]
			& {} + I_{2, 3}( -[b d a e f] 
			- 2 [b d a \infty e] 
			+ [b f a c e] 
			- [b f a d e] 
			- 2 [b f a d \infty] 
			+ 2 [b f a e \infty]  {} \\
			& \hspace{3em} {}
			- 2 [d \infty a b e] 
			- [d \infty a b f] 
			- [d \infty a c b] 
			+ [d \infty a e f] 
			- [f \infty a c e] 
			+ [f \infty a d e]  {} \\
			& \hspace{3em} {}
			- [f \infty a e b] 
			- [\infty b a c d] 
			- [\infty b a c e] 
			- [\infty b a d e] 
			- [\infty b a d f] ) + {} \\[1ex]
			& {} + I_{1, 4}( [b d a c \infty] 
			- [b d a e f] 
			+ [b f a c d] 
			- 3 [b f a d e] 
			+ 3 [b f a e \infty] 
			- 3 [d \infty a b e]  {} \\
			& \hspace{3em} {}
			+ [d \infty a e f] 
			- [f \infty a c d] 
			+ 3 [f \infty a d e] 
			- [\infty b a c d] 
			+ [\infty b a e f] ) + {} \\[1ex]
			& {} + I_{5}( 6 [a b d \infty] 
			+ [a b f \infty] 
			- [a c b d] 
			- [a c b f] 
			+ [a c d \infty] 
			+ [a c f \infty] 
			+ 2 [a c \infty b]  {} \\
			& \hspace{2.5em} {}
			+ 4 [a d b f] 
			- 4 [a d f \infty] 
			+ 2 [a d \infty b] 
			+ 4 [a e b d] 
			- 6 [a e b f] 
			- 4 [a e d \infty] 
			+ 6 [a e f \infty]  {} \\
			& \hspace{2.5em} {}
			+ 2 [a e \infty b] 
			- [a f b d] 
			+ [a f d \infty] 
			+ 2 [a f \infty b] 
			+ 4 [a \infty b d] 
			+ 4 [a \infty b f] ) \, .
		\end{align*}
	\end{Id}
	
	\subsection{\texorpdfstring{Reduction of \( I_{1,1,1,1,1} \) to \( I_{3,1,1} \), \( I_{3,2} \) and \( I_5 \) modulo products}{Reduction of I\_11111 to I\_311, I\_32 and I\_5}}
	\label{sec:dan_wt5_as_311_32_5}
	
	From Chapter 4 and Chapter 5 of \cite{charlton2016identities}, we have a number of identities which relate weight 5 multiple polylogarithms of depth 2 and depth 3.  We can use these identities to reduce all weight 5 multiple polylogarithms to a combination of \( I_{3,1,1} \), \( I_{3,2} \) and \( I_5 \) terms only.
	
	\subsubsection*{Supporting calculations}  We reproduce the necessary identities below; they may be checked directly by computing the symbol with Duhr's PolylogTools package \cite{PolylogTools}, or with the `lite' version implemented in \mathematicanb{PolyLogToolsLite.m}.  These identities are checked in the worksheet \mathematicanb{wt5\_check\_ids.nb}, using data from \mathematicanb{wt5\_rules\_to\_i311\_i32\_i5.m}.
	
	\subsubsection*{Depth 2} We first consider how to rewrite the depth 2 multiple polylogarithms in terms of \( I_{3,2} \) and \( I_5 \).
	
	 The first result allows us to relate \( I_{a,b} \) in terms of \( I_{b,a} \).  Therefore we can rewrite \( I_{1,4} \) in terms of \( I_{4,1} \), and we can rewrite \( I_{2,3} \) in terms of \( I_{3,2} \).
	
	\begin{Id}[Proposition 4.2.22 in \cite{charlton2016identities}]
		Modulo products, the following identity holds for any depth 2 multiple polylogarithm
		\[
			I_{n,m}(abcde) \modsh -I_{m,n}(badce) + I_{n+m}(abcd) \, .
		\]
		
		\begin{proof}
			This follows directly from the `stuffle' product of the series definition for multiple polylogarithms.  Expanding out the following product gives
			\[
				I_b(y) \ast I_a(\tfrac{x}{y}) = I_{a,b}(x,y) + I_{b,a}(x, \tfrac{x}{y}) - I_{a+b}(x) \, .
			\]
			Looking modulo products, and writing the arguments as coupled cross-ratios gives the above result.
		\end{proof}
	\end{Id}
	
	We can then reduce \( I_{4,1} \) to \( I_{3,2} \) terms, via the following identity.
	
	\begin{Id}\label{id:i41asi32}
		The following identity expresses \( I_{4,1} \) in terms of \( I_{3,2} \), modulo products.
		\[
			I_{4, 1}(a b c d e) {} \modsh -\tfrac{1}{3} I_{3, 2}( [a b c d e] 
			+ [a b c e d] ) \, .
		\]
	\end{Id}
	
	This means every depth 2 weight 5 multiple polylogarithm can expressed in terms of \( I_{3,2} \) and \( I_5 \) terms only.
	
	\subsubsection*{Depth 3}
	
	Now we consider how to rewrite any depth 3 multiple polylogarithms in terms of \( I_{3,1,1} \) and lower depth MPL's.
	
	Theorem 4.3.18 in \cite{charlton2016identities} tells us that all weight 5 depth 3 multiple polylogarithms are somehow `equivalent' modulo \( I_{3,2} \).  In particular, every such multiple polylogarithm can be written as \( I_{3,1,1} \), modulo \( I_{3,2} \) terms.  We try to give the shortest possible identities for this below. \medskip
	
	Firstly, we relate \( I_{1,3,1} \) to \( I_{3,1,1} \), and we relate \( I_{2,2,1} \) to \( I_{3,1,1} \).
	
	\begin{Id}
		The following identities express \( I_{1,3,1} \) in terms of \( I_{3,1,1} \), and \( I_{2,2,1} \) in terms of \( I_{3,1,1} \), modulo products.
		\begin{align*}
			I_{1, 3, 1}(a b c d e f) & \modsh I_{3, 1, 1}(a b c f e d) \\
			I_{2, 2, 1}(a b c d e f) & {} \modsh - I_{3, 1, 1}( [a b c d e f] 
			+ [a b c d f e] 
			+ [a b c f d e] 
			+ [a b c f e d] ) \, .
		\end{align*}
	\end{Id}
	
	The first identity here is an instance of a more general result relating \( I_{a,b,1} \) to \( I_{b,a,1} \), modulo products.  Details of this can be found in Proposition 4.3.15 of \cite{charlton2016identities}. \medskip
	
	The following slightly longer identity relates \( I_{1,1,3} \) and \( I_{3,1,1} \).

	\begin{Id}
		The following identity expresses \( I_{1,1,3} \) in terms of \( I_{3,1,1} \), \( I_{3,2} \) and \( I_{5} \) terms, modulo products.
		\begin{align*}
		I_{1, 1, 3}(a b c d e f) \modsh {} & I_{3, 1, 1}(a b d c f e) + {} \\ 
		& {} + \tfrac{1}{3} I_{3, 2}( [a b c e f] 
		+ [a b e c f] 
		+ [a b e d f] 
		+ [a b e f c] + {} \\
		& \hspace{3.7em} {}
		- [b a e f c] 
		- [b a e f d] 
		+ [b a f e c] ) + {} \\ 
		& {} + \tfrac{1}{3} I_{5}( -4 [a b d e] 
		+ 6 [a b d f]
		- 4 [a b e c] 
		+ 7 [a b e f] 
		+ 16 [a b f c] ) \, .
		\end{align*}
	\end{Id}
	
	Finally, we write \( I_{2,1,2} \) in terms of \( I_{3,1,1} \) and \( I_{1,2,2} \) in terms of \( I_{3,1,1} \).  These involve a more complicated identities, with multiple \( I_{3,1,1} \) terms.
	
	\begin{Id}
		The following identity expresses \( I_{2,1,2} \) in terms of \( I_{3,1,1} \), \( I_{3,2} \), and \( I_5 \) terms, modulo products.
		\begin{align*}
		I_{2, 1, 2}(a b c d e f) \modsh {} & I_{3, 1, 1}( [a b c d f e] 
		+ [a b c f d e] 
		+ [a b c f e d] 
		+ [a b d c e f] 
		+ [a b d e c f] 
		+ [a b e d c f] ) + {} \\ 
		& {} + I_{3, 2}( [a b c d f] 
		+ 2 [a b c e f] 
		- [a b c f d] 
		- [a b c f e] 
		+ [a b d c f] + \\
		& \hspace{3.5em} {}
		+ [a b d e f] 
		+ 2 [a b e c f] 
		+ [a b e d f] 
		- [a b e f c] ) + {} \\ 
		& {} + I_{5}( 12 [a b c f] 
		+ 6 [a b d f] 
		+ 12 [a b e f] ) \, .
		\end{align*}
	\end{Id}
	
	\begin{Id}
		The following identity expresses \( I_{1,2,2} \) in terms of \( I_{3,1,1} \), \( I_{3,2} \), and \( I_5 \) terms, modulo products.
		\begin{align*}
		I_{1, 2, 2}(a b c d e f) \modsh {} & I_{3, 1, 1}( - [a b c f e d] 
		- [a b d c e f] 
		- [a b d c f e] 
		- [a b d e c f] ) + {} \\ 
		& {} + I_{3, 2}( -2 [a b c e f] 
		+ [a b c f e] 
		- 2 [a b e c f] 
		- [a b e f d] 
		- [a b f e d] ) + {} \\ 
		& {} + I_{5}( -12 [a b c f] 
		- 6 [a b d e] 
		- 6 [a b d f] 
		- 6 [a b e f] )
		\end{align*}
	\end{Id}  
	
	Notice that all of the identities above make use of coupled cross-ratio arguments only.  Therefore, applying the above identities to \( \phi_5(a; b, c, d, e, f) \) from \autoref{id:i11111_dan_reduction} produces the following result.
	
	\begin{Thm}[Reduction of \( I_{1,1,1,1,1} \) to \( I_{3,1,1} \), \( I_{3,2} \) and \( I_5 \)]\label{thm:dan_wt5_as_311_32_5_nonexplicit}
		Modulo products, we can write
		\[
			[a \mid b, c, d, e, f \mid g] \modsh \phi_5'(a; b, c, d, e, f) - \phi_5'(g; b, c, d, e, f) \, ,
		\]
		where \( \phi_5' \) is an \emph{explicit} combination of \( I_{3,1,1} \), \( I_{3,2} \) and \( I_5 \) terms involving only coupled cross-ratio arguments.
	\end{Thm}
	
	If the above identities are applied to \( \phi_5 \) from \autoref{id:i11111_dan_reduction}, some \( I_5 \) in the result can be combined after rewriting the cross-ratio arguments using \( \CR(a,b,c,d) = \CR(b,a,d,c) = \CR(c,d,b,a) = \CR(d,c,b,a) \).  The resulting \( \phi_5'(a; b, c, d, e, f) \) has the following distribution of terms.
	\begin{center}
		\begin{tabular}{c|c}
			Integral & Number in \( \phi_5' \) \\ \hline
			$ I_{3, 1, 1} $ & 69 \\
			$ I_{3, 2}$ & 125 \\ 
			$ I_{5} $ & 48 \\ \hline
			Total number & 244
		\end{tabular}
	\end{center}
	The explicit expression for this \( \phi_5'(a; b, c, d, e, f) \) is given in \autoref{thm:dan_wt5_as_311_32_5}.
	
	\subsection{\texorpdfstring{Reduction of \( I_{1,1,1,1,1} \) to \( I_{3,2} \), modulo \( \delta \)}{Reduction of I\_11111 to I\_32, modulo delta}}
	
	Ideally, the final step of this reduction would to be write \( I_{3,1,1} \) in terms of \( I_{3,2} \) and \( I_5 \), modulo products.  Then we can completely reduce \( I_{1,1,1,1,1} \) to \( I_{3,2} \) and \( I_5 \), and explicitly confirm that the index 1 can always be eliminated.
	
	It is not possible to express \( I_{3,1,1} \) in terms of \( I_{3,2} \) using only coupled cross-ratio arguments.  However, a more `brute-force' approach does succeed, modulo \( \delta \).
	
	\subsubsection*{`Brute-force' reduction to \( I_{3,2} \) modulo \( \delta \)}
	
	The first step towards brute force identities involves computing \( I_{4,1}(x,y) \), modulo \( \delta \).  One finds
	\[
		I_{4,1}(x,y) \xrightarrow{\delta} - I_{2}(x) \wedge I_3(y) + I_{3}(x) \wedge I_2(y) \, .
	\]
	Modulo \( \delta \), both \( x \) and \( y \) appear separately on the right hand side; there are no terms where combinations of \( x \) and \( y \) appear in the same argument.  Moreover, we also see that
	\[
		I_{4,1}(x,y) + I_{4,1}(x,\tfrac{1}{y}) \xrightarrow{\delta}  - 2 I_{2}(x) \wedge I_3(y) \, ,
	\]
	using the inversion relation \( I_{3}(\tfrac{1}{y}) = I_3(y) \).  Notice that compared to \( I_{4,1}(x,y) \), the arguments in \( I_{4,1}(x, \tfrac{1}{y}) \) no longer constitute coupled cross-ratios.  We have \( I_{4,1}(x, y) = I_{4,1}(\infty 0 1 x, 0 \infty 1 y) \), and the first three variables no longer agree!
	
	If we can recognise some expression modulo \( \delta \), as a sum of terms of the form \( I_{2}(x) \wedge I_3(y) \), we can immediately write down \( I_{4,1} \) terms which agree with this expression modulo \( \delta \).   By using \autoref{id:i41asi32} we can also write down \( I_{3,2} \) terms which agree with this expression modulo \( \delta \).  We can then try to find \( \Li_5 \) terms to get an identity which holds modulo products.
		
	\subsubsection*{\( I_{3,2} \) in terms of \( I_{4,1} \)}  It is not possible to express \( I_{3,2} \) in terms of \( I_{4,1} \) using coupled cross-ratio arguments.  We can, however, find a `brute-force' identity which expresses \( I_{3,2} \) in terms of \( I_{4,1} \).
	
	\begin{Id}[\( I_{3,2} \) in terms of \( I_{4,1} \)]
		\label{id:i32asi41del}
		The following identity expresses a \( I_{3,2} \) in terms of \( I_{4,1} \) terms, modulo \( \delta \).
		\begin{align*}
			I_{3,2}(x,y) & \moddel -\tfrac{1}{2} \big( 
			3 I_{4,1}(x,y) 
			+ I_{4,1}(x,\tfrac{1}{y})
			+ I_{4,1}(x, \tfrac{x}{y}) + {} \\
			& \hspace{3.1em} {}
			+ I_{4,1}(x, \tfrac{y}{x})
			- I_{4,1}(y, \tfrac{x}{y})
			- I_{4,1}(y, \tfrac{y}{x}) \big) \, .
		\end{align*}
	\end{Id}
	
	Now one tries to find \( \Li_5 \) terms which lifts this to an identity modulo products.  This turns out to be possible, but the \( \Li_5 \) terms involved are \emph{significantly} more complicated that one might initially expect.
	
	\begin{Thm}[\( I_{3,2} \) in terms of \( I_{4,1} \) and \( \Li_5 \)]\label{thm:i32_as_i41_nonexplicit}
		It is possible to express \( I_{3,2} \) in terms of explicit \( I_{4,1} \), and explicit \( \Li_5 \) terms, modulo products.
	\end{Thm}
	
	The explicit expression for \autoref{thm:i32_as_i41_nonexplicit} is given in \autoref{id:app:i32_as_i41_and_li5}.  This serves as proof-of-concept that the `brute-force' approach can be a viable way of finding identities relating multiple polylogarithms.  The worksheet  \mathematicanb{wt5\_check\_bruteforce\_i32\_as\_i41.nb} checks the explicit version of this identity using the symbol
	
	\subsubsection*{\( I_{3,1,1} \) in terms of \( I_{3,2} \)}  We can try to proceed in the same way to find a reduction of \( I_{3,1,1} \) to \( I_{3,2} \) and \( I_5 \).
	
	\begin{Id}[\( I_{3,1,1} \) in terms of \( I_{3,2} \)] \label{id:i311_as_i32_mod_delta}
		The following identity expresses \( I_{3,1,1} \) in terms of \( I_{3,2} \) terms, modulo \( \delta \).
		\begin{align*}
		& \hspace{-2em} 3 I_{3, 1, 1}(a b c d e f) {} \moddel {} \\
		& I_{3, 2}( [a b c d e] 
		+ [a b c d f] 
		+ [a b c e d] 
		+ [a b c f d] 
		- [a c b d f] 
		- [a c b f d] 
		- [a d b e f] + {} \\
		& \hspace{2em} {}
		- [a d b f e] 
		+ [b a f c e] 
		+ [b a f e c] 
		- [b f a c e] 
		- [b f a e c] ) + {} \\[1ex]
		& {} + I_{3, 2}( [a b c e, a c b d] 
		+ [a b c e, a d b c] 
		+ [a b c f, a d c b] 
		- [a b d f, a e b f] 
		- [a b e f, a d e b] + {} \\
		& \hspace{3em} {}
		- [a b e f, a e d b] 
		+ [a c b d, a b c e] 
		+ [a d b c, a b c e] 
		- [a d b c, a b f c] 
		+ [a d b e, a b f e]  + {} \\
		& \hspace{3em} {}
		- [a d b f, a e b d] 
		+ [a e b d, a b f e] 
		- [a e b d, a d b f] 
		- [a e b f, a b d f] ) + {} \\[1ex]
		& {} + I_{3, 2}( [a b c d, a b f e] 
		- [a b e f, a b d c] 
		- [a b e f, a d b c] 
		+ [a c b d, a e c f] 
		- [a c b d, b c e f]  + {} \\
		& \hspace{3em} {}
		- [a c b d, b e c f] 
		- [a c d f, a d b e] 
		- [a c d f, a e b d] 
		+ [a c d f, a e b f] 
		- [a c e f, a d b c]  + {} \\
		& \hspace{3em} {}
		+ [a c e f, a d b e] 
		+ [a c e f, a e b d] 
		- [a d b c, a b e f] 
		- [a d b c, a c e f] 
		- [a d b e, a c d f]  + {} \\
		& \hspace{3em} {}
		+ [a d b e, a c e f] 
		- [a e b d, a c d f] 
		+ [a e b d, a c e f] 
		+ [a e b f, a c d f] 
		+ [a e c f, a c b d] + {} \\
		& \hspace{3em} {}
		+ [a f b e, b c d f] 
		+ [b c d f, a f b e] 
		- [b c e f, a c b d] 
		- [b e c f, a c b d] )
		\end{align*}
	\end{Id}
	
	\begin{proof}
		The worksheet \mathematicanb{wt5\_check\_ids.nb} checks this identity by computing the symbol.  The identity is given in \mathematicanb{wt5\_rules\_to\_i311\_i32\_i5.m}.
	\end{proof}
	
	\begin{Rem}
		In this identity, the \( I_{3,2} \) terms are grouped (roughly) according to their complexity.  Initially we have 12 terms of the form \( I_{3,2}(abcde) \), which constitute coupled cross-ratios.  One should think of these as the simplest kind of term.  Then we have 14 terms of the form \( I_{3,2}(abce, acbd) \); these do not exactly fit the form of a coupled cross-ratio, but they do involve only 5 of the 6 variables \( abcdef \).  This makes them of intermediate complexity.  Finally, we have 24 terms of the form \( I_{3,2}(abcd, abef) \), which contain all 6 of the variables in each term.  These are the most complex type of term.
	\end{Rem}

	This expression for \( I_{3,1,1} \) in terms of \( I_{3,2} \) holds modulo \( \delta \), only.  One would hope to be able to find \( \Li_5 \) terms to make this identity hold exactly modulo products.  Unfortunately, I have so far been unsuccessful in this step.  But given the existence of the \( \Li_5 \) terms in \autoref{id:app:i32_as_i41_and_li5}, which make the earlier `brute-force' identity (\autoref{id:i32asi41del}) expressing \( I_{3,2} \) in terms of \( I_{4,1} \) hold modulo products, one is cautiously optimistic that similar \( \Li_5 \) terms to make  \autoref{id:i311_as_i32_mod_delta} hold modulo products, \emph{do} in fact exist. \medskip
	
	If we apply \autoref{id:i311_as_i32_mod_delta} to the \( \phi_5' \) found in \autoref{thm:dan_wt5_as_311_32_5_nonexplicit} (and given explicitly in \autoref{thm:dan_wt5_as_311_32_5}), we obtain the following.
	
	\begin{Thm}[Reduction of \( I_{1,1,1,1,1} \) to \( I_{3,2} \)]\label{thm:dan_wt5_as_i32}
		Modulo \( \delta \), we can write
		\[
		[a \mid b, c, d, e, f \mid g] \modsh \phi_5''(a; b, c, d, e, f) - \phi_5''(g; b, c, d, e, f) \, ,
		\]
		where \( \phi_5'' \) is an \emph{explicit} combination of the following type of \( I_{3,2} \) terms:
		\begin{itemize}
			\item `Coupled cross-ratio terms' \( I_{3,2}(abcde) \),
			\item 5-variable cross-ratio terms \( I_{3,2}(abcd, abde) \), and
			\item 6-variable cross-ratio terms \( I_{3,2}(abcd, abef) \).
		\end{itemize}
	\end{Thm}
	
	In each case in \autoref{thm:dan_wt5_as_i32} above, (viewing \( [abcde] = [abcd,abce] \)), the number of cross-ratios which have a variable set to infinity is either 0, 1, or 2.  If we use the expression for \( \phi_5''(a; b, c, d, e, f) \) obtained by applying \autoref{id:i311_as_i32_mod_delta} to \autoref{thm:dan_wt5_as_311_32_5}, then we obtain the following distribution of terms.
		\begin{center}
			\begin{tabular}{c|c|c}
				Integral & Number of \( \infty \)'s & Number in \( \phi_5'' \) \\ \hline
				Coupled cross-ratio \( I_{3,2} \) & 0 & 68 \\
				Coupled cross-ratio \( I_{3,2} \) & 1 & 88 \\
				Coupled cross-ratio \( I_{3,2} \) & 2 & 276 \\
				5-variable cross-ratio \( I_{3,2} \) & 0 & 78 \\
				5-variable cross-ratio \( I_{3,2} \) & 1 & 155 \\
				5-variable cross-ratio \( I_{3,2} \) & 2 & 578 \\
				6-variable cross-ratio \( I_{3,2} \) & 0 & 48 \\
				6-variable cross-ratio \( I_{3,2} \) & 1 & 686\\
				6-variable cross-ratio \( I_{3,2} \) & 2 & 480 \\ \hline
				\multicolumn{2}{c|}{Total number} & 2457
			\end{tabular}
		\end{center}	
	Unfortunately, the expression for \( \phi_5'' \) is far too long to give explicitly, even in an appendix.
	
	{
		\emergencystretch10em
		\printbibliography[nottype=software]
		
		\defbibheading{bibliography}[\bibname]{%
			\subsection*{#1}
		}
		\printbibliography[type=software,title={Software}]
	}
	
	\hfill
	
	\appendix
	
	\section{Long identities relating weight 5 MPL's}\label{app:wt5}
	
	\subsection{\texorpdfstring{Reduction of \( I_{1,1,1,1,1} \) to \( I_{3,1,1} \), \( I_{3,2} \) and \( I_5 \)}{Reduction of I\_11111 to I\_311, I\_32 and I\_5}}
	
	The following identity gives explicitly the reduction of \( I_{1,1,1,1,1} \) to \( I_{3,1,1} \), \( I_{3,2} \) and \( I_{5} \), modulo products, which was alluded to in \autoref{thm:dan_wt5_as_311_32_5_nonexplicit}.
	
	\begin{Thm}\label{thm:dan_wt5_as_311_32_5}
			As shorthand, write \( abcd \coloneqq \CR(a,b,c,d) \) for the cross-ratio.  Moreover, recall the coupled cross-ratio shorthand notation from \autoref{not:ccr}, which has \( I_{s_1,\ldots,s_k}(abcd_1 \ldots d_k) \coloneqq \allowbreak I_{s_1,\ldots,s_k}(abcd_1, \ldots, abcd_k) \).  Then modulo products, we can give the following identity for \( I_{1,1,1,1,1} \).
		\[
		[a \mid b, c, d, e, f \mid g] \modsh \phi_5'(a; b, c, d, e, f) - \phi_5'(g; b, c, d, e, f) \, ,
		\]
		where
		\begin{align*}
		& \phi_5'(a; b, c, d, e, f) \coloneqq \\
		& I_{3, 1, 1}( -[b d a c e f] 
		+ [b d a e \infty c] 
		- [b d a f \infty c] 
		+ [b d a \infty e f] 
		+ [b d c a e \infty] 
		+ [b d c a \infty e] 
		+ [b d c \infty a e] + {} \\
		& \hspace{2.3em} {}
		+ [b d e \infty a f] 
		+ [b d \infty a e f] 
		+ [b d \infty e a f] 
		- [b f a c \infty e] 
		+ [b f a d \infty e] 
		- [b f a \infty c e] 
		+ [b f a \infty d c] + {} \\
		& \hspace{2.3em} {} 
		+ [b f a \infty d e] 
		- [b f a \infty e c] 
		+ [b f c a d \infty] 
		- [b f c a e d] 
		- [b f c a e \infty] 
		+ [b f c a \infty d] 
		+ [b f c d a \infty]  + {} \\
		& \hspace{2.3em} {}
		- [b f c e a \infty] 
		+ [b f d a \infty e] 
		- [b f e c a \infty] 
		+ [b f e d a \infty] 
		- [d \infty a c b f] 
		+ [d \infty a c e b] 
		+ [d \infty a c e f]  + {} \\
		& \hspace{2.3em} {}
		- [d \infty a c f b] 
		+ [d \infty a e b c] 
		+ [d \infty a e c b] 
		- [d \infty a f b c] 
		- [d \infty a f c b] 
		+ [d \infty b c a e] 
		+ [d \infty c a b e]  + {} \\
		& \hspace{2.3em} {}
		+ [d \infty c b a e] 
		+ [d \infty e b a f] 
		- [f \infty a b c e] 
		+ [f \infty a b d c] 
		+ [f \infty a b d e] 
		- [f \infty a b e c] 
		- [f \infty a c b e]  + {} \\
		& \hspace{2.3em} {}
		- [f \infty a c e b] 
		+ [f \infty a d b e] 
		+ [f \infty a d e b] 
		+ [f \infty c a e d] 
		+ [f \infty d a b e] 
		+ [f \infty d a e b] 
		+ [f \infty d e a b]  + {} \\
		& \hspace{2.3em} {}
		- [\infty b a c d f] 
		+ [\infty b a c e d] 
		- [\infty b a c f d] 
		- [\infty b a c f e] 
		+ [\infty b a d e f] 
		+ [\infty b a d f e] 
		+ [\infty b a e c d]  + {} \\
		& \hspace{2.3em} {}
		+ [\infty b a e d c] 
		- [\infty b a f c d] 
		- [\infty b a f c e] 
		+ [\infty b a f d e] 
		- [\infty b a f e c] 
		+ [\infty b c a d e] 
		+ [\infty b c a e d]  + {} \\
		& \hspace{2.3em} {}
		+ [\infty b c d a e] 
		+ [\infty b d a e f] 
		+ [\infty b d a f e] 
		+ [\infty b d c a e] 
		+ [\infty b d e a f] 
		+ [\infty b e d a f] ) + {} \\[1ex]
		& {} + I_{3, 2}( [b d a c e] 
		- [b d a e \infty] 
		+ [b d a \infty f] 
		+ [b d e \infty c] 
		+ [b d e \infty f] 
		+ 2 [b d \infty a e] 
		+ [b d \infty a f]  + {} \\
		& \hspace{3em} {}
		+ [b d \infty e c] 
		+ [b d \infty e f] 
		+ [b f a d \infty] 
		+ [b f a \infty c] 
		- 2 [b f a \infty d] 
		+ [b f a \infty e] 
		- [b f c a \infty]  + {} \\
		& \hspace{3em} {}
		- [b f c e \infty] 
		+ 3 [b f d a \infty] 
		+ [b f d e \infty] 
		+ [b f d \infty c] 
		- [b f e c \infty] 
		- [b f e \infty d] 
		+ [b f \infty d c]  + {} \\
		& \hspace{3em} {}
		- [b f \infty e d] 
		+ [b \infty c a d] 
		+ [b \infty c a e] 
		+ [b \infty d a e] 
		+ [b \infty d a f] 
		+ [d b e a f] 
		+ 2 [d b \infty a e]  + {} \\
		& \hspace{3em} {}
		+ [d \infty a b e] 
		- [d \infty a e b] 
		- [d \infty a e c] 
		+ [d \infty a e f] 
		+ 3 [d \infty b a e] 
		+ [d \infty b a f] 
		+ [d \infty b c e]  + {} \\
		& \hspace{3em} {}
		+ [d \infty b e f] 
		+ [d \infty c a e] 
		+ [d \infty c b e] 
		- [d \infty e f a] 
		- [d \infty e f b] 
		- [d \infty f e b] 
		- [f b c a e]  + {} \\
		& \hspace{3em} {}
		+ [f b d a e] 
		+ 2 [f b d a \infty] 
		- 2 [f b e a \infty] 
		- [f \infty a b d] 
		+ [f \infty a e b] 
		+ [f \infty b e d] 
		+ 2 [f \infty e a b]  + {} \\
		& \hspace{3em} {}
		+ [f \infty e b d] 
		- [\infty b a e c] 
		- [\infty b a e d] 
		+ 2 [\infty b a e f] 
		- [\infty b a f d] 
		+ [\infty b c a e] 
		+ [\infty b c d e]  + {} \\
		& \hspace{3em} {}
		+ [\infty b d a f] 
		+ [\infty b d e f] 
		+ 2 [\infty b e a f] 
		- [\infty b e f a] 
		+ 2 [\infty d b a e] 
		+ [\infty d b a f] 
		+ [\infty d c a b]  + {} \\
		& \hspace{3em} {}
		- [\infty d e a f] 
		+ [\infty f c a e] 
		- [\infty f d a e] 
		+ [\infty f e a b] ) + {} \\[1ex]
		& {} + \tfrac{1}{3} I_{3, 2}( -[b d a c f] 
		+ 4 [b d a e f] 
		- [b d a f c] 
		- 2 [b d a f e] 
		- [b d c a \infty] 
		- [b d c \infty a] 
		+ 7 [b d e a f]  + {} \\
		& \hspace{3.5em} {}
		- 2 [b d e f a] 
		- [b f a d e] 
		- [b f a e \infty] 
		- [b f c a d] 
		- [b f c d a] 
		+ 2 [b f d a e] 
		- [b f d c e] + {} \\
		& \hspace{3.5em} {}
		+ 2 [b f d e a] 
		- 7 [b f e a \infty] 
		+ 5 [b f e d \infty] 
		- [b f e \infty a] 
		- [b \infty d e a] 
		- [b \infty d e c] 
		+ [b \infty e d a]  + {} \\
		& \hspace{3.5em} {}
		- [b \infty e f a] 
		- [b \infty e f d] 
		+ [b \infty f e a] 
		+ [d \infty a c f] 
		+ [d \infty a f c] 
		- [d \infty a f e] 
		+ 4 [d \infty e b f]  + {} \\
		& \hspace{3.5em} {}
		+ [f b d e a] 
		+ [f b d e c] 
		- [f b e d a] 
		- 2 [f b e \infty a] 
		- 2 [f b e \infty d] 
		+ 2 [f b \infty e a] 
		+ [f \infty a b c]  + {} \\
		& \hspace{3.5em} {}
		+ [f \infty a c b] 
		+ [f \infty a d e] 
		+ [f \infty c a d] 
		+ [f \infty c d a] 
		- 2 [f \infty d a e] 
		+ [f \infty d c e] 
		- 2 [f \infty d e a]  + {} \\
		& \hspace{3.5em} {}
		+ 2 [\infty b a c f] 
		+ 4 [\infty b a d e] 
		+ 2 [\infty b a f c] 
		- [\infty b a f e] 
		+ [\infty b c a d] 
		+ [\infty b c d a] 
		+ 7 [\infty b d a e]  + {} \\
		& \hspace{3.5em} {}
		+ 4 [\infty b d c e] 
		- 2 [\infty b d e a] 
		+ 4 [\infty b e d f] 
		- [\infty d e f a] 
		- [\infty d e f b] 
		+ [\infty d f e a] 
		- [\infty f d e a]  + {} \\
		& \hspace{3.5em} {}
		- [\infty f d e c] 
		+ [\infty f e d a] ) + {} \\[1ex]
		& {} + I_{5}( 3 [a b d \infty] 
		+ 13 [a b f \infty] 
		- [a c b d] 
		- 3 [a d b f] 
		+ [a d f \infty] 
		+ 4 [a d \infty b] 
		+ 15 [a e b d]  + {} \\
		& \hspace{2.5em} {}
		+ 8 [a e b f] 
		+ 4 [a e b \infty] 
		+ 21 [a e d \infty] 
		- 8 [a e f b] 
		- 7 [a e f \infty] 
		+ 14 [a e \infty b] 
		+ 15 [a f b d]  + {} \\
		& \hspace{2.5em} {}
		- 3 [a f d \infty] 
		+ 10 [a f \infty b] 
		- 2 [a \infty b d] 
		- 8 [a \infty b f] 
		+ 6 [b d c e] 
		+ 5 [b d c \infty] 
		+ 13 [b d e f]  + {} \\
		& \hspace{2.5em} {}
		+ 6 [b d \infty e] 
		+ 12 [b d \infty f] 
		+ 6 [b e \infty f] 
		- 2 [b f c e] 
		- 8 [b f d e] 
		+ 12 [b f d \infty] 
		+ 8 [b \infty e c]  + {} \\
		& \hspace{2.5em} {}
		+ 13 [b \infty e d] 
		+ 8 [b \infty f d] 
		+ 6 [c e d \infty] 
		+ 2 [c e f \infty] ) + {} \\[1ex]
		& {} + \tfrac{1}{3} I_{5}( -4 [a d b \infty] 
		+ 4 [a d f b] 
		- 4 [a d \infty f] 
		- 4 [a e \infty d] 
		+ 16 [a e \infty f] 
		+ 16 [a f b \infty] 
		+ 16 [a f \infty d]  + {} \\
		& \hspace{3em} {}
		+ 32 [a \infty f b] 
		+ 23 [b e d \infty] 
		+ 19 [b f c d] 
		- 13 [b f e \infty] 
		- [b \infty d c] 
		+ 40 [b \infty f e] 
		- [c d f \infty]  + {} \\
		& \hspace{3em} {}
		+ 16 [d e f \infty] 
		+ 22 [d \infty e f] ) \, .
		\end{align*}
	\end{Thm}

	\subsection{\texorpdfstring{\( I_{3,2} \) in terms of \( I_{4,1} \) and \( \Li_5 \)}{I\_32 in terms of I\_41 and Li\_5}}
	
	Introduce the notation
	\[
	\Li_5(\{ \pm ; a, b, c, d, e \}) \coloneqq \Li_5\big( \pm x^a (1-x)^b y^c (1-y)^d (x-y)^e \big) \, .
	\]
	Then following identity lifts the `brute-force' identity for \( I_{3,2} \) in terms of \( I_{4,1} \) from \autoref{id:i32asi41del}, to an identity with explicit \( \Li_5 \) terms, modulo products.
	
	\begin{Thm}[\( I_{3,2} \) in terms of \( I_{4,1} \) and \( \Li_5 \)]\label{id:app:i32_as_i41_and_li5}
		The following identity expresses \( I_{3,2}(x,y) \) in terms of 6 \( I_{4,1} \) term, and 141 \( \Li_5 \) terms, modulo products.
		\begin{align*}
		& \tfrac{22}{9} \Big(I_{3,2}(x, y) + {} \\
		& \hspace{1em} +\tfrac{1}{2} \left( I_{4,1}(x, \tfrac{1}{y})+I_{4,1}(x, \tfrac{x}{y})+I_{4,1}(x, \tfrac{y}{x})+3 I_{4,1}(x, y)-I_{4,1}(y, \tfrac{x}{y})-I_{4,1}(y, \tfrac{y}{x}) \right) \Big) \modsh \\[1ex]
		& \tfrac{1}{9} \Li_5\big( \{+; 0, 1, -1, -3, 1\}
		+ \{+; 0, 1, -1, 3, -2\}
		+ \{+; 0, 2, -2, 0, -1\}
		+ \{+; 0, 2, -2, 0, 1\} + {} \\ 
		& \hspace{1em} {}
		- \{+; 0, -2, -3, 1, 3\}
		+ \{+; 0, -3, -1, -1, 2\}
		+ \{-; 1, 0, 0, -1, -2\}
		+ \{-; 1, 0, 0, -1, 2\} + {} \\ 
		& \hspace{1em} {}
		+ \{+; 1, 0, -3, 2, 1\}
		- \{+; -1, -1, 0, -3, 2\}
		- \{-; 1, 1, -1, 2, 0\}
		- \{-; -1, 1, -1, -2, 3\} + {} \\ 
		& \hspace{1em} {}
		- \{+; -1, 2, -2, 1, 0\}
		- \{+; 1, 2, -2, -1, 0\}
		- \{-; 1, -2, -3, 0, 1\}
		+ \{-; -1, 3, 0, -2, 1\} + {} \\ 
		& \hspace{1em} {}
		- \{+; -1, -3, -1, 0, 3\}
		- \{-; 1, -3, -2, 3, 0\}
		- \{-; 2, 0, -1, 3, -3\}
		+ \{-; 2, 0, -3, 1, -1\} + {} \\ 
		& \hspace{1em} {}
		- \{-; -2, 1, -1, -1, 0\}
		- \{-; 2, 1, -1, 1, -3\}
		- \{-; 2, -1, -3, 3, -1\}
		+ \{+; -2, 3, 0, -1, -1\} + {} \\ 
		& \hspace{1em} {}
		+ \{+; 3, 0, -1, 2, -1\}
		- \{-; -3, -1, 0, -1, 3\}
		+ \{+; 3, 1, -1, 0, -1\}
		+ \{+; -3, 1, -1, 0, 2\} + {} \\ 
		& \hspace{1em} {}
		- \{-; 3, -1, -3, 2, 0\}
		+ \{+; 3, -3, -2, 1, 1\} \big) + {} \\ 
		& {} + \tfrac{4}{9} \Li_5\big( -\{-; 0, 0, -1, 2, -1\}
		- \{+; 0, 0, 2, -1, -1\}
		+ \{+; 0, -1, 0, -1, 0\}
		- \{-; 0, 1, -1, 0, -1\} + {} \\ 
		& \hspace{1em} {}
		- \{-; 0, 1, -1, 0, 1\}
		- \{-; 0, -2, -1, 0, 1\}
		- \{+; 1, 0, 0, -1, -1\}
		- \{+; 1, 0, 0, -1, 1\} + {} \\ 
		& \hspace{1em} {}
		+ \{-; 1, 1, 0, 0, -2\}
		+ \{+; -1, 1, -1, 1, 0\}
		+ \{+; 1, 1, -1, -1, 0\}
		+ \{-; 1, -1, -2, 2, 0\} + {} \\ 
		& \hspace{1em} {}
		- \{+; -2, 0, 0, -1, 1\}
		+ \{+; 2, -1, -2, 1, 0\} \big) + {} \\ 
		& {} + \tfrac{5}{9} \Li_5\big( \{-; 0, 0, 1, 1, -2\}
		- \{+; 0, 1, 0, -2, 0\}
		+ \{-; 0, -1, -2, 0, 1\}
		+ \{+; 0, -1, -2, 0, 2\} + {} \\ 
		& \hspace{1em} {}
		- \{+; 0, 2, 0, -1, 0\}
		+ \{+; -1, 0, 0, -2, 1\}
		- \{+; -1, 2, 0, 0, -1\}
		+ \{+; -2, 0, 0, -1, 2\} + {} \\ 
		& \hspace{1em} {}
		- \{-; 2, -1, 0, 0, -1\}
		- \{-; 2, -2, -1, 1, 0\} \big) + {} \\ 
		& {} + \tfrac{2}{3} \Li_5\big( \{+; 0, 0, 0, 0, -1\}
		+ \{+; 0, 0, 0, 1, -1\}
		+ \{+; 0, -1, -1, 1, 0\}
		- \{-; 0, -1, -1, 1, 1\} + {} \\ 
		& \hspace{1em} {}
		+ \{-; 1, 0, 0, -1, 0\}
		- \{+; 1, 0, 0, -1, 0\}
		+ \{+; -1, 0, -1, 0, 1\}
		- \{-; 1, -1, -1, 0, 0\} + {} \\ 
		& \hspace{1em} {}
		+ \{+; 1, -1, -1, 1, 0\} \big) + {} \\ 
		& {} + \Li_5\big( -\{+; 0, 0, 0, -2, 1\}
		- \{-; 0, 0, 1, -1, -1\}
		- \{+; 0, 1, 0, -1, 1\}
		- \{+; 0, 1, -1, -1, 0\} + {} \\ 
		& \hspace{1em} {}
		+ \{-; 0, 1, -1, 1, 0\}
		+ \{-; 0, -1, -1, -1, 1\}
		+ \{-; 0, 1, -1, -1, 1\}
		- \{+; 0, 1, -1, 1, -1\} + {} \\ 
		& \hspace{1em} {}
		+ \{-; 0, -2, -1, 1, 1\}
		+ \{+; -1, 0, 0, -1, 0\}
		+ \{+; -1, 0, 0, -1, 2\}
		- \{+; 1, 0, -1, 0, -1\} + {} \\ 
		& \hspace{1em} {}
		+ \{-; 1, 0, -1, 2, -1\}
		- \{-; 1, 0, -2, 1, 0\}
		- \{-; -1, 1, 0, 0, -1\}
		- \{+; -1, -1, 0, -1, 1\} + {} \\ 
		& \hspace{1em} {}
		+ \{-; -1, -1, 0, -1, 2\}
		+ \{-; -1, 1, -1, 0, 0\}
		- \{+; 1, 1, -1, 0, 0\}
		- \{+; -1, 1, -1, 0, 1\} + {} \\ 
		& \hspace{1em} {}
		+ \{-; 1, 1, -1, 0, -1\}
		+ \{-; -1, -1, -1, 0, 2\}
		+ \{+; 1, -1, -1, 2, -1\}
		+ \{+; 1, -1, -2, 1, 0\} + {} \\ 
		& \hspace{1em} {}
		+ \{-; 1, -1, -2, 1, 1\}
		- \{-; -1, 2, 0, -1, 0\}
		- \{+; 1, -2, -1, 0, 1\}
		- \{+; 2, 0, 0, 0, -1\} + {} \\ 
		& \hspace{1em} {}
		+ \{+; 2, 0, -1, 1, -1\}
		- \{+; 2, -1, -1, 1, 0\} \big) + {} \\  
		& {} + \tfrac{5}{3} \Li_5\big( -\{+; 0, 0, -1, -1, 1\}
		- \{+; 0, 1, 0, 0, 0\}
		+ \{-; 0, 1, 0, 0, -1\}
		+ \{+; 0, -1, 0, -1, 1\} + {} \\ 
		& \hspace{1em} {}
		+ \{+; 0, -1, -1, 1, 1\}
		- \{-; 0, -1, -1, -1, 2\}
		- \{-; -1, 0, -1, 0, 1\}
		+ \{-; 1, 1, 0, 0, -1\} + {} \\ 
		& \hspace{1em} {}
		- \{+; -1, 1, 0, -1, 0\}
		- \{-; -1, 1, 0, -1, 1\}
		+ \{+; -1, 1, 0, -1, 1\}
		- \{+; -1, 1, 0, -2, 1\} + {} \\ 
		& \hspace{1em} {}
		+ \{+; 1, -1, -1, 0, 0\}
		- \{-; 1, -1, -1, 0, 1\}
		- \{-; 1, -2, -1, 1, 1\} \big) + {} \\ 
		& {} + \tfrac{8}{3} \Li_5\big( -\{-; 0, 0, 0, 0, -1\}
		+ \{-; 0, 0, -1, 0, 1\}
		+ \{+; 0, 1, 0, -1, 0\}
		- \{-; 0, -1, 0, -1, 1\} + {} \\ 
		& \hspace{1em} {}
		- \{-; 0, 1, -1, 0, 0\}
		+ \{+; 0, 1, -1, 0, 0\}
		+ \{-; 1, 0, -1, 1, 0\}
		- \{+; 1, 0, -1, 1, 0\} + {} \\ 
		& \hspace{1em} {}
		+ \{+; 1, 0, -1, 1, -1\} \big) + {} \\ 
		& {} + \tfrac{11}{3} \Li_5\big( \{+; 0, 0, 0, 1, 0\}
		- \{-; 0, 0, -1, 1, 0\}
		- \{+; 1, 0, 0, 0, -1\}
		- \{-; -1, 1, 0, -1, 0\} \big) + {} \\ 
		& {} + \tfrac{4}{3} \Li_5\big( -\{-; 0, -1, -1, 0, 1\}
		- \{+; -1, 0, 0, -1, 1\} \big) + \tfrac{20}{3} \Li_5\big( \{-; -1, 1, 0, 0, 0\} \big) + {} \\
		& {} + \tfrac{14}{3} \Li_5\big( -\{-; 0, 0, -1, -1, 1\}
		- \{+; 1, -1, -1, 0, 1\} \big) + \tfrac{1}{6} \Li_5\big( \{+; 1, 0, -1, 0, 0\} \big) + {} \\
		& {} + \tfrac{103}{54} \Li_5\big( -\{+; -1, 0, -1, 0, 0\} \big) + \tfrac{107}{54} \Li_5\big( \{+; 1, 0, -2, 0, 0\} \big) + \tfrac{157}{54} \Li_5\big( -\{+; 2, 0, -1, 0, 0\} \big) + {} \\ 
		& {} + 3 \Li_5\big( \{-; 1, 0, -1, 1, -1\}
		+ \{+; 1, 1, 0, 0, -1\} \big) + 4 \Li_5\big( \{-; 0, -1, -1, 1, 0\} \big) + {} \\ 
		& {} + \tfrac{5}{27} \Li_5\big( \{-; -1, 0, -1, -3, 3\}
		+ \{+; 1, -3, -2, 0, 3\}
		+ \{-; 2, -3, -1, 3, 0\} \big) + {} \\ 
		& {} + \tfrac{7}{18} \Li_5\big( \{+; 0, 0, -1, 0, 0\}
		- \{+; 1, -2, -1, 2, 0\} \big) + \tfrac{23}{18} \Li_5\big( -\{+; -1, 0, 0, 0, 0\} \big) + {} \\ 
		& {} + \tfrac{5}{18} \Li_5\big( -\{+; 0, -2, -1, 0, 2\}
		- \{+; -1, 0, 0, -2, 2\} \big) \, .
		\end{align*}
	\end{Thm}
	
	\begin{proof}
		This identity is given in \mathematicanb{wt5\_bruteforce\_i32\_as\_i41\_and\_li5.m}.  Using the worksheet  \mathematicanb{wt5\_check\_bruteforce\_i32\_as\_i41.nb}, this identity can be checked by computing the symbol.
	\end{proof}
	
	\begin{Rem}
		The candidate \( \Li_5 \) arguments which eventually produced this identity were generated using Radchenko's sage package MESA \cite{Mesa}, and the \texttt{set\_{}extra\_primes\_tree\_search} routine.
		
		This allows a good choice of \emph{extra} factors to appear in \( 1 - \alpha \), when computing the symbol of \( \Li_n(\alpha) \).  For example, even though the factor \(  -x + x^2 - x y + y^2 \) does not appear anywhere in the symbol of the left hand side of \autoref{id:app:i32_as_i41_and_li5}, it does appear in the symbol of the following \( \Li_5 \) terms  on the right hand side.
		\begin{align*}
		& \Li_5\Big( -\tfrac{1}{9} \{ +; 1, 2, -2, -1, 0 \}
		-\tfrac{1}{9} \{ -; 2, 1, -1, 1, -3 \}
		+ \{ -; 0, 1, -1, -1, 1 \} + {} \\
		& \hspace{2.5em} {}
		-\tfrac{1}{9} \{ -; -1, 1, -1, -2, 3 \} 
		+ \{ -; 1, 1, -1, 0, -1 \}
		+ \{ +; -1, 0, 0, -1, 2 \} \Big) \\
		& = \Li_5\Big( -\tfrac{1}{9}\left[\tfrac{(1-x)^2 x}{(1-y) y^2}\right]
		-\tfrac{1}{9}\left[-\tfrac{(1-x) x^2 (1-y)}{y (x-y)^3}\right]
		+\left[-\tfrac{(1-x) (x-y)}{(1-y) y}\right] + {} \\
		& \hspace{3.5em} {} 
		-\tfrac{1}{9}\left[-\tfrac{(1-x) (x-y)^3}{x (1-y)^2 y}\right]
		+\left[-\tfrac{(1-x) x}{y (x-y)}\right]
		+\left[\tfrac{(x-y)^2}{x (1-y)}\right] \Big) \, .
		\end{align*}
		Somehow, these terms conspire to combine in just the right way as to make this factor disappear in the end.  Because of this phenomenon, and the large number of potential arguments otherwise, finding \autoref{id:app:i32_as_i41_and_li5} would potentially be very difficult if not for the MESA software \cite{Mesa}.
	\end{Rem}
	
\end{document}